\documentclass[12pt,a4paper]{article}

\usepackage[english]{styleLucModif}

\usepackage{cases}

\author{L. Leh\'ericy\footnote{Laboratoire de Math\'ematiques d'Orsay, Univ. Paris-Sud, CNRS, Universit\'e Paris-Saclay, 91405 Orsay, France.} , A. Touron$^*$\footnote{EDF R\&D}}
\title{Consistency of the maximum likelihood estimator in hidden Markov models with trends}

\date{}

\newcommand{\xfarT}{{x^*_\Tcal}}
\newcommand{\xfarU}{{x_\Ucal}}

\newcommand{\Kast}{{K^*}}
\newcommand{\xast}{{x^*}}
\newcommand{\Xcaltheta}{{[K^\theta]}}

\newcommand{\ThetaOK}{\Theta^{\text{OK}}_n}

\begin{document}

\maketitle

\begin{abstract}
A hidden Markov model with trends is a hidden Markov model whose emission distributions are translated by a trend that depends on the current hidden state and on the current time.
Unlike standard hidden Markov models, such processes are not homogeneous and cannot be made homogeneous by a simple de-trending step.
We show that when the trends are polynomial, the maximum likelihood estimator is able to recover the trends together with the other parameters and is strongly consistent.
More precisely, the supremum norm of the difference between the true trends and the estimated ones tends to zero.
The convergence of the maximum likelihood estimator is illustrated on simulated data.
\end{abstract}


\setcounter{tocdepth}{2}
 \tableofcontents

\section{Introduction}

A hidden Markov model (shortened HMM in the following) is a joint process $(X_t, Y_t)_{t \geq 1}$ where $(X_t)_{t \geq 1}$ is an (unobserved) Markov chain and conditionally on $(X_t)_{t \geq 1}$, the observations $(Y_t)_{t \geq 1}$ are independent and the distribution of $Y_s$ (the \emph{emission distribution}) depends only on $X_s$. In most generalizations of HMM, the joint process $(X_t, Y_t)_{t \geq 1}$ is still a Markov chain. We say that the process is \emph{inhomogeneous} when this chain is inhomogeneous, that is when the distribution of $(X_t,Y_t)$ conditionally to $(X_{t-1}, Y_{t-1})$ depends on $t$. In particular, in inhomogeneous HMM, the transition matrix and emission distributions may vary over time.



\bigskip

As far as we know, few inhomogeneous generalizations of HMM have been studied theoretically.
Let us review them before stating our contribution.
They can be divided in four categories.

\cite{diehn2018inhomogeneousHMM} focus on the case where a rapidly fading phenomenon affects the distribution of the observations. Their model is a trivariate process $(X_t, Y_t, Z_t)_{t \geq 1}$ where only $(Z_t)_{t \geq 1}$ is observed, such that $(X_t, Y_t)_{t \geq 1}$ is an homogeneous HMM and $(X_t, Z_t)_{t \geq 1}$ is an inhomogeneous HMM. Their key assumption is that the distance between $Z_t$ and $Y_t$ tends to zero fast enough when $t$ tends to infinity. In this sense, the process $(Z_t)_{t \geq 1}$ is a perturbation of the process $(Y_t)_{t \geq 1}$ by a rapidly fading inhomogeneous noise.
The main theoretical result of their article is that the estimators based on maximizing either the true likelihood $\theta \mapsto p_{Z_1^n | \theta}(Z_1^n)$ or a quasi-likelihood $\theta \mapsto p_{Y_1^n | \theta}(Z_1^n)$ are consistent (here $p_{Y_1^n | \theta}$ is the density of the vector $(Y_1, \dots, Y_n)$ under the parameter $\theta$).
This can be seen as a proof that the maximum likelihood estimator for homogeneous HMM is robust to a temporary perturbation of the data.

A second generalization has been considered by \cite{touron2018SHMM_consistency} to handle periodic phenomenon. The author introduces periodic hidden Markov models, where the transition matrix and the emission densities vary periodically over time. He shows that such models are identifiable under general assumptions and that the maximum likelihood estimator is consistent. 
The consistency proof relies on a transformation of the process into an homogeneous HMM. Models with periodic parameters are of particular interest in the study of meteorological phenomena. For instance, \cite{ailliot2012} use Markov-Switching Autoregressive (MS-AR) models with periodic emission distributions and transition matrices in order to account for both the daily and yearly seasonalities of wind speed.

Inhomogeneity can also be introduced in the process by considering that the transition probabilities of the hidden Markov chain depend on an exogeneous variable. \cite{bellone2000hidden} and~\cite{hughes1999} chose this approach to model precipitations, taking a vector of atmospheric variables as the exogeneous variable. However, no theoretical results on the consistency of the MLE are given in these papers.

Finally, we may assume that the transition probabilities or the emission densities of the HMM change over time in a way not encompassed by the previous cases.
For example, this allows to model phenomena with slow or long-term evolution, such as economic or meteorological data recordings spanning over several decades, as in \cite[Part II]{touron2019PhD}, where the author is interested in generating realistic temperature recordings over the last half century. To account for the global warming, he uses HMM with trends similar to the ones studied here. We could not find any theoretical result on this type of generalization outside of our paper.

Note that some models that are presented as time-inhomogeneous in the literature are actually homogeneous according to our definition.
This is the case of the Markov-switching models considered since the work of~\cite{hamilton89} and more recently by \cite{pouzo2016misspecified} and \cite{ailliot2015consistency} for instance. These models are a generalization of HMM where the hidden state $X_t$ depends both of the previous hidden state $X_{t-1}$ and on previous observations, let's say $Y_{t-1}$ for an order one model, and where the observation $Y_t$ depends both on the corresponding hidden state $X_t$ and on previous observations. The authors state that this model is ``non-homogeneous'' because the transition kernel of the hidden Markov chain depends on previous observations.
Our motivation for not calling these models inhomogeneous is that the joint process $(X_t, Y_t)_{t \geq 1}$ is a homogeneous Markov chain. As such, their proofs are based on the same approach as for HMM and other homogeneous generalizations such as the autoregressive models with Markov regime of~\cite{douc2004asymptotic}.


\bigskip

In this paper, we introduce a new inhomogeneous generalization: hidden Markov models with trends.
A HMM with trends is a trivariate process $(X_t, Y_t, Z_t)_{t \geq 1}$ taking values in $\Xcal \times \Rbb^d \times \Rbb^d$ for some finite set $\Xcal$ and some integer $d$ (which we assume to be 1 for the sake of simplicity) where only the process $(Y_t)_{t \geq 1}$ is observed, such that $(X_t, Z_t)_{t \geq 1}$ is a homogeneous HMM and that there exists a vector of functions $(T_x)_{x \in \Xcal}$ from $\Nbb^*$ to $\Rbb^d$, the \emph{trends}, such that
\begin{equation*}
Y_t = T_{X_t}(t) + Z_t.
\end{equation*}
Here, the trends are polynomial functions of time. As a consequence, they may diverge.

Adding trends to a HMM allows to deal with non periodic and non vanishing inhomogeneities. However, this makes existing approaches such as~\cite{douc2004asymptotic} inapplicable since they rely heavily on the homogeneity of the process. Our main contribution is a new method of proof that shows that the maximum likelihood estimator recovers the parameter of the homogeneous HMM $(X_t, Z_t)_{t \geq 1}$ as well as the trends with respect to the supremum norm, see Theorem~\ref{th_MLEconsistency}.

\bigskip
Let us give an overview of the proof of this theorem.

First, we introduce the notion of ``blocks". This notion relies solely on the true parameters. A block is a set of hidden states whose trends are equal up to translation. In particular, since the trends are polynomials, if two trends are not in the same block, they will eventually diverge from one another.

For each time $t$, let $B_t$ be the block of $X_t$. The block variables $(B_t)_{t \geq 1}$ are unobserved, yet the log-likelihood of the process $(Y_t)_{t \geq 1}$ is asymptotically the same as the log-likelihood of the process $(Y_t, B_t)_{t \geq 1}$, see Theorem~\ref{th_approx_blocs}: since trends of different blocks diverge, it eventually becomes apparent which block a given observation comes from.

Given a constant $M > 0$, we say that a function belongs to the tube of a block at time $n$ if its distance to the trends of this block remains bounded by $M$ during the first $n$ time steps.
In the following, the bound $M$ on the distance must not depend on $n$.

The second step of the proof shows that the trends of the maximum likelihood estimator end up in close proximity to the true trends: Theorem~\ref{th_localization} shows that eventually, the trends of the maximum estimator are all in the tube of a block and each tube contains at least one trend of the maximum likelihood estimator.

Let us give an idea of the proof of Theorem~\ref{th_localization}. When a tube contains no estimated trend, each observation from its block is far from the estimated trends, which inflicts a heavy loss on the likelihood. On the other hand, if an estimated trend is in no tube, then eventually it is far from all observations. We assume that the transition matrix is lower bounded, so each state should be visited a non-zero proportion of the time. However, when a trend is far from all observations, the posterior proportion of observations generated by the corresponding state tends to zero: each time this superfluous state isn't seen worsens the likelihood a little more, to the point that eventually, removing the trend improves the likelihood of the estimator. Thus, asymptotically, each estimated trend is in a tube and no tube is empty.

The third step is to de-trend the observations by substracting a representative $\Tbb^*_{B_t}$ of the trends of the block:
\begin{equation*}
Z'_t = Y_t - \Tbb^*_{B_t}(t).
\end{equation*}

The process $(X_t, (Z'_t, B_t))_{t \geq 1}$ obtained this way is also a HMM with trends, yet it is still not homogeneous under the parameter $\theta$ since the residual trends $T^\theta_x - \Tbb^*_{b}$ of $(Z'_t)_{t \geq 1}$ (where $b$ is the block that contains $T^\theta_x$), while bounded, are not constant. Fortunately, as the number of observations grow, these residual trends become ``flatter", because they are polynomials with bounded degree that are bounded on $[0,n]$. Thus, during small time intervals, the de-trended process behaves like a homogeneous HMM under the parameter $\theta$, which allows to explicitly compute the limit of the log-likelihood, see Theorem~\ref{th_conv_int_loglik}.
The identifiability of the model and the consistency of the maximum likelihood estimator follow from the properties of this limit.

Simulations based on two parametric models illustrate the consistency of the maximum likelihood estimator. Moreover, an empirical study of the rates of convergence suggests that it converges with parametric rates.

Section~\ref{sec_model} introduces the model and assumptions used in the following. Section~\ref{sec_main} presents our main result: the consistency of the maximum likelihood estimator (Theorem~\ref{th_MLEconsistency}), as well as an outline of its proof. The technical proofs are postponed to the appendices. Section~\ref{sec_simulations} contains the numerical study. Finally, Appendices A to D contain the proofs of the different steps of the consistency proof.

\paragraph*{Notation} For each positive integer $K$, $[K]$ denotes the set $\{1, \dots, K\}$. For $a \leq b$ integers, write $Y_a^b$ instead of $(Y_a, \dots, Y_b)$.

\section{Model and assumptions}
\label{sec_model}

Let $\Kast$ be a positive integer. Let $\gamma^* = (\gamma^*_{x})_{x \in [\Kast]}$ be a vector of probability densities on $\Rbb$ with respect to the Lebesgue measure. Let $(X_t)_{t \geq 1}$ be a Markov chain on $[\Kast]$ with transition matrix $Q^*$ and initial distribution $\pi^*$. For all $x \in [\Kast]$, let $(Z_t^{x})_{t \geq 1}$ be a sequence of i.i.d. random variables in $\Rbb$ such that these sequences are mutually independent and independent on $(X_t)_{t \geq 1}$ and such that for all $x \in [\Kast]$, $Z^{x}_1$ has density $\gamma^*_{x}$ with respect to the Lebesgue measure. Let $Z^{\max}_t = \max_{x \in [\Kast]} |Z_t^{x}|$ and $Z_t = Z_t^{X_t}$. Finally, let $T^* = (T^*_{x})_{x \in [\Kast]}$ be a family of functions $\Rbb_+ \longrightarrow \Rbb$ and let $Y_t = Z_t + T^*_{X_t}(t)$ for all integers $t \geq 1$. The $(T^*_{x})_{x \in [\Kast]}$ are called \emph{trends}.

The process $(X_t, Z_t)_{t \geq 1}$ is a homogeneous hidden Markov model with parameter $(\Kast, \pi^*, Q^*, \gamma^*)$ and $(X_t, Y_t)_{t \geq 1}$ is a hidden Markov model with trends with parameter $(\Kast, \pi^*, Q^*, \gamma^*, T^*)$.

\begin{remark}
The random variables $Z^{\max}_t$ are i.i.d. and independent of $(X_t)_{t \geq 1}$. They allow to bound $Z_t$ uniformly for all possible values of $X_t$.
\end{remark}

Consider a sample $(Y_1, \dots, Y_n)$ generated by a hidden Markov model with trends $(X_t, Y_t)_{t \geq 1}$ with parameter $\theta^* := (\Kast, \pi^*, Q^*, \gamma^*, T^*)$, which we call the \emph{true parameter}. The goal is to recover this parameter. In the following, we write $\Pbb^*$ the distribution of the process $(X_t, Y_t)_{t \geq 1}$ and $\Ebb^*$ the corresponding expectation.

\bigskip

Let $\sigma_- \in (0,1)$.
\begin{description}
\item[\textbf{(Aerg)}] A stochastic matrix $Q$ satisfies \textbf{(Aerg)} when all its coefficients are lower bounded by $\sigma_-$:
\begin{equation*}
\forall x,x' , \quad Q(x,x') \geq \sigma_-.
\end{equation*}
\end{description}

For each $K \in \Nbb^*$, let $\Sigma_K^{\sigma_-}$ be the set of stochastic matrices of size $K$ which satisfy \textbf{(Aerg)} and $\Delta_K$ be the set of probability vectors of size $K$. Let $\Gamma$ be a set of density functions on $\Rbb$ and for each $d \in \Nbb$, write $\Rbb_d[X]$ the set of polynomials whose degree is at most $d$.

Let $K$ and $d$ be positive integers. The model considered in this paper is
\begin{equation*}
\Theta = \bigcup_{K'=1}^K \left(
	\{[K']\}\times \Delta_{K'} \times \Sigma_{K'}^{\sigma_-} \times \Gamma^{K'} \times (\Rbb_d[X])^{K'} \right)
\end{equation*}

Each $\theta = (K^\theta, \pi^\theta, Q^\theta, \gamma^\theta, T^\theta) \in \Theta$ is a parameter of a HMM with trends. Write $\Pbb^\theta$ the distribution under the parameter $\theta$. Assume that the true parameter belong to the model: $\theta^* \in \Theta$. 

Endow $\Gamma$ with the pointwise convergence topology (i.e. $g_n\underset{n \rightarrow \infty}{\longrightarrow}g$ if and only if for all $z\in\Rbb$, $g_n(z)\underset{n \rightarrow \infty}{\longrightarrow}g(z)$). Assume that $\Gamma$ is compact.

We will use the following assumptions.

\begin{description}
\item[\textbf{(Amax)}] \textit{Envelope function.} There exists a nonincreasing function $g : \Rbb_+ \longrightarrow \Rbb_+$ such that $g \underset{+ \infty}{\longrightarrow} 0$ and
\begin{equation*}
\forall \gamma \in \Gamma \quad \forall z \in \Rbb \quad \gamma(z) \leq g(|z|).
\end{equation*}

\item[\textbf{(Amin)}] \textit{Lower bound function.} There exists a nonincreasing function $m : \Rbb_+ \longrightarrow \Rbb_+$ such that
\begin{equation*}
\forall \gamma \in \Gamma \quad \forall z \in \Rbb \quad \gamma(z) \geq m(|z|) > 0.
\end{equation*}

\item[\textbf{(Aint)}] \textit{Integrability of the lower bound function.} The function $m$ defined in \textbf{(Amin)} satisfies
\begin{equation*}
\forall M > 0 \quad \Ebb^* | \log m(M + Z^{\max}_t) | < \infty.
\end{equation*}

\item[\textbf{(Acentering)}] \textit{Centering of the emission densities.} 0 is a median of the emission densities, that is:
\begin{equation*}
\forall \gamma \in \Gamma \quad
\int_{z \leq 0} \gamma(z) dz = \frac{1}{2}.
\end{equation*}

\item[\textbf{(Aid)}] \textit{Identifiability.} 
$Q^*$ is invertible and there exists $t \in \Nbb$ such that the densities $(\gamma^*_x(\cdot - T^*_x(t)))_{x \in [\Kast]}$ are pairwise distinct.

Equivalently, $Q^*$ is invertible and the couples $(\gamma^*_{x}(\cdot - \Delta(x)), \bbf^*(x))_{x \in [\Kast]}$ are pairwise distinct, where the functions $\bbf^*$ and $\Delta$ are defined in Definition~\ref{def_trend_blocks}.

\item[\textbf{(Areg)}] \textit{Regularity of the emission densities.} There exists a modulus of continuity $\omega$ (that is a nondecreasing function $\Rbb_+ \longrightarrow \Rbb_+ \cup \{ + \infty \}$ that is continuous at 0 and such that $\omega(0) = 0$) and a nondecreasing function $L : \Rbb_+ \longrightarrow \Rbb_+$ such that
\begin{equation*}
\forall (z, \eta) \in \Rbb^2
\quad \forall \gamma \in \Gamma
\quad \left| \log \frac{\gamma(z+\eta)}{\gamma(z)} \right| \leq L(|z|) \omega(|\eta|)
\end{equation*}
and such that
\begin{equation*}
\forall M > 0, \quad \Ebb^*[L(M + Z^{\max}_1)] < \infty.
\end{equation*}

\end{description}

As an example, consider $\Gamma$ as the set of densities of the distributions $\Ncal(0, s^2)$ where $s \in [s_-, s_+]$ for $0 < s_- \leq s_+ < \infty$. Then \textbf{(Amax)}, \textbf{(Amin)} and \textbf{(Acentering)} are immediate, \textbf{(Aint)} is straightforward since $|\log m|$ can be chosen as growing at most quadratically to infinity, the part of \textbf{(Aid)} on the densities holds when all states with equal trends have different variances, and the functions $L$ and $\omega$ in \textbf{(Areg)} can be computed explicitely.

\section{Consistency of the MLE}
\label{sec_main}

Let $n \in \Nbb^*$. The maximum likelihood estimator $\hat{\theta}_n = (K^{\hat{\theta}_n}, \pi^{\hat{\theta}_n}, Q^{\hat{\theta}_n}, \gamma^{\hat{\theta}_n}, T^{\hat{\theta}_n})$ is an element of
\begin{equation*}
\argmax_{\theta \in \Theta} \frac{1}{n} \log p^\theta_{Y_1^n}(Y_1^n).
\end{equation*}
where $p^\theta_{Y_1^n}$ is the density of the distribution of $Y_1^n$ with respect to the Lebesgue measure under the parameter $\theta$.

The main result of this paper is the following theorem.

\begin{theorem}
\label{th_MLEconsistency}
Assume that $\Kast$ is known and that \textbf{(Amax)}, \textbf{(Amin)}, \textbf{(Aint)}, \textbf{(Acentering)}, \textbf{(Aid)} and \textbf{(Areg)} hold.
Then almost surely, there exists a sequence of (random) permutations $\tau_n$ of $[\Kast]$ such that
\begin{equation*}
\begin{cases}
\forall x,x' \in [\Kast], \quad Q^{\hat{\theta}_n}(\tau_n(x),\tau_n(x')) \underset{n \rightarrow \infty}{\longrightarrow} Q^*(x,x'), \\
\forall z \in \Rbb, \quad \forall x \in [\Kast], \quad \gamma^{\hat{\theta}_n}_{\tau_n(x)}(z) \underset{n \rightarrow \infty}{\longrightarrow} \gamma^*_x(z), \\
\forall x \in [\Kast], \quad \| T^{\hat{\theta}_n}_{\tau_n(x)} - T^*_x \|_{\infty,[0,n]} \underset{n \rightarrow \infty}{\longrightarrow} 0.
\end{cases}
\end{equation*}
\end{theorem}

We say that the maximum likelihood estimator converges \emph{up to permutation of the hidden states}. Without convention on the ordering of the hidden states, it is not possible to get rid of this permutation.

Note that trends converge in supremum norm. This is considerably stronger than assuming that the coefficients of the polynomial converge. It is also intrinsic in the sense that it shows the convergence of the trends when seen as continuous functions and not just as polynomials. Therefore, it could be extended to other types of continuous trends.

\bigskip

Since the trends are polynomial, two trends are either equal up to translation (we say that they are \emph{in the same block}) or diverge from one another.
The first part of the proof, resulting in Theorem~\ref{th_approx_blocs}, makes use of this fact to show that the block 
from which an observation comes can be assumed observed once the number of observations is large enough.

\begin{definition}[Blocks of trends]
\label{def_trend_blocks}
Let $\Rcal^*$ be the equivalence relation defined on $[\Kast]$ by $x \Rcal^* x'$ if and only if $T^*_x - T^*_{x'}$ is constant. Let
\begin{equation*}
\Bcal^* := [\Kast] / \Rcal^*
\end{equation*}
be the set of ``blocks" of true trends.

Let us denote by $\bbf^* : [\Kast] \longrightarrow \Bcal^*$ the quotient mapping and let
\begin{equation*}
B_t := \bbf^*(X_t)
\end{equation*}
be the block from which the observation $Y_t$ is generated.
For each $b \in \Bcal^*$, write
\begin{equation*}
\Tbb^*_b = T^*_{\min \{x | x \in b\}}
\end{equation*}
the reference trend of block $b$ and let
\begin{equation*}
\Delta : \xast \in [\Kast] \longmapsto T^*_\xast(1) - \Tbb^*_{\bbf^*(\xast)}(1)
\end{equation*}
be the function which maps the index of a trend to the difference between the corresponding trend and its reference trend.


For all $\theta \in \Theta$, write
\begin{align*}
\ell_n(\theta) &:= \log p^\theta_{Y_1^n}(Y_1^n) \\
\ell_n^{(Y,B)}(\theta) &:= \log p^\theta_{(Y,B)_1^n}((Y,B)_1^n)
\end{align*}
the log-likelihoods of $\theta$ with respect to the processes $(Y_t)_{t \geq 1}$ and $(Y_t,B_t)_{t \geq 1}$.
\end{definition}

Define the tube of size $M$ of a block \emph{at time $n$} as the set of functions $\Rbb_+ \rightarrow \Rbb$ which are at a distance smaller than $M$ of the reference trend of this block on $[0,n]$ (and as a consequence is at a distance smaller than $M$ during the first $n$ time steps).
Since the observations $(Y_t)_{t \geq 1}$ are clustered around the true trends (see Figure~\ref{fig_ex2_data} for instance), the estimated trends and the tubes will eventually match.
This behaviour is formalized in Theorem~\ref{th_localization} below by proving that there exists a large enough $M$ such that for $n$ large enough, $\hat{\theta}_n \in \ThetaOK(M)$, where for all $M > 0$ and $n \in \Nbb$, $\ThetaOK(M)$ is the subset of $\Theta$ defined by
\begin{align}
\label{eq_vraies_tendances_pas_toutes_seules0}
\forall \theta \in \ThetaOK(M), \qquad
& \forall \xast\in[\Kast],\quad \exists x\in\Xcaltheta,\quad \|T^*_\xast - T_x^{\theta}\|_{\infty,[0,n]}\leq M \\
\label{eq_tendances_parametres_pas_toutes_seules0}
\text{ and } \  & \forall x\in\Xcaltheta,\quad \exists \xast\in[\Kast],\quad \|T^*_\xast - T_x^{\theta}\|_{\infty,[0,n]}\leq M.
\end{align}
If $\theta \in \ThetaOK(M)$, each tube of size $M + \|\Delta\|_\infty$ contains at least one trend $T^\theta_x$, $x \in [K^\theta]$, and each trend $T^\theta_x$, $x \in [K^\theta]$ is in a tube of size $M + \|\Delta\|_\infty$, see Figure~\ref{dessin_ThetaOK} for an illustration.

\begin{figure}
\centering
\includegraphics[scale = 0.6]{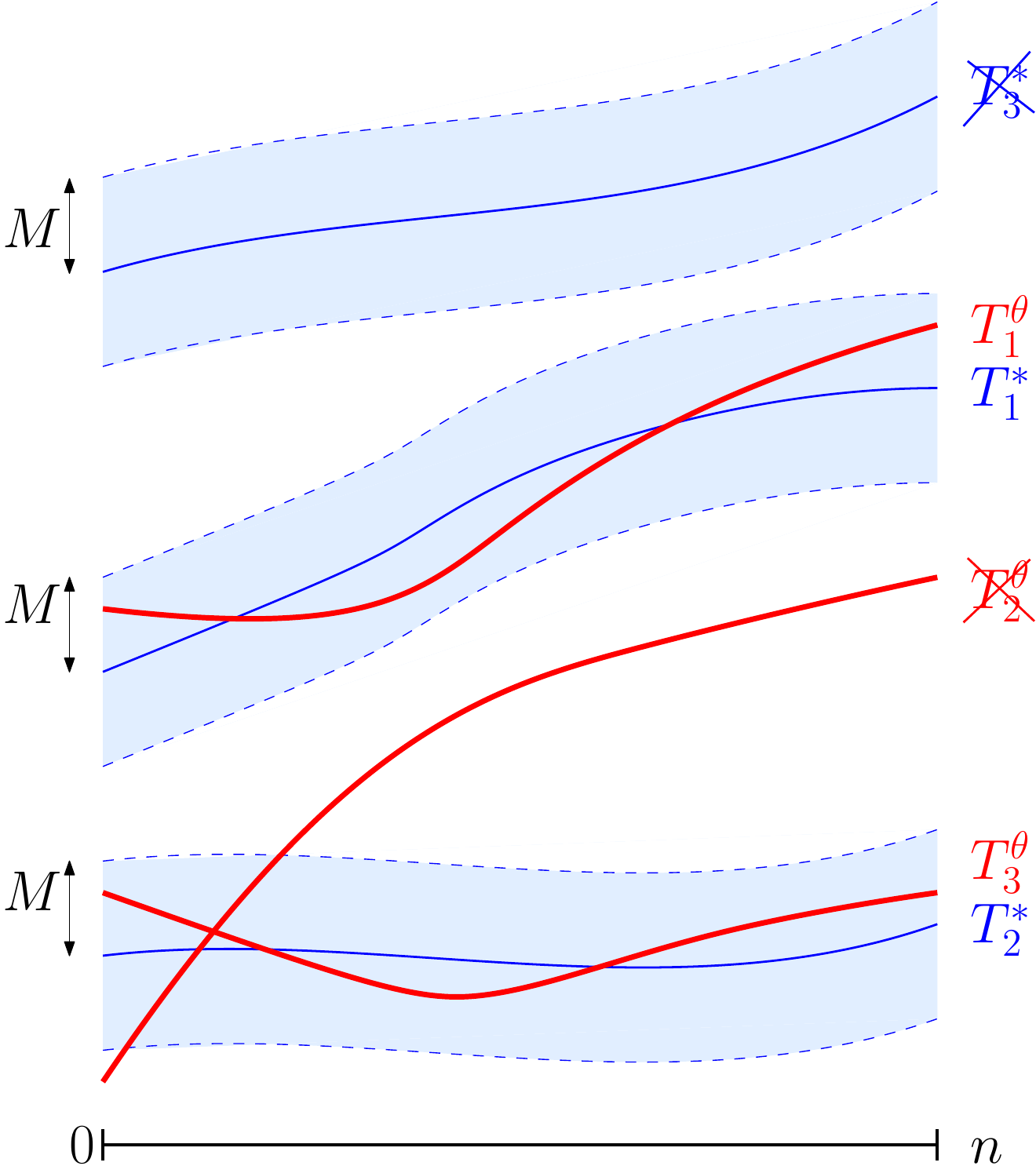}
\caption{Illustration of the tubes of trends and of a $\theta$ that is not in $\ThetaOK(M)$.}
\label{dessin_ThetaOK}
\end{figure}

The following theorem states that when $\theta \in \ThetaOK(M)$ (that is, informally, when the trends are loosely identified) and provided that the number of observations is large enough, the log-likelihood associated to the observed process $(Y_t)_{t\geq 1}$ can be approximated by the log-likelihood of the process $(Y_t,B_t)_{t\geq 1}$ where we know which block of trend each observation comes from.

Let us first specify what we mean by the log-likelihood of the process $(Y_t,B_t)_{t\geq 1}$. Since the true trends are either in the same block or diverge from one another, the tubes eventually have empty intersection. Once this is the case, given $\theta \in \ThetaOK(M)$ and $x \in \Xcaltheta$, we may define $\bbf^\theta(x)$ as the block whose tube contains the trend $T^\theta_x$. The log-likelihood $\ell_n^{(Y,B)}(\theta)$ is the log-likelihood of a inhomogeneous HMM whose emission densities at time $t$ are $(y,b) \in \Rbb \times \Bcal^* \mapsto \gamma^\theta_x(y - T^\theta_x(t)) \one(b = \bbf^\theta(x))$ for each $x \in \Xcaltheta$.

\begin{theorem}[Adding block information]
\label{th_approx_blocs}
Assume \textbf{(Amax)}, \textbf{(Amin)} and \textbf{(Aint)}.
Then for all $M > 0$, almost surely,
\begin{align*}
\sup_{\theta \in \Theta^\text{OK}_n(M)}
	\left| \frac{1}{n} \ell_n(\theta)
		- \frac{1}{n} \ell_n^{(Y,B)}(\theta) \right| \underset{n \rightarrow +\infty}{\longrightarrow} 0.
\end{align*}
\end{theorem}

Theorem \ref{th_approx_blocs} in proved in Appendix \ref{sec_block}, together with the following corollary.

\begin{corollary}
\label{cor_Arate}
Assume \textbf{(Amax)}, \textbf{(Amin)} and \textbf{(Aint)}. Then there exists a finite $\ell(\theta^*)$ such that
\begin{equation*}
\frac{1}{n} \ell_n(\theta^*) \underset{n \rightarrow \infty}{\longrightarrow} \ell(\theta^*).
\end{equation*}
\end{corollary}

Corollary \ref{cor_Arate} is a key argument in the proof of the following theorem, proved in Appendix~\ref{sec_localisation_MLE}.

\begin{theorem}[Localization of the maximum likelihood estimator]
\label{th_localization}
Assume \textbf{(Amax)}, \textbf{(Amin)} and \textbf{(Aint)}.
Then there exists $M > 0$ such that almost surely, there exists $n_\text{loc} \in \Nbb$ such that for all $n \geq n_\text{loc}$,
\begin{equation*}
\hat{\theta}_n \in \ThetaOK(M).
\end{equation*}
\end{theorem}

Let us give an intuition on the last steps of the proof.
We know by Theorem~\ref{th_approx_blocs} that $\frac{1}{n}\ell_n(\theta) \simeq \frac{1}{n}\ell_n^{(Y,B)}(\theta)$. That means we may de-trend the observations by letting $Z'_t = Y_t - \Tbb^*_{B_t}(t)$, and $\frac{1}{n}\ell_n(\theta) \simeq \frac{1}{n}\ell_n^{(Y,B)}(\theta) = \frac{1}{n}\ell_n^{(Z',B)}(\theta)$. Under a parameter $\theta \in \ThetaOK(M)$, this new HMM with trends $(X_t, (Z'_t, B_t))_{t \geq 1}$ is ``locally homogeneous". Let us explain what this means.

Rescale the trends of $Z'_t$ so that the time index lives in $[0,1]$ by defining
\begin{equation}
\label{def_Dthetan}
D^{\theta,n}_x : u \in [0,1] \longmapsto T^\theta_x(nu) - \Tbb^*_{\bbf^\theta(x)}(nu).
\end{equation}
Since the functions $D^{\theta,n}_x$ are bounded and polynomial, they live in a relatively compact subset of the continuous functions on $[0,1]$ for all $n$, $x$ and $\theta \in \ThetaOK(M)$. In particular, we control how fast they vary, and thus how well they can be approximated by piecewise constant functions (see Figure~\ref{dessin_approx_homogene}), which correspond to a HMM with a succession of homogeneous regimes.
\begin{figure}[h]
\centering
\includegraphics[scale = 0.8]{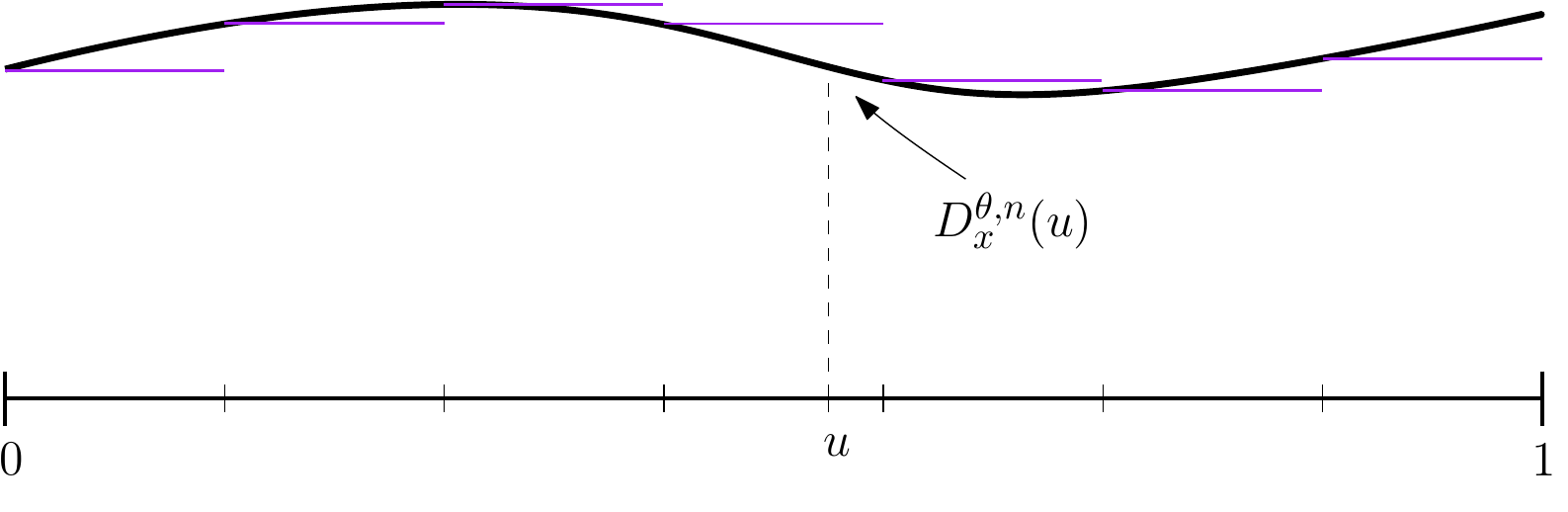}
\caption{Approximation of one trend of the de-trended HMM (black) by the corresponding trend in a succession of homogeneous HMM (purple).}
\label{dessin_approx_homogene}
\end{figure}

Let $\frac{1}{n} \ell_n^{(Y,B)}\left[N \right](\theta)$ be the log-likelihood corresponding to such an approximation when the sample (i.e. the time interval $[0,n]$ before rescaling) is cut into $N$ pieces, and let $\ell^\text{hom}(\theta,n,u)$ be the limit of the normalized log-likelihood of a homogeneous HMM with parameters given by $\theta$ frozen at time $nu$, that is of a homogeneous HMM
$(X_t, (Z'_t, B_t))_{t \geq 1}$ with
transition matrix $Q^\theta$ and emission densities $(z',b) \mapsto \gamma^\theta_x(z' - D^{\theta,n}_x(u)) \one(b = \bbf^\theta(x))$, $x \in [K^\theta]$. The previous paragraph means that the following sequence of approximations holds:
\begin{align*}
\frac{1}{n}\ell_n(\theta) & \simeq \frac{1}{n}\ell_n^{(Y,B)}(\theta) \\
						  & \simeq \frac{1}{n}\ell_n^{(Y,B)}\left[N \right](\theta) \\
						  & \simeq \frac{1}{N}\sum_{i=1}^N \ell^\text{hom}\left(\theta,n,\frac{i-1}{N}\right) \quad \text{when } n \rightarrow \infty \\
						  & \simeq \int_{[0,1]}\ell^{\text{hom}}(\theta,n,u)du \quad \text{by Riemann integration.}
\end{align*}

To conclude, we use that the latter quantity, the \emph{integrated log-likelihood}, is maximized only at the true parameters, as shown in Proposition~\ref{prop_identifiabilite_via_lint_0}.

\bigskip

Let us now state these intermediate results more rigorously.
For all $K' \in \Nbb^*$, for all $K'$-uple $\gamma = (\gamma_x)_{x \in [K']}$ of measurable functions and for all $\Dbf = (D_x)_{x \in [K']} \in \Rbb^{K'}$, let
\begin{equation*}
\tau(\gamma, \Dbf) := (z' \longmapsto \gamma_x(z' - D_x))_{x \in [K']}
\end{equation*}
be the vector of functions $\gamma$ translated by the vector $\Dbf$.

\begin{definition}[Log-likelihood of homogeneous HMM]
\label{def_lhom_0}
Let $K' \in \Nbb^*$ be a positive integer, $\pi$ be a probability measure on $[K']$, $Q$ be a transition matrix of size $K'$, $\gamma$ be a vector of $K'$ emission densities on $\Rbb$ and $\bbf$ be a function $[K'] \longrightarrow \Bcal^*$.
Let $(X_t, (\tilde{Z}_t,\tilde{B}_t))_{t \geq 1}$ be a homogeneous HMM taking values in $[K'] \times (\Rbb \times \Bcal^*)$ with parameter $(\pi, Q, (\gamma_x \otimes \one_{\bbf(x)})_{x \in [K']})$.

Denote by $\frac{1}{n} \ell_n^\text{hom}(\pi,Q,\gamma,\bbf)\{(\tilde{z},\tilde{b})_1^n\}$ (resp. $\ell^\text{hom}(Q,\gamma,\bbf)$) its normalized log-likelihood associated with the observations $(\tilde{z},\tilde{b})_1^n$ (resp. the limit of the log-likelihood, if it exists), that is
\begin{equation*}
\frac{1}{n} \ell_n^\text{hom}(\pi,Q,\gamma,\bbf)\{(\tilde{z},\tilde{b})_1^n\}
	= \frac{1}{n} \log \!\!\! \sum_{x_1^n \in [K']^n} \!\!\! \pi(x_1) Q(x_1, x_2) \dots Q(x_{n-1}, x_n) \prod_{t=1}^n \gamma_{x_t}(\tilde{z}_t) \one_{\bbf(x_t) = \tilde{b}_t}
\end{equation*}
and
\begin{equation}
\label{eq_cvg_logV_homogene_0}
\ell^\text{hom}(Q,\gamma,\bbf)
	= \lim_{n \rightarrow \infty} \frac{1}{n} \ell_n^\text{hom}(\pi,Q,\gamma,\bbf)\{(\tilde{Z},\tilde{B})_1^n\}.
\end{equation}
\end{definition}

The following Lemma ensures the existence of the limit of the normalized log-likelihood in Definition~\ref{def_lhom_0}. It relies on results from \cite{douc2004asymptotic} on homogeneous HMM.

\begin{lemma}
Assume \textbf{(Amax)}, \textbf{(Amin)} and \textbf{(Aint)}. Let $K' \in \Nbb^*$.
Then almost surely, for all $Q \in \Sigma_{K'}^{\sigma_-}$, $\gamma \in \Gamma^{K'}$, $\Dbf \in \Rbb^{K'}$ and $\bbf : [K'] \longrightarrow \Bcal^*$, the quantity
\begin{equation*}
\ell^\text{hom}(Q, \tau(\gamma,\Dbf), \bbf)
\end{equation*}
from Equation~\eqref{eq_cvg_logV_homogene_0} exists, does not depend on the choice of the initial measure $\pi$ and is finite when $(\tilde{Z}_t,\tilde{B}_t)_t = (Y_t - \Tbb^*_{B_t}(t),B_t)_t$.
\end{lemma}

\begin{definition}[Integrated log-likelihood]\label{def_int_loglik1}
We call \emph{integrated log-likelihood} the following mapping:
\begin{equation*}
\ell^\text{int} : (Q,\gamma, \Dfrak, \bbf) \in
\bigcup_{K'=1}^K \Sigma_{K'}^{\sigma_-} \times \Gamma^{K'} \times \Lbf^\infty([0,1])^{K'} \times (\Bcal^*)^{K'} \longmapsto \int_0^1 \ell^\text{hom}(Q, \tau(\gamma,\Dfrak(u)), \bbf) du
\end{equation*}
\end{definition}

For all $M > 0$, let
\begin{equation}
\label{eq_def_Dcal1}
\Dcal(M) := \bigcup_{n \geq 4K(d+1)} \left\{ D^{\theta,n}_x \, | \, \theta \in \Theta^\text{OK}_n(M), x \in \Xcaltheta \right\}
\end{equation}
where $D^{\theta,n}_x$ is defined in Equation~\eqref{def_Dthetan}.
By definition of $\ThetaOK(M)$, $\Dcal(M)$ is a subset of $\Lbf^\infty([0,1])$ uniformly bounded by $M + \|\Delta\|_\infty$. Let $\Cl(\Dcal(M))$ be its closure in $\Lbf^\infty([0,1])$.

\begin{prop}
\label{prop_compacite_Dcal}
Let $M > 0$. Assume \textbf{(Amax)} and \textbf{(Amin)}.
Then $\Cl(\Dcal(M))$ is a compact subset of $\Ccal^0([0,1])$.
\end{prop}

\begin{theorem}
\label{th_conv_int_loglik}
Assume \textbf{(Amax)}, \textbf{(Amin)}, \textbf{(Aint)} and \textbf{(Areg)}.
For all $M > 0$, the function $\ell^\text{int}$ is continuous on
$\bigcup_{K'=1}^K \Sigma_{K'}^{\sigma_-} \times \Gamma^{K'} \times \Cl(\Dcal(M))^{K'} \times (\Bcal^*)^{K'}$
and almost surely,
\begin{align}
\label{eq_cvg_unif_vers_lint}
\sup_{\theta \in \Theta^\text{OK}_n(M)}
	&\left| \frac{1}{n} \ell_n(\theta)
	- \ell^\text{int}(Q^\theta,\gamma^\theta,(D^{\theta,n}_x)_x,\bbf^\theta) \right|
	\underset{n \rightarrow \infty}{\longrightarrow} 0.
\end{align}
\end{theorem}
\noindent
Proposition~\ref{prop_compacite_Dcal} and Theorem~\ref{th_conv_int_loglik} are proved in the first part of Appendix \ref{sec_integrated_likelihood}.
\bigskip

Now, assume that $\Kast$ is known and take $K = \Kast$.
The following proposition ensures that the only maximizer of the integrated log-likelihood is the true parameter.
\begin{prop}
\label{prop_identifiabilite_via_lint_0}
Assume that \textbf{(Amax)}, \textbf{(Amin)}, \textbf{(Aint)}, \textbf{(Acentering)}, \textbf{(Aid)} and \textbf{(Areg)} hold. Let $(Q,\gamma, \Dfrak, \bbf) \in \Sigma_{K}^{\sigma_-} \times \Gamma^{K} \times \text{Cl}(\Dcal)^K \times (\Bcal^*)^K$ be a maximizer of $\ell^\text{int}$, then $\Dfrak$ is constant and $( Q, \gamma,\Dfrak, \bbf) = ( Q^*, \gamma^*,\Delta, \bbf^*)$ up to permutation of the hidden states.
\end{prop}

The proof of Proposition \ref{prop_identifiabilite_via_lint_0} is in the second part of Appendix \ref{sec_integrated_likelihood}.\bigskip

We may now prove the consistency of the maximum likelihood estimator. By Theorem~\ref{th_localization}, there exists $M > 0$ such that almost surely, there exists a (random) integer $n_\text{loc}$ such that for all $n \geq n_\text{loc}$, $\hat{\theta}_n \in \Theta^\text{OK}_n(M)$. For $n \geq n_\text{loc}$, let
\begin{equation*}
\begin{cases}
(Q_n,\gamma_n) := (Q^{\hat{\theta}_n},\gamma^{\hat{\theta}_n}), \\
\Dfrak_n := (D^{\hat{\theta}_n, n}_x)_{x \in [K]}, \\
\bbf_n := \bbf^{\hat{\theta}_n}.
\end{cases}
\end{equation*}

For all $n \geq n_\text{loc}$, $(Q_n,\gamma_n, \Dfrak_n, \bbf_n) \in \Sigma_{K}^{\sigma_-} \times \Gamma^{K} \times \text{Cl}(\Dcal)^K \times (\Bcal^*)^K$, and this set is compact by compactness of $\Gamma$ and Proposition~\ref{prop_compacite_Dcal}. Let $(Q,\gamma,\Dfrak,\bbf)$ be the limit of a convergent subsequence $(Q_{\varphi(n)},\gamma_{\varphi(n)}, \Dfrak_{\varphi(n)}, \bbf_{\varphi(n)})_{n \geq 1}$, then by continuity of $\ell^\text{int}$ and by the uniform convergence of equation~\eqref{eq_cvg_unif_vers_lint},
\begin{equation*}
\frac{1}{\varphi(n)} \ell_{\varphi(n)}(\hat{\theta}_{\varphi(n)})
	\underset{n \rightarrow \infty}{\longrightarrow} \ell^\text{int}(Q,\gamma,\Dfrak,\bbf) \leq \ell^\text{int}(Q^*,\gamma^*,\Delta,\bbf^*) = \ell(\theta^*),
\end{equation*}
by Proposition~\ref{prop_identifiabilite_via_lint_0}, and by definition of the maximum likelihood estimator
\begin{equation*}
\frac{1}{\varphi(n)} \ell_{\varphi(n)}(\theta^*)
	\leq \frac{1}{\varphi(n)} \ell_{\varphi(n)}(\hat{\theta}_{\varphi(n)}).
\end{equation*}

Hence $\ell^\text{int}(Q,\gamma,\Dfrak,\bbf) = \ell(\theta^*) = \ell^\text{int}(Q^*,\gamma^*,\Delta,\bbf^*)$, which means that $(Q,\gamma,\Dfrak,\bbf) = (Q^*,\gamma^*,\Delta,\bbf^*)$ up to permutation of the hidden states by Proposition~\ref{prop_identifiabilite_via_lint_0}. Thus, the MLE sequence $(Q_n,\gamma_n, \Dfrak_n, \bbf_n)_{n \geq 1}$ has only one possible limit: the true parameter. Theorem~\ref{th_MLEconsistency} follows.

\section{Simulations}
\label{sec_simulations}

\subsection{First experiment: rate of convergence}

In this first example, we consider the following HMM with trends $(X_t,Y_t)_{t\geq 1}$
with $\Kast = 3$ states. The emission distributions are centered Gaussian distributions with respective variances $(\sigma_1^*)^2 = 5$, $(\sigma_2^*)^2 = 10$ and $(\sigma_3^*)^2 = 15$. The trends are given by
$$T_1^*(t) = \alpha (t + 10^4)^2,\quad T_2^*(t) = T_1^*(t) - 5,\quad T_3^*(t) = 3T_1^*(t),$$
with $\alpha = 10^{-8}$. Thus $T_1^*$ and $T_2^*$ belong to the same block while $T_3^*$ diverges from the two other trends. Finally, the transition matrix is
$$Q^* = \begin{pmatrix}
0.7 & 0.2 & 0.1\\
0.2 & 0.6 & 0.2\\
0.1 & 0.1 & 0.8
\end{pmatrix}.$$

The model is as follows. The number of states $\Kast=3$ is assumed known. The set $\Gamma$ of possible emission distributions is taken as the set of centered Gaussian distributions, without constraint on the variance. The lower bound on the transition probabilities is chosen as $\sigma_- = 0$. Even if $\Gamma$ is not compact and $\sigma_-$ is not positive, contrary to the theoretical result, the maximum likelihood estimator is still able to recover the parameters. Finally, the maximum degree of the trends is $d=4$.
Note that the model is over-parametrized, as it contains all the trends in $\mathbb{R}_4[X]$ while the degree of the true trends is only $2$. This reflects the fact that in practice, we may not know the degree of the true trends.

We simulate $10$ independent realizations $(X_t,Y_t)_{1\leq t\leq n_{\max}}$ with $n_{\max} = 10^5$. Figure \ref{fig_ex2_data} shows the data points corresponding to one of these $10$ trajectories.
Figure \ref{fig_conv_trends} illustrates the fact that for each $x \in [\Kast]$, $\|T^*_x - T^{\hat{\theta}_n}_x\|_{\infty,[0,n]}\underset{n\to\infty}{\longrightarrow}0$, where $T^{\hat{\theta}_n}_x$ is the maximum likelihood estimator of $T^*_x$ computed from the first $n$ observations.

\begin{figure}[p]
\centering
\includegraphics[scale = 0.5]{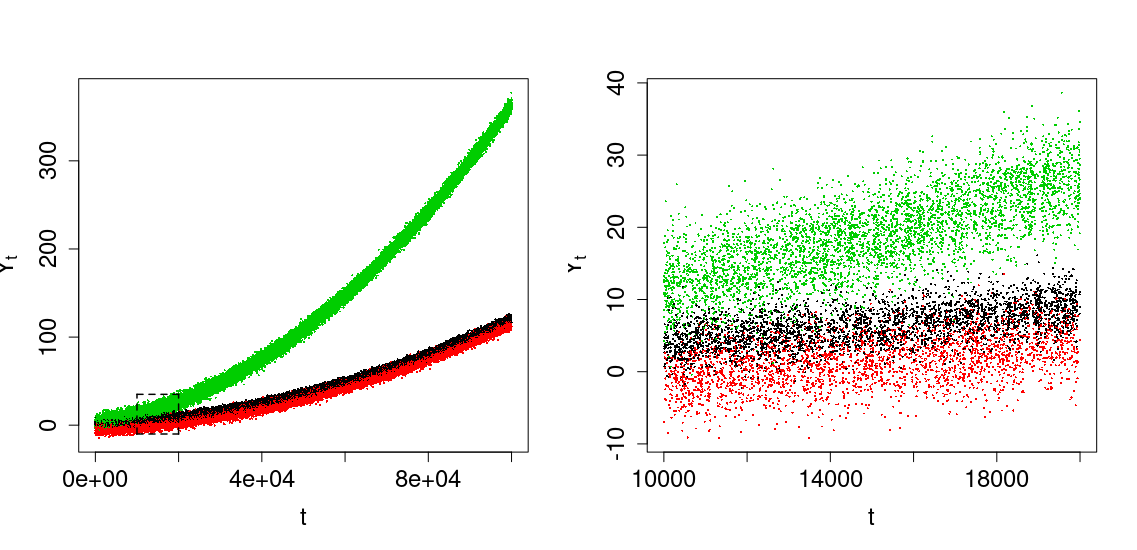}
\caption{Left panel: a simulated trajectory of length $100000$ of the observations of the HMM with trends $(X_t,Y_t)_{t\geq 1}$. Each color corresponds to a different state (state $1$: black, state $2$: red, state $3$: green). Right panel: a focus on $10000\leq t\leq 20000$.}\label{fig_ex2_data}
\end{figure} 

\begin{figure}[p]
\centering
\includegraphics[scale = 0.5]{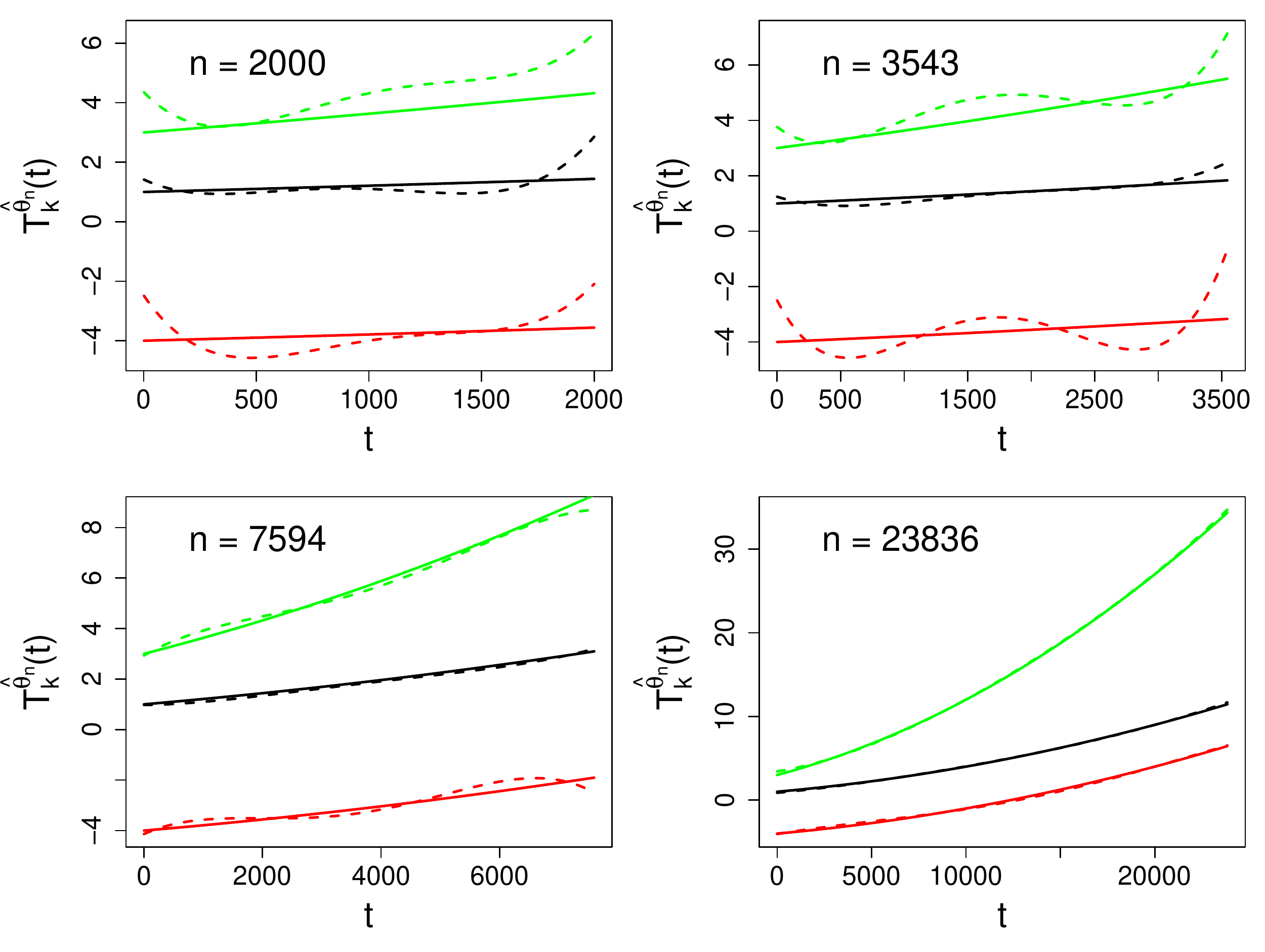}
\caption{Convergence of the estimated trends $T^{\hat{\theta}_n}_x$ (dashed lines) to the true trends $T^*_x$ (solid lines).}\label{fig_conv_trends}
\end{figure} 

For each simulated trajectory and for several $n\in\{1,\dots,n_{\max}\}$, we compute:
\begin{itemize}
\item[$\bullet$] The errors on the trends $\|T^*_x - T^{\hat{\theta}_n}_x\|_{\infty,[0,n]}$, $1\leq x\leq \Kast$,
\item[$\bullet$] The error on the transition matrix $\|Q^* - Q^{\hat{\theta}_n}\|_F$, where $\|\cdot\|_F$ denotes the Frobenius norm,
\item[$\bullet$] The error on the variances $\max_x|(\sigma^*_x)^2 - \hat{\sigma}^2_x|$.
\end{itemize}

We plotted the logarithm of these errors against $\log n$ (see Figure \ref{fig_ex2_errors}). These graphs suggest a linear decrease of the logarithm of the errors with respect to $\log n$. Therefore, we can conjecture that as $n$ tends to infinity, $n^\alpha \|Q^* - Q^{\hat{\theta}_n}\|_F$ is bounded in probability for some $\alpha > 0$, and that the same property holds for the other parameters (possibly with a different $\alpha$). Based on this experiment, it seems reasonable to conjecture that the maximum likelihood estimator achieves the parametric rate of convergence $\alpha = 0.5$. However, proving this claim would require further investigation that go beyond the scope of this paper. It is also worth noting that in this example, we chose simple parametric emission distributions, whereas the model described in Section \ref{sec_model} for which our main result holds only requires the set of emission densities to be compact, not necessarily finite-dimensional. 

\begin{figure}[t]
\centering
\includegraphics[width = \textwidth]{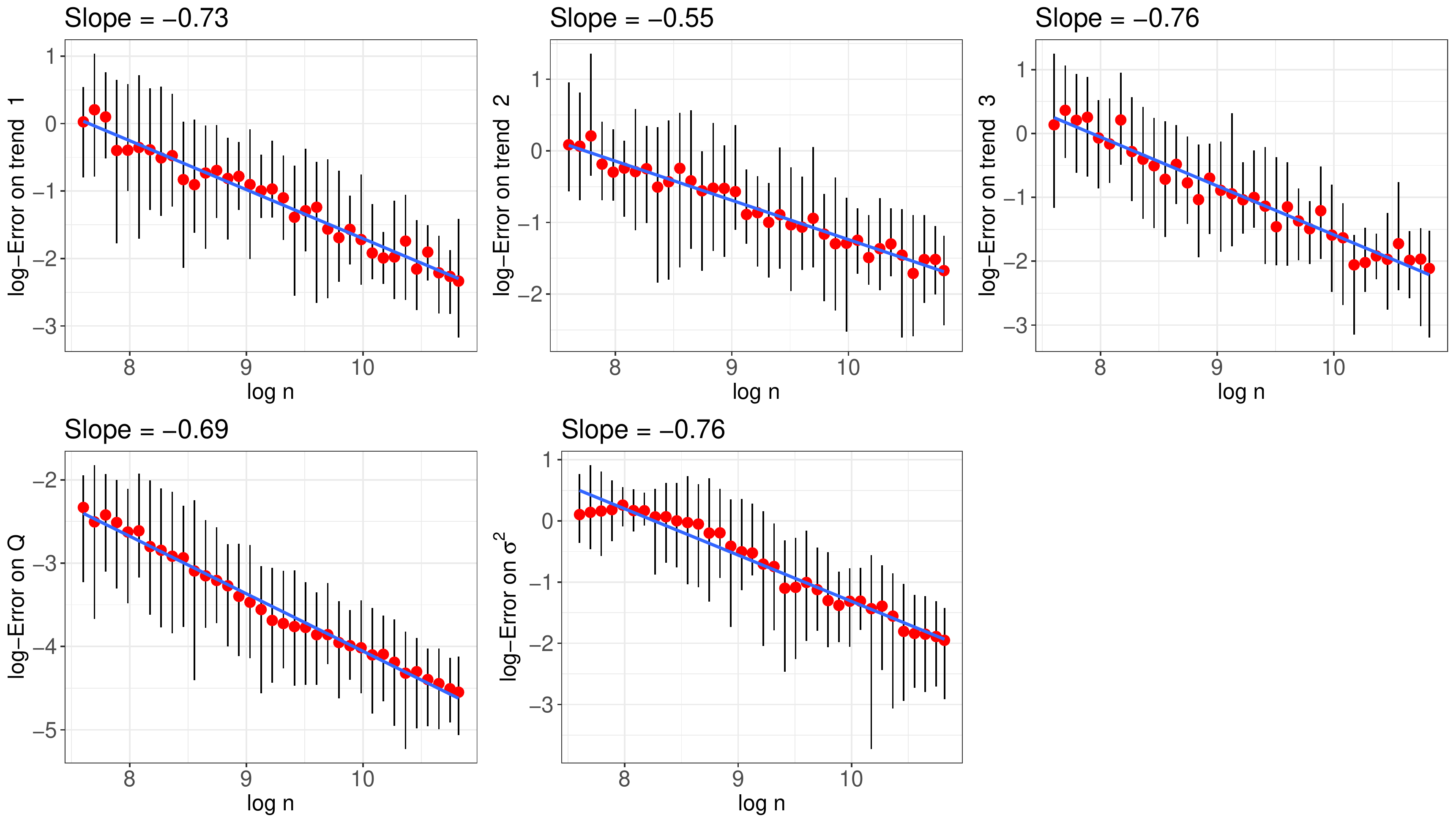}
\caption{Rate of convergence of the maximum likelihood estimator. The red dots are the means of the log-errors across the $10$ simulations, the blue line is the linear fit to these means and the black bars represent the range of the log-errors across the simulations.}\label{fig_ex2_errors}
\end{figure}

\subsection{Second experiment: the trends have not diverged yet}

In this section, we consider a HMM with trends whose trends have not diverged during the experiment, which is an assumption on which the proofs rely heavily. We show that the MLE is still able to recover the trends and the homogeneous parameter accurately. This is especially relevant for practical applications where one may not have enough time to see the trends diverge.

Fix the maximum number of observations to $n=10000$ and consider the following HMM with trends $(X_t,Y_t)_{t\geq 1}$ with $\Kast=2$ states. The trends are defined by $T^*_1(t) = 0$ and $T^*_2(t) = 3\left(\frac{t-\frac{n}{2}}{\frac{n}{2}}\right)^2-1$. The emission distributions are centered gaussian distributions with respective variances $(\sigma_1^*)^2 = 1$ and $(\sigma_2^*)^2 = 2$ and the transition matrix is
$$Q^* = \begin{pmatrix}
0.7 & 0.3 \\
0.2 & 0.8 \\
\end{pmatrix}.$$

\begin{figure}
\centering
\includegraphics[scale=0.5]{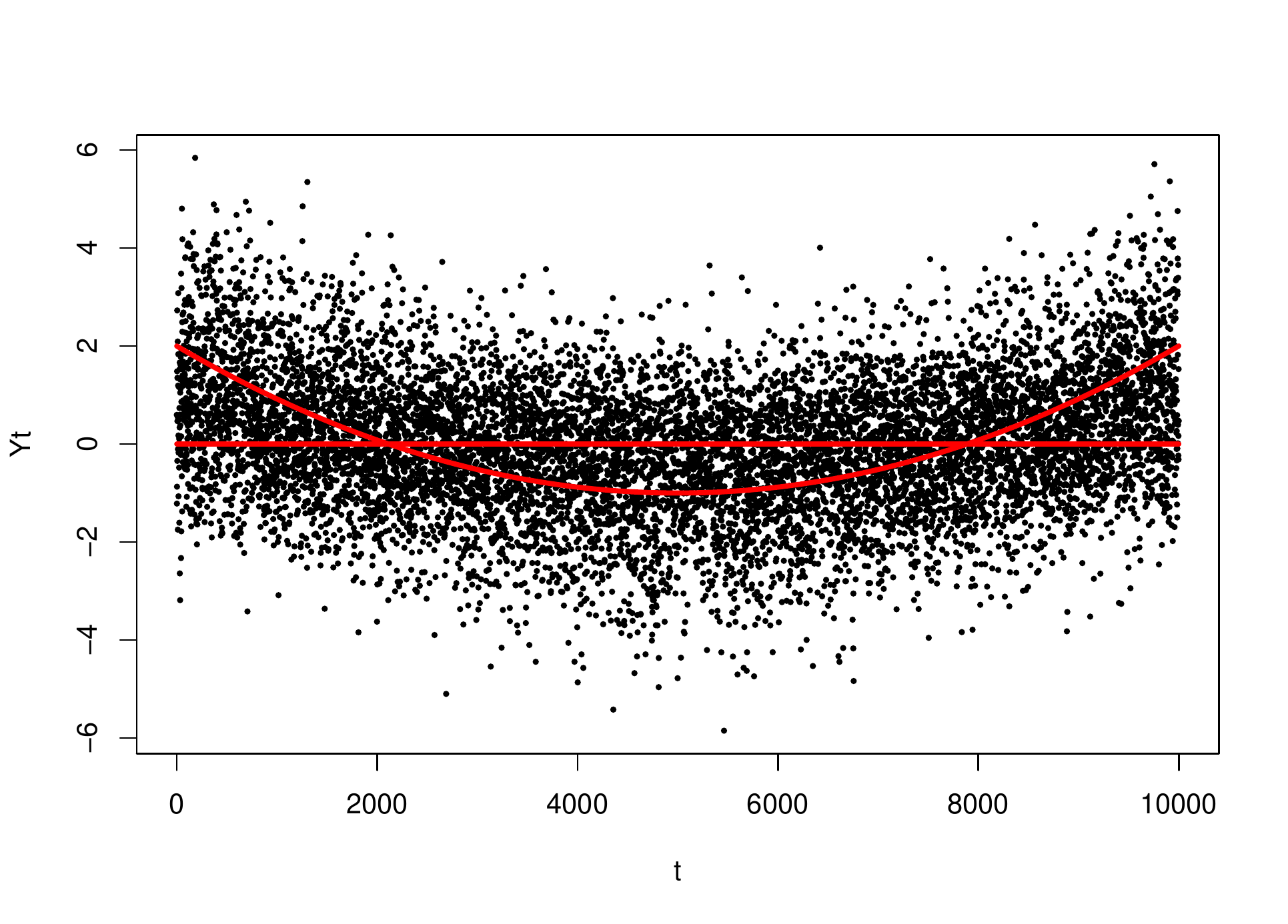}
\caption{Simulated data points. The red lines are the true trends.}
\label{fig_simu_data}
\end{figure}

Figure \ref{fig_simu_data} shows the simulated observations $(Y_t)_{1\leq t\leq n}$ as well as the true trends $T^*_1$ and $T^*_2$. The two states are not clearly separated: the trends will eventually diverge, but we don't have enough observations to make use of this. However, the maximum likelihood estimator is able to recover the trends even in this situation.

The model is as follows. The number of states $\Kast = 2$ is assumed known. As in the previous section, we take $\sigma_- = 0$, $\Gamma$ as the set of centered Gaussian distributions and $d = 4$. 

\begin{figure}
\centering
\includegraphics[scale=0.5]{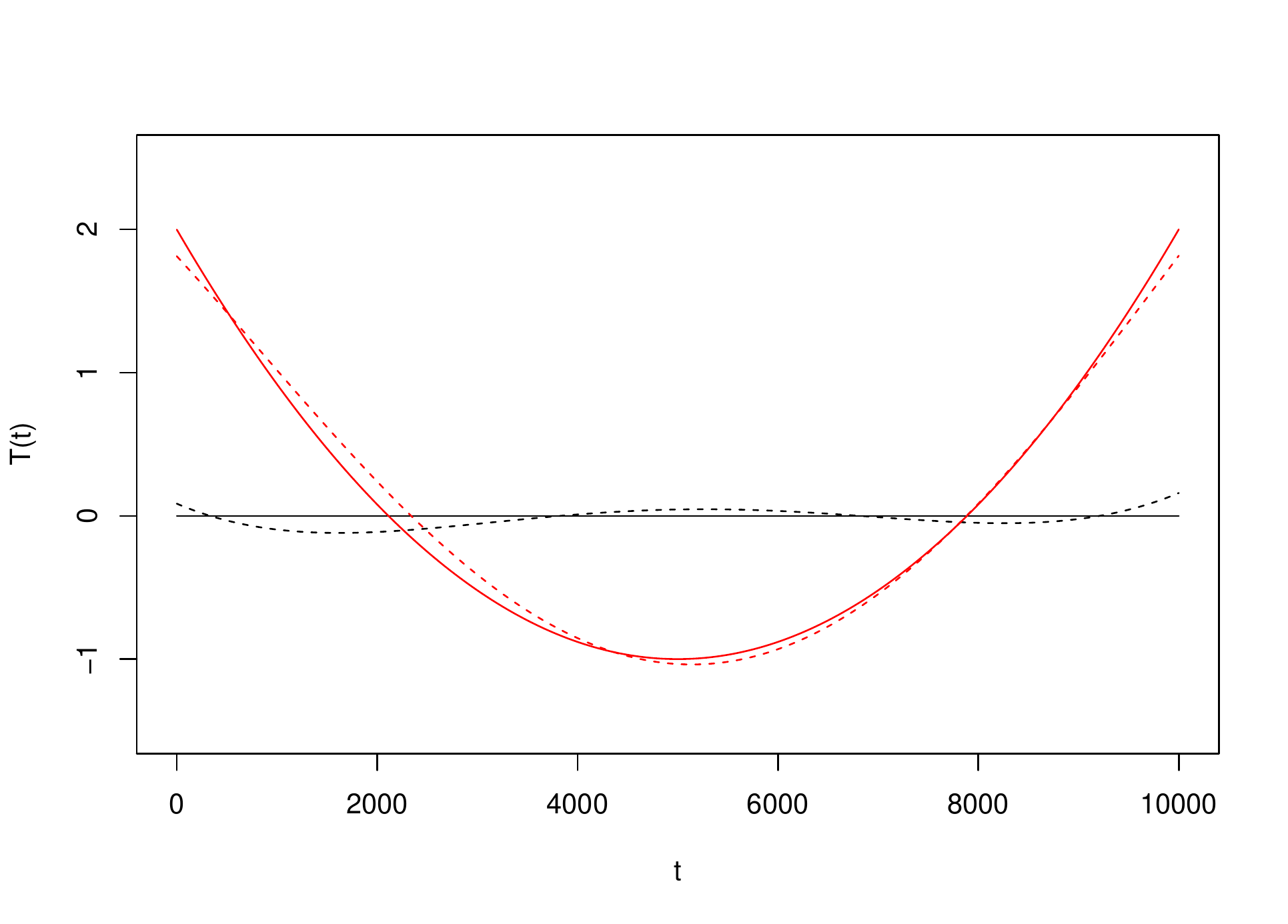}
\caption{True (full lines) and estimated (dashed lines) trends.}\label{fig_estim_result}
\end{figure}

Figure \ref{fig_estim_result} shows the estimated trends obtained 
using the EM algorithm~\citep{EM} together with the true trends. The estimated transition matrix and variances are
\begin{equation*}
\hat{Q} = \begin{pmatrix}
0.74 & 0.26 \\
0.22 & 0.78
\end{pmatrix},\quad
(\hat{\sigma}^2_1,\hat{\sigma}^2_2) = (1.13,2.11).
\end{equation*}

The precision of these estimations can be further improved by increasing the number of data points. The trends and the homogeneous parameter are already well estimated. An intuition to explain this convergence even when the trends have not diverged is that the trends vary slowly enough for the homogeneous approximation of Theorem~\ref{th_conv_int_loglik} to hold.

\appendix

\section{Block approximation}
\label{sec_block}

In this appendix, we shall prove Theorem \ref{th_approx_blocs} and Corollary \ref{cor_Arate}.

Let us begin with a few definitions. Let
\begin{equation*}
E : t \in \Nbb^* \longmapsto \inf_{b, b' \in \Bcal^* \text{ s.t. } b \neq b'} |\Tbb^*_b(t) - \Tbb^*_{b'}(t)|
\end{equation*}
be the minimum difference between the reference trends of two distinct blocks at time $t$. Note that $E(t)$ diverges to $+\infty$ since the true trends are polynomials.

Let $M > 0$ and
\begin{equation*}
n_1(M) := \inf \{ n \in \Nbb^* \, | \, \forall t \geq n, E(t) > 2M \}.
\end{equation*}

Let $n \geq n_1(M)$, so that the tubes of size $M$ at time $n$ have no intersection.
Let $\theta \in \ThetaOK(M - \|\Delta\|_\infty)$ and $x \in \Xcaltheta$, where $\Delta$ is as in Definition~\ref{def_trend_blocks}. Then, by equation~\eqref{eq_tendances_parametres_pas_toutes_seules0}, there exists a block $\bbf^\theta(x) \in \Bcal^*$ (which is unique since $n \geq n_1(M)$) such that $T^\theta_x$ is in the tube of size $M$ of $\bbf^\theta(x)$ a time $n$.
In particular, 
\begin{equation}
\label{eq_majoration_ecart_vraieettheta}
\sup_{t \in \{1, \dots, n\}} | \Tbb^*_{\bbf^\theta(x)}(t) - T^\theta_{x}(t) | \leq M.
\end{equation}

The proof aims to make the following approximations rigorous.
\begin{align*}
\log p^\theta_{Y_t | Y_1^{t-1}}(Y_t | Y_1^{t-1})
	&\approx \log p^\theta_{Y_t, B_t | Y_1^{t-1}} (Y_t, B_t | Y_1^{t-1}) \\
	&\approx \log p^\theta_{Y_t, B_t | Y_1^{t-1}, B_1^{t-1}} (Y_t, B_t | Y_1^{t-1}, B_1^{t-1}),
\end{align*}
hence the two following steps.

\subsection{Step 1: introduction of the trend block \texorpdfstring{$B_t$}{} in the log-likelihood}

Assume \textbf{(Amax)}, \textbf{(Amin)} and \textbf{(Aint)}. Let us show that the following quantity tends to $0$ uniformly in $\theta$.
\begin{multline*}
\log p^\theta_{Y_t | Y_1^{t-1}}(Y_t | Y_1^{t-1})
	- \log p^\theta_{Y_t, B_t | Y_1^{t-1}}(Y_t, B_t | Y_1^{t-1}) \\
	= \log \left( \frac{\displaystyle
			\sum_{x_t \in \Xcaltheta} p^\theta(X_t = x_t | Y_1^{t-1}) \gamma_{x_t}(Y_t - T^\theta_{x_t}(t))
		}{\displaystyle
			\sum_{x_t \in \Xcaltheta} p^\theta(X_t = x_t | Y_1^{t-1}) \gamma_{x_t}(Y_t - T^\theta_{x_t}(t)) \one_{\bbf^\theta(x_t) = B_t}
		} \right).
\end{multline*}

This can be rewritten as
\begin{align}
\nonumber \log p^\theta_{Y_t | Y_1^{t-1}}(Y_t | Y_1^{t-1})
	&- \log p^\theta_{Y_t, B_t | Y_1^{t-1}}(Y_t, B_t | Y_1^{t-1}) \\
\nonumber     &= \log ( p^\theta_{B_t | Y_1^t}(B_t | Y_1^t)^{-1} ) \\
    \label{eq_block1_interpretationPreuve}
    &= \log \left( \left\{ 1 - p^\theta_{B_t | Y_1^t}\left(\{b \in \Bcal^* \text{ s.t. } b \neq B_t\} | Y_1^t\right) \right\}^{-1} \right).
\end{align}

Intuitively, when $t$ is large, since the trends get further from one another, the probability to get the wrong block converges to zero.

\begin{lemma}
\label{lemma_approx_logV_blocPresent}
Assume \textbf{(Amax)} and \textbf{(Amin)}. Then for all $t \in \Nbb^*$,
\begin{align*}
&\sup_{\theta \in \Theta^\text{OK}_n(M-\|\Delta\|_{\infty})}
\left| \log p^\theta_{Y_t | Y_1^{t-1}}(Y_t | Y_1^{t-1})
	- \log p^\theta_{Y_t, B_t | Y_1^{t-1}}(Y_t, B_t | Y_1^{t-1}) \right| \\
	&\quad \leq \log\left( \left\{1 - \frac{g(\{E(t) - M - Z^{\max}_t - \|\Delta\|_\infty\}_+)}{\sigma_- m(M + Z^{\max}_t + \|\Delta\|_\infty)}\right\}_+^{-1} \wedge \frac{g(0)}{\sigma_- m( M + Z^{\max}_t + \|\Delta\|_\infty)} \right) \\
    &\quad =: h(E(t), Z^{\max}_t)
\end{align*}
with the convention $\{z\}_+^{-1} = +\infty$ if $z \leq 0$.
\end{lemma}


\begin{proof}
Proof in Section~\ref{sec_proof_approx_logV_blocPresent}.
\end{proof}

Summing in $t$ results in
\begin{align*}
\sup_{\theta \in \Theta^\text{OK}_n(M-\|\Delta\|_{\infty})}
\left| \frac{1}{n} \ell_n(\theta) - \frac{1}{n} \sum_{t=1}^n \log p^\theta_{Y_t, B_t | Y_1^{t-1}}(Y_t, B_t | Y_1^{t-1}) \right|
	\leq \frac{1}{n} \sum_{t=1}^n h(E(t), Z^{\max}_t).
\end{align*}

The function $e \in \Rbb_+ \longmapsto h(e,z)$ is non-negative, non-increasing for all $z \in \Rbb$ and tends to $0$ as $e$ tends to $+\infty$. Moreover, under Assumption \textbf{(Aint)}, $h(0, Z^{\max}_1)$ is integrable by definition of $h$. Thus, under Assumption \textbf{(Aint)}, the law of large numbers implies that for all $E > 0$
\begin{align*}
\limsup_{n \rightarrow \infty} \sup_{\theta \in \Theta^\text{OK}_n(M-\|\Delta\|_{\infty})}
\left| \frac{1}{n} \ell_n(\theta) - \frac{1}{n} \sum_{t=1}^n \log p^\theta_{Y_t, B_t | Y_1^{t-1}}(Y_t, B_t | Y_1^{t-1}) \right|
	\leq \Ebb^*[h(E, Z^{\max}_1)].
\end{align*}

The dominated convergence theorem ensures that $\Ebb^*[h(E, Z^{\max}_1)] \longrightarrow 0$ as $E \longrightarrow +\infty$. Thus, we obtain the following uniform approximation of the normalized log-likelihood:
\begin{align}
\label{eq_bilan_ajoutblocpresent}
\sup_{\theta \in \Theta^\text{OK}_n(M-\|\Delta\|_{\infty})}
\left| \frac{1}{n} \ell_n(\theta) - \frac{1}{n} \sum_{t=1}^n \log p^\theta_{Y_t, B_t | Y_1^{t-1}}(Y_t, B_t | Y_1^{t-1}) \right|
	\underset{n \rightarrow \infty}{\longrightarrow} 0.
\end{align}

\subsection{Step 2: conditioning on the blocks \texorpdfstring{$B_1^{t-1}$}{}}

Assume \textbf{(Amax)} and \textbf{(Amin)}. The following lemma is a consequence of the lower bound on the transition matrices, see for instance Lemma 1 and Corollary 1 of \cite{douc2004asymptotic}.

\begin{lemma}[Exponential forgetting]
There exists $C > 0$ such that for all $n \in \Nbb^*$, $y_1^n \in \Rbb^n$, $\theta \in \Theta$ and for all probability measures $\pi, \pi'$ on $\Xcaltheta$:
\begin{equation*}
\sum_{x \in \Xcaltheta} | p^\theta_{X_n | Y_1^{n-1}}(x | y_1^{n-1}, X_0 \sim \pi) - p^\theta_{X_n | Y_1^{n-1}}(x | y_1^{n-1}, X_0 \sim \pi') | \leq C \rho^n
\end{equation*}
with $\rho = 1 - \frac{\sigma_-}{1 - \sigma_-} \in (0,1)$.
\end{lemma}

Besides, under \textbf{(Aerg)}, for all $\theta \in \Theta$, $x \in \Xcaltheta$, $y_1^{n-1} \in \Rbb^n$ and for all probability measure $\pi$ on $\Xcaltheta$:
\begin{equation*}
p^\theta_{X_n | Y_1^{n-1}}(x | y_1^{n-1}, X_0 \sim \pi) \geq \sigma_-.
\end{equation*}

Hence, using the inequality $| \log x - \log y | \leq \frac{|x - y|}{x \wedge y}$ for all $x,y > 0$: for all $n \in \Nbb^*$, $\theta \in \Theta$, $y_1^n \in \Rbb^n$, $b \in \Bcal^*$ and for all probability measures $\pi, \pi'$ on $\Xcaltheta$:
\begin{align*}
&| \log p^\theta_{Y_n, B_n | Y_1^{n-1}}(y_n, b | y_1^{n-1}, X_0 \sim \pi)
	- \log p^\theta_{Y_n, B_n | Y_1^{n-1}}(y_n, b | y_1^{n-1}, X_0 \sim \pi') | \\
	&\leq \frac{
		\displaystyle \sum_{x \in \Xcaltheta} | p^\theta_{X_n | Y_1^{n-1}}(x | y_1^{n-1}, X_0 \sim \pi) - p^\theta_{X_n | Y_1^{n-1}}(x | y_1^{n-1}, X_0 \sim \pi') | p^\theta_{Y_n, B_n | X_n}(y_n, b | x)
	}{
		\displaystyle \sigma_- \sum_{x \in \Xcaltheta} p^\theta_{Y_n, B_n | X_n}(y_n, b | x)} \\
	&\leq \frac{C}{\sigma_-} \rho^n.
\end{align*}

Changing the constant $C$ if necessary, for all $a \in \Nbb^*$:
\begin{multline*}
\sup_{\theta \in \Theta^\text{OK}_n(M-\|\Delta\|_{\infty})}
\Bigg| \frac{1}{n} \sum_{t=1}^n \log p^\theta_{Y_t, B_t | Y_1^{t-1}}(Y_t, B_t | Y_1^{t-1}) \\
	- \frac{1}{n} \sum_{t=1}^n \log p^\theta_{Y_t, B_t | Y_1^{t-1}, B_1^{t-a}}(Y_t, B_t | Y_1^{t-1}, B_1^{t-a}) \Bigg|
	\leq C \rho^a.
\end{multline*}

It remains to condition on $B_{t-a+1}^{t-1}$.

\begin{lemma}
\label{lemma_conditionnementBlocsPasses}
Assume \textbf{(Amax)} and \textbf{(Amin)}. Then for all $a \in \Nbb^*$,
\begin{align*}
\sup_{\theta \in \Theta^\text{OK}_n(M-\|\Delta\|_{\infty})}
\Bigg| \frac{1}{n} \sum_{t=1}^n & \log p^\theta_{Y_t, B_t | Y_1^{t-1}, B_1^{t-a}} (Y_t, B_t | Y_1^{t-1}, B_1^{t-a}) \\
	& - \frac{1}{n} \sum_{t=1}^n \log p^\theta_{Y_t, B_t | Y_1^{t-1}, B_1^{t-1}}(Y_t, B_t | Y_1^{t-1}, B_1^{t-1}) \Bigg| \\
	&\qquad \leq \frac{2 a K^2}{\sigma_-^3} \frac{1}{n} \sum_{i=1}^{n-1}
		\left( 1 \wedge \frac{g(\{E(i) - M - Z^{\max}_i - \|\Delta\|_\infty\}_+)}{m(M + Z^{\max}_i + \|\Delta\|_\infty)} \right) \\
	&\qquad =: \frac{2 a K^2}{\sigma_-^3} \frac{1}{n} \sum_{i=1}^{n-1} h'(E(i), Z^{\max}_i).
\end{align*}
\end{lemma}

\begin{proof}
We show that when \textbf{(Aerg)} holds, for all $\theta \in \Theta^\text{OK}_n(M-\|\Delta\|_{\infty})$, $t \leq n$, $a \in \Nbb^*$, $y_1^t \in \Ycal^t$ and $b_1^t \in (\Bcal^*)^t$,
\begin{align}
\nonumber
\Big| \log p^\theta_{Y_t, B_t | Y_1^{t-1}, B_1^{t-a}} & (y_t, b_t | y_1^{t-1}, b_1^{t-a})
	- \log p^\theta_{Y_t, B_t | Y_1^{t-1}, B_1^{t-1}}(y_t, b_t | y_1^{t-1}, b_1^{t-1}) \Big| \\
	\label{eq_proof_recentBlocks_erreurBlock}
	&\leq \frac{2}{\sigma_-} p^\theta_{B_{(t-a+1) \vee 1}^{t-1} | Y_1^{t-1}, B_1^{t-a}}\left((\Bcal^*)^{(a \wedge t)-1} \setminus \{ b_{(t-a+1) \vee 1}^{t-1} \} | y_1^{t-1}, b_1^{t-a}\right) \\
	\label{eq_proof_recentBlocks_densites}
	&\leq \frac{2 K^2}{\sigma_-^3} \sum_{i=(t-a+1) \vee 1}^{t-1}
		\frac{\displaystyle \sup_{x_i \in \Xcaltheta \text{ s.t. } \bbf^\theta(x_i) \neq b_i}\gamma^\theta_{x_i}(y_i - T^\theta_{x_i}(i))}{\displaystyle \sum_{x \in \Xcaltheta} \gamma^\theta_{x}(y_i - T^\theta_x(i))}.
\end{align}

Then, we show that under \textbf{(Amax)} and \textbf{(Amin)}, for all $\theta \in \Theta^\text{OK}_n(M-\|\Delta\|_{\infty})$ and $i \leq n$:
\begin{align}
\label{eq_proof_recentBlocks_utilisationEncadrements}
\frac{\displaystyle \sup_{x_i \in \Xcaltheta \text{ s.t. } \bbf^\theta(x_i) \neq B_i}\gamma^\theta_{x_i}(Y_i - T^\theta_{x_i}(i))}{\displaystyle \sum_{x \in \Xcaltheta} \gamma^\theta_{x}(Y_i - T^\theta_x(i))}
	&\leq 1 \wedge \frac{g(\{E(i) - M - Z^{\max}_i - \|\Delta\|_\infty\}_+)}{m(M + Z^{\max}_i + \|\Delta\|_\infty)} \\
\nonumber	&=: h'(E(i), Z^{\max}_i),
\end{align}
and the lemma follows by summing over $t$ and $i$.
The details of the proof can be found in Section~\ref{sec_proof_conditionnementBlocsPasses}.
\end{proof}

Therefore, almost surely,
\begin{align*}
\limsup_{n \rightarrow +\infty} \sup_{\theta \in \Theta^\text{OK}_n(M-\|\Delta\|_{\infty})}
\Bigg| \frac{1}{n} \sum_{t=1}^n &\, \log p^\theta_{Y_t, B_t | Y_1^{t-1}}(Y_t, B_t | Y_1^{t-1}) \\
	&- \frac{1}{n} \sum_{t=1}^n \log p^\theta_{Y_t, B_t | Y_1^{t-1}, B_1^{t-1}} (Y_t, B_t | Y_1^{t-1}, B_1^{t-1}) \Bigg| \\
	&\qquad\qquad \leq \limsup_{n \rightarrow +\infty} \left(C \rho^a
		+ \frac{2 a K^2}{\sigma_-^3} \frac{1}{n} \sum_{i=1}^{n-1} h'(E(i), Z^{\max}_i) \right) \\
	&\qquad\qquad \leq C \rho^a + \frac{2 a K^2}{\sigma_-^3} \Ebb^*[h'(E, Z^{\max}_1)] \\
	&\qquad\qquad \leq C' \left( - \Ebb^*[h'(E, Z^{\max}_1)] \log \Ebb^*[h'(E, Z^{\max}_1)] \right)
\end{align*}
for some explicit constant $C'$ and for all $E$ sufficiently large to have $\Ebb^*[h'(E, Z^{\max}_1)] < 1/2$ using Assumption \textbf{(Adiv)}, the law of large numbers, the fact that the mapping $e \longmapsto h'(e,z)$ is non-negative, bounded by $0$ and $1$ and non-increasing for all $z$, and by taking $a = \lceil \frac{\log \Ebb^*[h'(E, Z^{\max}_1)]}{\log \rho} \rceil$. Since the function $e \longmapsto h'(e,z)$ tends to $0$ as $e$ tends to $+\infty$ for all $z$, the dominated convergence theorem ensures that almost surely,
\begin{multline*}
\sup_{\theta \in \Theta^\text{OK}_n(M-\|\Delta\|_{\infty})}
\Bigg| \frac{1}{n} \sum_{t=1}^n \log p^\theta_{Y_t, B_t | Y_1^{t-1}}(Y_t, B_t | Y_1^{t-1}) \\
	- \frac{1}{n} \sum_{t=1}^n \log p^\theta_{Y_t, B_t | Y_1^{t-1}, B_1^{t-1}}(Y_t, B_t | Y_1^{t-1}, B_1^{t-1}) \Bigg|
	\underset{n \rightarrow \infty}{\longrightarrow} 0.
\end{multline*}

Let $\frac{1}{n} \ell_n^{(Y,B)}(\theta) = \frac{1}{n} \log p^\theta_{(Y, B)_1^n}((Y, B)_1^n)$. Combining the above equation with equation~\eqref{eq_bilan_ajoutblocpresent} yields Theorem \ref{th_approx_blocs}.

\subsection{Application: existence and finiteness of the relative entropy rate}

Theorem \ref{th_approx_blocs} implies that Corollary \ref{cor_Arate} holds: since $\theta^* \in \Theta^\text{OK}_n(M)$ for all $n\in\Nbb^*$ and $M>0$,
\begin{align*}
	\left| \frac{1}{n} \ell_n(\theta^*)
		- \frac{1}{n} \ell_n^{(Y,B)}(\theta^*) \right| \underset{n \rightarrow +\infty}{\longrightarrow} 0.
\end{align*}

Let $Z'_t = Y_t - \Tbb^*_{B_t}(t)$. Then $\frac{1}{n} \ell_n^{(Y,B)}(\theta^*) = \frac{1}{n} \ell_n^{(Z',B)}(\theta^*)$. Moreover, under $\theta^*$, the process $(X_t, (Z'_t, B_t))_{t \geq 1}$ is a homogeneous and ergodic HMM with emission densities $(z',b) \longmapsto \gamma^*_\xast(z' - \Delta(\xast)) \one(b = \bbf^*(\xast))$ for $\xast \in [\Kast]$ with respect to the measure $\text{Leb} \otimes \mu_{\Bcal^*}$, where $\mu_{\Bcal^*}$ is the counting measure on $\Bcal^*$.

Since it is homogeneous and ergodic, \cite{barron1985entropyrate} shows that there exists $\ell(\theta^*) > -\infty$ such that
\begin{equation*}
\frac{1}{n} \ell_n^{(Z',B)}(\theta^*) \longrightarrow \ell(\theta^*).
\end{equation*}

Then, all emission densities are upper bounded by $g(0)$ under \textbf{(Amax)}, so that the positive part of their logarithm is integrable. Therefore, \cite{leroux92MLEHMM} implies that $\ell(\theta^*) < +\infty$ and Corollary \ref{cor_Arate} is proved.

\subsection{Proofs}

\subsubsection{Proof of Lemma~\ref{lemma_approx_logV_blocPresent} (current block)}
\label{sec_proof_approx_logV_blocPresent}

First note that this quantity is non-negative: the denominator contains less terms, and all of them are non-negative. Hence it is enough to find an upper bound. To this aim we will use Assumptions \textbf{(Amax)}, \textbf{(Amin)} and \textbf{(Aerg)}:
\begin{align*}
&\left| \log p^\theta_{Y_t | Y_1^{t-1}}(Y_t | Y_1^{t-1})
	- \log p^\theta_{Y_t, B_t | Y_1^{t-1}}(Y_t, B_t | Y_1^{t-1}) \right| \\
	&\leq \log \left(\frac{ g(0) }{\displaystyle \sigma_- \sup_{x_t \in \Xcaltheta \text{ s.t. } \bbf^\theta(x_t) = B_t} m( | Y_t - T^\theta_{x_t}(t) | ) } \right) \\
	&\leq \log \frac{g(0)}{\sigma_-} + \sum_{x^*_t \in [\Kast]} \one_{X_t = x^*_t} \left(-\log m\left( \inf_{x_t \in \Xcaltheta \text{ s.t. } \bbf^\theta(x_t) = \bbf^*(x^*_t)} | Z_t + T^*_{x^*_t}(t) - T^\theta_{x_t}(t) | \right)\right).
\end{align*}

\noindent
When $X_t = x^*_t$,
\begin{align*}
\inf_{x_t \in \Xcaltheta \text{ s.t. } \bbf^\theta(x_t) = \bbf^*(x^*_t)} & | Z_t + T^*_{x^*_t}(t) - T^\theta_{x_t}(t) | \\
	&\leq | Z_t | + \inf_{x_t \in \Xcaltheta \text{ s.t. } \bbf^\theta(x_t) = \bbf^*(x^*_t)} |\Tbb^*_{\bbf^*(x^*_t)}(t) - T^\theta_{x_t}(t) | + \Delta(x^*_t) \\
	&\leq Z^{\max}_t + M + \| \Delta \|_\infty
\end{align*}
using equation~\eqref{eq_majoration_ecart_vraieettheta}, hence
\begin{align*}
-\log m\left( \inf_{x_t \in \Xcaltheta \text{ s.t. } \bbf^\theta(x_t) = \bbf^*(x^*_t)} | Z_t + T^*_{x^*_t}(t) - T^\theta_{x_t}(t) | \right)
	\leq -\log m( M + Z^{\max}_t + \| \Delta \|_\infty).
\end{align*}

\noindent
This yields
\begin{align*}
\left| \log p^\theta_{Y_t | Y_1^{t-1}}(Y_t | Y_1^{t-1})
	- \log p^\theta_{Y_t, B_t | Y_1^{t-1}}(Y_t, B_t | Y_1^{t-1}) \right|
	\leq \log \frac{g(0)}{\sigma_- m( M + Z^{\max}_t + \| \Delta \|_\infty)}.
\end{align*}

Let us show the second bound. We can rewrite it as
\begin{align*}
&\log p^\theta_{Y_t | Y_1^{t-1}}(Y_t | Y_1^{t-1})
	- \log p^\theta_{Y_t, B_t | Y_1^{t-1}}(Y_t, B_t | Y_1^{t-1}) \\
	&= - \log \left( 1 - \frac{\displaystyle
			\sum_{x_t \in \Xcaltheta} p^\theta(X_t = x_t | Y_1^{t-1}) \gamma_{x_t}(Y_t - T^\theta_{x_t}(t)) \one_{\bbf^\theta(x_t) \neq B_t}
		}{\displaystyle
			\sum_{x_t \in \Xcaltheta} p^\theta(X_t = x_t | Y_1^{t-1}) \gamma_{x_t}(Y_t - T^\theta_{x_t}(t))
		} \right) \\
	&= - \log \left( 1 - \sum_{x^*_t \in [\Kast]} \one_{X_t = x^*_t} \underbrace{\frac{\displaystyle
			\sum_{x_t \in \Xcaltheta} p^\theta(X_t = x_t | Y_1^{t-1}) \gamma_{x_t}(Z_t + T^*_{x^*_t}(t) - T^\theta_{x_t}(t)) \one_{\bbf^\theta(x_t) \neq B_t}
		}{\displaystyle
			\sum_{x_t \in \Xcaltheta} p^\theta(X_t = x_t | Y_1^{t-1}) \gamma_{x_t}(Z_t + T^*_{x^*_t}(t) - T^\theta_{x_t}(t))
		}}_{\textbf{(*)}} \right).
\end{align*}

Using \textbf{(Amax)}, \textbf{(Amin)} and \textbf{(Aerg)},
\begin{align*}
\textbf{(*)} &\leq \frac{\displaystyle
			\sup_{x_t \in \Xcaltheta \text{ s.t. } \bbf^\theta(x_t) \neq B_t} g(|Z_t + T^*_{x^*_t}(t) - T^\theta_{x_t}(t)|)
		}{\displaystyle \sigma_-
			\sum_{x_t \in \Xcaltheta} m(|Z_t + T^*_{x^*_t}(t) - T^\theta_{x_t}(t)|)
		} \\
	&\leq \frac{\displaystyle \sup_{x_t \in \Xcaltheta \text{ s.t. } \bbf^\theta(x_t) \neq B_t} g(\{|T^*_{x^*_t}(t) - T^\theta_{x_t}(t)| - Z^{\max}_t\}_+)
		}{\displaystyle \sigma_- \sup_{x_t \in \Xcaltheta} m(|T^*_{x^*_t}(t) - T^\theta_{x_t}(t)| + Z^{\max}_t) }.
\end{align*}

Let $x \in \Xcaltheta$ such that $\bbf^\theta(x) \neq B_t$. When $X_t = x^*_t$,
\begin{align*}
|Y_t - T^\theta_x(t)|
	&= |Z_t + T^*_{x^*_t}(t) - T^\theta_x(t)| \\
	&\geq |\Tbb^*_{B_t}(t) - T^\theta_x(t)| - Z^{\max}_t - \Delta(x^*_t) \\
	&\geq |\Tbb^*_{B_t}(t) - \Tbb^*_{\bbf^{\theta}(x)}(t)| - M - Z^{\max}_t - \|\Delta\|_\infty \\
	&\geq E(t) - M - Z^{\max}_t - \|\Delta\|_\infty
\end{align*}
and
\begin{align*}
\sup_{x_t \in \Xcaltheta} m(|T^*_{x^*_t}(t) - T^\theta_{x_t}(t)| + Z^{\max}_t)
	&\leq m(\inf_{x_t \in \Xcaltheta} |T^*_{x^*_t}(t) - T^\theta_{x_t}(t)| + Z^{\max}_t) \\
	&\leq m(M + Z^{\max}_t + \|\Delta\|_\infty)
\end{align*}
using equation \eqref{eq_majoration_ecart_vraieettheta}. Therefore,
\begin{align}
\label{eq_controle_ca}
\textbf{(*)}
	&\leq \frac{g(\{E(t) - M - Z^{\max}_t - \|\Delta\|_\infty\}_+)}{\sigma_- m(M + Z^{\max}_t + \|\Delta\|_\infty)}.
\end{align}

\subsubsection{Proof of Lemma~\ref{lemma_conditionnementBlocsPasses} (recent blocks)}
\label{sec_proof_conditionnementBlocsPasses}

\paragraph*{Proof of equation~\eqref{eq_proof_recentBlocks_erreurBlock}}

Without loss of generality, one may assume $a \leq t$ (otherwise, the proof holds by replacing $a$ by $a \wedge t$).

\begin{align*}
p^\theta_{Y_t, B_t | Y_1^{t-1}, B_1^{t-a}} & (y_t, b_t | y_1^{t-1}, b_1^{t-a}) \\
	&= p^\theta_{Y_t, B_t | Y_1^{t-1}, B_1^{t-1}}(y_t, b_t | y_1^{t-1}, b_1^{t-1}) \\
			&\qquad \times p^\theta_{B_{t-a+1}^{t-1} | Y_1^{t-1}, B_1^{t-a}}(b_{t-a+1}^{t-1} | y_1^{t-1}, b_1^{t-a}) \\
		&\quad + p^\theta_{Y_t, B_t | Y_1^{t-1}, B_1^{t-a}, B_{t-a+1}^{t-1}}(y_t, b_t | y_1^{t-1}, b_1^{t-a}, (\Bcal^*)^{a-1} \setminus \{b_{t-a+1}^{t-1}\} ) \\
			&\qquad \times p^\theta_{B_{t-a+1}^{t-1} | Y_1^{t-1}, B_1^{t-a}}((\Bcal^*)^{a-1} \setminus \{b_{t-a+1}^{t-1}\} | y_1^{t-1}, b_1^{t-a}),
\end{align*}
hence
\begin{align*}
&| p^\theta_{Y_t, B_t | Y_1^{t-1}, B_1^{t-a}}(y_t, b_t | y_1^{t-1}, b_1^{t-a})
	- p^\theta_{Y_t, B_t | Y_1^{t-1}, B_1^{t-1}}(y_t, b_t | y_1^{t-1}, b_1^{t-1}) | \\
	&\qquad \leq p^\theta_{B_{t-a+1}^{t-1} | Y_1^{t-1}, B_1^{t-a}}((\Bcal^*)^{a-1} \setminus \{b_{t-a+1}^{t-1}\} | y_1^{t-1}, b_1^{t-a}) \Big[ p^\theta_{Y_t, B_t | Y_1^{t-1}, B_1^{t-1}}(y_t, b_t | y_1^{t-1}, b_1^{t-1}) \\
		&\hspace{3cm} + p^\theta_{Y_t, B_t | Y_1^{t-1}, B_1^{t-1}}(y_t, b_t | y_1^{t-1}, b_1^{t-a}, (\Bcal^*)^{a-1} \setminus \{b_{t-a+1}^{t-1}\}) \Big] \\
	&\qquad \leq 2 p^\theta_{B_{t-a+1}^{t-1} | Y_1^{t-1}, B_1^{t-a}}((\Bcal^*)^{a-1} \setminus \{b_{t-a+1}^{t-1}\} | y_1^{t-1}, b_1^{t-a})
		\sum_{x \in \Xcaltheta} p^\theta_{Y_t, B_t | X_t}(y_t, b_t | x).
\end{align*}

Finally, since under \textbf{(Aerg)}
\begin{align*}
p^\theta_{Y_t, B_t | Y_1^{t-1}, B_1^{t-a}}(y_t, b_t | y_1^{t-1}, b_1^{t-a})
	&= \sum_{x \in \Xcal} p^\theta_{Y_t, B_t | X_t}(y_t, b_t | x) p^\theta_{X_t | Y_1^{t-1}, B_1^{t-a}}(x | y_1^{t-1}, b_1^{t-a}) \\
	&\geq \sigma_- \sum_{x \in \Xcal} p^\theta_{Y_t, B_t | X_t}(y_t, b_t | x)
\end{align*}
and the same inequality holds for $p^\theta_{Y_t, B_t | Y_1^{t-1}, B_1^{t-1}}(y_t, b_t | y_1^{t-1}, b_1^{t-1})$, we obtain equation~\eqref{eq_proof_recentBlocks_erreurBlock} using $| \log x - \log y| \leq \frac{|x-y|}{x \wedge y}$ for all $x,y > 0$.

\paragraph*{Proof of equation~\eqref{eq_proof_recentBlocks_densites}}

Since
\begin{equation*}
(\Bcal^*)^{a-1} \setminus \{ b_{t-a+1}^{t-1} \}
	= \bigcup_{i=t-a+1}^{t-1} \left( (\Bcal^*)^{i-(t-a+1)} \times \left(\Bcal^* \setminus \{b_i\}\right) \times (\Bcal^*)^{t-1-i} \right),
\end{equation*}
by union bound,
\begin{align*}
&p^\theta_{B_{t-a+1}^{t-1} | Y_1^{t-1}, B_1^{t-a}}((\Bcal^*)^{a-1} \setminus \{ b_{t-a+1}^{t-1} \} | y_1^{t-1}, b_1^{t-a}) \\
	&\leq \sum_{i=t-a+1}^{t-1} p^\theta_{B_i | Y_1^{t-1}, B_1^{t-a}}(\Bcal^* \setminus \{b_i\} | y_1^{t-1}, b_1^{t-a}) \\
	&= \sum_{i=t-a+1}^{t-1} \sum_{x_i \in \Xcaltheta}
		p^\theta_{B_i | X_i}(\Bcal^* \setminus \{b_i\} | x_i)
		p^\theta_{X_i | Y_1^{t-1}, B_1^{t-a}}(x_i | y_1^{t-1}, b_1^{t-a}) \\
	&= \sum_{i=t-a+1}^{t-1} \sum_{x_i \in \Xcaltheta} \one_{\bbf^\theta(x_i) \neq b_i} \sum_{x_{i-1}, x_{i+1} \in \Xcaltheta}
		p^\theta_{X_i | Y_i, X_{i-1}, X_{i+1}}(x_i | y_i, x_{i-1}, x_{i+1}) \\
		&\hspace{3cm} \times p^\theta_{X_{i-1}, X_{i+1} | Y_1^{t-1}, B_1^{t-a}}(x_{i-1}, x_{i+1} | y_1^{t-1}, b_1^{t-a}).
\end{align*}

Then, use that for all $x_{i-1}, x_{i+1} \in \Xcaltheta$,
\begin{equation*}
p^\theta_{X_{i-1}, X_{i+1} | Y_1^{t-1}, B_1^{t-a}}(x_{i-1}, x_{i+1} | y_1^{t-1}, b_1^{t-a}) \leq 1
\end{equation*}
and that by the Markov property and \textbf{(Aerg)} for all $x_{i-1}, x, x_{i+1} \in \Xcaltheta$
\begin{align*}
p^\theta_{X_i | X_{i-1}, X_{i+1}}(x | x_{i-1}, x_{i+1})
	&= \frac{p^\theta_{X_{i+1} | X_i}(x_{i+1} | x) p^\theta_{X_i | X_{i-1}}(x | x_{i-1})}{p^\theta_{X_{i+1} | X_{i-1}}(x_{i+1} | x_{i-1})}\\
		&\geq Q^\theta(x,x_{i+1}) Q^\theta(x_{i-1},x) \\
&\geq \sigma_-^2,
\end{align*}
so that
\begin{align*}
p^\theta_{X_i | Y_i, X_{i-1}, X_{i+1}} & (x_i | y_i, x_{i-1}, x_{i+1}) \\
	&= \frac{p^\theta_{X_i | X_{i-1}, X_{i+1}}(x_i | x_{i-1}, x_{i+1}) \gamma^\theta_{x_i}(y_i - T^\theta_{x_i}(i))}{\displaystyle \sum_{x \in \Xcaltheta} p^\theta_{X_i | X_{i-1}, X_{i+1}}(x | x_{i-1}, x_{i+1}) \gamma^\theta_{x}(y_i - T^\theta_x(i))} \\
	&\leq \frac{p^\theta_{X_i | X_{i-1}, X_{i+1}}(x_i | x_{i-1}, x_{i+1}) \gamma^\theta_{x_i}(y_i - T^\theta_{x_i}(i))}{\displaystyle \sigma_-^2 \sum_{x \in \Xcaltheta} \gamma^\theta_{x}(y_i - T^\theta_x(i))}.
\end{align*}

Thus,
\begin{align*}
&p^\theta_{B_{t-a+1}^{t-1} | Y_1^{t-1}, B_1^{t-a}}((\Bcal^*)^{a-1} \setminus \{b_{t-a+1}^{t-1} \} | y_1^{t-1}, b_1^{t-a}) \\
	&\leq\sum_{i=t-a+1}^{t-1} \sum_{x_{i-1}, x_{i+1} \in \Xcaltheta}
		\frac{\displaystyle \sum_{x_i \in \Xcaltheta} \one_{\bbf^\theta(x_i) \neq b_i} p^\theta_{X_i | X_{i-1}, X_{i+1}}(x_i | x_{i-1}, x_{i+1}) \gamma^\theta_{x_i}(y_i - T^\theta_{x_i}(i))}{\displaystyle \sigma_-^2 \sum_{x \in \Xcaltheta} \gamma^\theta_{x}(y_i - T^\theta_x(i))} \\
	&\leq \frac{K^2}{\sigma_-^2} \sum_{i=t-a+1}^{t-1}
		\frac{\displaystyle \sup_{x_i \in \Xcaltheta \text{ s.t. } \bbf^\theta(x_i) \neq b_i}\gamma^\theta_{x_i}(y_i - T^\theta_{x_i}(i))}{\displaystyle \sum_{x \in \Xcaltheta} \gamma^\theta_{x}(y_i - T^\theta_x(i))}.
\end{align*}

\paragraph*{Proof of equation~\eqref{eq_proof_recentBlocks_utilisationEncadrements}}

Using \textbf{(Amax)} and \textbf{(Amin)},
\begin{align*}
&\frac{\displaystyle \sup_{x_i \in \Xcaltheta \text{ s.t. } \bbf^\theta(x_i) \neq B_i}\gamma^\theta_{x_i}(Y_i - T^\theta_{x_i}(i))}{\displaystyle \sum_{x \in \Xcaltheta} \gamma^\theta_{x}(Y_i - T^\theta_x(i))}
	\leq \left( 1 \wedge \frac{\displaystyle \sup_{x_i \in \Xcaltheta \text{ s.t. } \bbf^\theta(x_i) \neq B_i} g(|Y_i - T^\theta_{x_i}(i)|)}{\displaystyle \sup_{x \in \Xcaltheta} m(|Y_i - T^\theta_x(i)|)} \right) \\
	&\leq \sum_{x^*_i \in [\Kast]} \one_{X_t = x^*_i} \left( 1 \wedge \frac{\displaystyle \sup_{x_i \in \Xcaltheta \text{ s.t. } \bbf^\theta(x_i) \neq \bbf^*(x^*_i)} g(|Z_i + T^*_{x^*_i}(i) - T^\theta_{x_i}(i)|)}{\displaystyle \sup_{x \in \Xcaltheta} m(|Z_i + T^*_{x^*_i}(i) - T^\theta_x(i)|)} \right) \\
	&\leq 1 \wedge \frac{g(\{E(t) - M - Z^{\max}_i - \|\Delta\|_\infty\}_+)}{m(M + Z^{\max}_i + \|\Delta\|_\infty)}
\end{align*}
by the same arguments as in the control of \textbf{(*)} in equation~\eqref{eq_controle_ca}.

\section{Localization of the MLE}
\label{sec_localisation_MLE}

In this section we shall prove Theorem \ref{th_localization}. Assume \textbf{(Aerg)}, \textbf{(Amin)}, \textbf{(Amax)} and \textbf{(Aint)}.

\subsection{Preliminary: compactness results}
\label{sec_compacity}

Recall that for all $M > 0$ and $n \in \Nbb$, $\ThetaOK(M)$ is the subset of $\Theta$ defined by
\begin{align}
\label{eq_vraies_tendances_pas_toutes_seules_rappel}
\forall \theta \in \ThetaOK(M), \qquad
& \forall \xast \in [\Kast],\quad \exists x \in \Xcaltheta, \quad \|T^*_\xast - T_x^{\theta}\|_{\infty,[0,n]}\leq M \\
\label{eq_tendances_parametres_pas_toutes_seules_rappel}
\text{ and } \  & \forall x \in \Xcaltheta, \quad \exists \xast \in [\Kast],\quad \|T^*_\xast - T_x^{\theta}\|_{\infty,[0,n]}\leq M.
\end{align}

To prove that the maximum likelihood estimator belongs to such a set, we consider relaxed versions of~\eqref{eq_vraies_tendances_pas_toutes_seules_rappel} and~\eqref{eq_tendances_parametres_pas_toutes_seules_rappel}. Let
\begin{eqnarray*}
&\displaystyle I_{n,D}(x,\xast,\theta) := \left\{
		t \in \{1, \dots, n\} : |T^*_\xast(t) - T^\theta_x(t)| \leq D
	\right\}, \\
&\displaystyle \Tcal(\alpha,n,D) := \left\{ \theta \in \Theta : \left|
		\bigcap_{\xast \in [\Kast]} \bigcup_{x \in \Xcaltheta} I_{n,D}(x,\xast,\theta)
	\right| \geq n\alpha \right\},
\\
&\displaystyle \Ucal(\alpha,n,D) := \left\{ \theta \in \Theta : \forall x \in \Xcaltheta, \left|
		\bigcup_{\xast \in [\Kast]} I_{n,D}(x,\xast,\theta)
	\right| \geq n\alpha \right\}.
\end{eqnarray*}

$\Tcal$ corresponds to a relaxation of~\eqref{eq_vraies_tendances_pas_toutes_seules_rappel} and contains the parameters $\theta$ such that all true trends are close to at least one parameter trend during most of the first time steps. Likewise, $\Ucal$ corresponds to a relaxation of~\eqref{eq_tendances_parametres_pas_toutes_seules_rappel} and contains the parameters $\theta$ whose trends are close to at least one true trend during most of the first time steps.

By construction of these sets, for each $\theta \in \Tcal(\alpha,n,D)$ (resp. $\theta \in \Ucal(\alpha,n,D)$) and for each $\xast \in [\Kast]$ (resp. $x \in \Xcaltheta$), there exists at least one $x \in \Xcaltheta$ (resp. $\xast \in [\Kast]$) such that $I_{n,D}(x,\xast,\theta) \geq n\alpha/K$.

\begin{prop}
\label{lemma_theta_dans_I_implique_tendances_proches}
For all $\alpha \in (0,1]$ and $D > 0$, there exists $M(\alpha,D) > 0$ and $n_0 := 4K(d+1) / \alpha$ such that
\begin{multline*}
\forall n \geq n_0, \  \forall \theta \in \Theta, \  \forall x \in \Xcaltheta, \  \forall \xast \in [\Kast], \\
	|I_{n,D}(x,\xast,\theta)| \geq n \alpha / K
		\  \Rightarrow \  \| T^*_\xast - T^\theta_{x} \|_{\infty, [0,n]} \leq M(\alpha,D).
\end{multline*}
\end{prop}
It follows that $\theta \in \Tcal(\alpha,n,D)$ (resp. $\theta \in \Ucal(\alpha,n,D)$) is equivalent to~\eqref{eq_vraies_tendances_pas_toutes_seules_rappel} (resp.~\eqref{eq_tendances_parametres_pas_toutes_seules_rappel}) up to changing $M$:

\begin{definition}
\label{def_xetxstar_reference}
Let $\alpha \in (0,1]$, $n \geq 1$ and $D > 0$.
\begin{itemize}
\item[---] For all $\theta \in \Tcal(\alpha,n,D)$ and $\xast \in [\Kast]$, let $x(\theta,\xast,n, \alpha,D)$ be the smallest $x \in \Xcaltheta$ such that $I_{n,D}(x,\xast,\theta) \geq n\alpha/K$.

\item[---] For all $\theta \in \Ucal(\alpha,n,D)$ and $x \in \Xcaltheta$, let $\xast(\theta,x,n, \alpha,D)$ be the smallest $\xast \in [\Kast]$ such that $I_{n,D}(x,\xast,\theta) \geq n\alpha/K$.
\end{itemize}
In the following, we omit the dependency in $\alpha$ and $D$ and write $\xast(\theta,x,n)$ and $x(\theta,\xast,n)$.
\end{definition}

\begin{corollary}
\label{cor_TcalUcal_tendances_proches}
For all $\alpha \in (0,1]$ and $D > 0$, there exists $M(\alpha,D) > 0$ and $n_0 := 4K(d+1) / \alpha$ such that
\begin{equation*}
\forall n \geq n_0,
\quad \begin{cases}
\theta \in \Tcal(\alpha,n,D) \Rightarrow \forall \xast \in [\Kast], \  \| T^*_\xast - T^\theta_{x(\theta, \xast, n)} \|_{\infty, [0,n]} \leq M(\alpha,D), \\
\theta \in \Ucal(\alpha,n,D) \Rightarrow \forall x \in \Xcaltheta, \  \| T^*_{\xast(\theta,x,n)} - T^\theta_x \|_{\infty, [0,n]} \leq M(\alpha,D).
\end{cases}
\end{equation*}
\end{corollary}

\begin{figure}
\centering
\includegraphics[scale = 0.4]{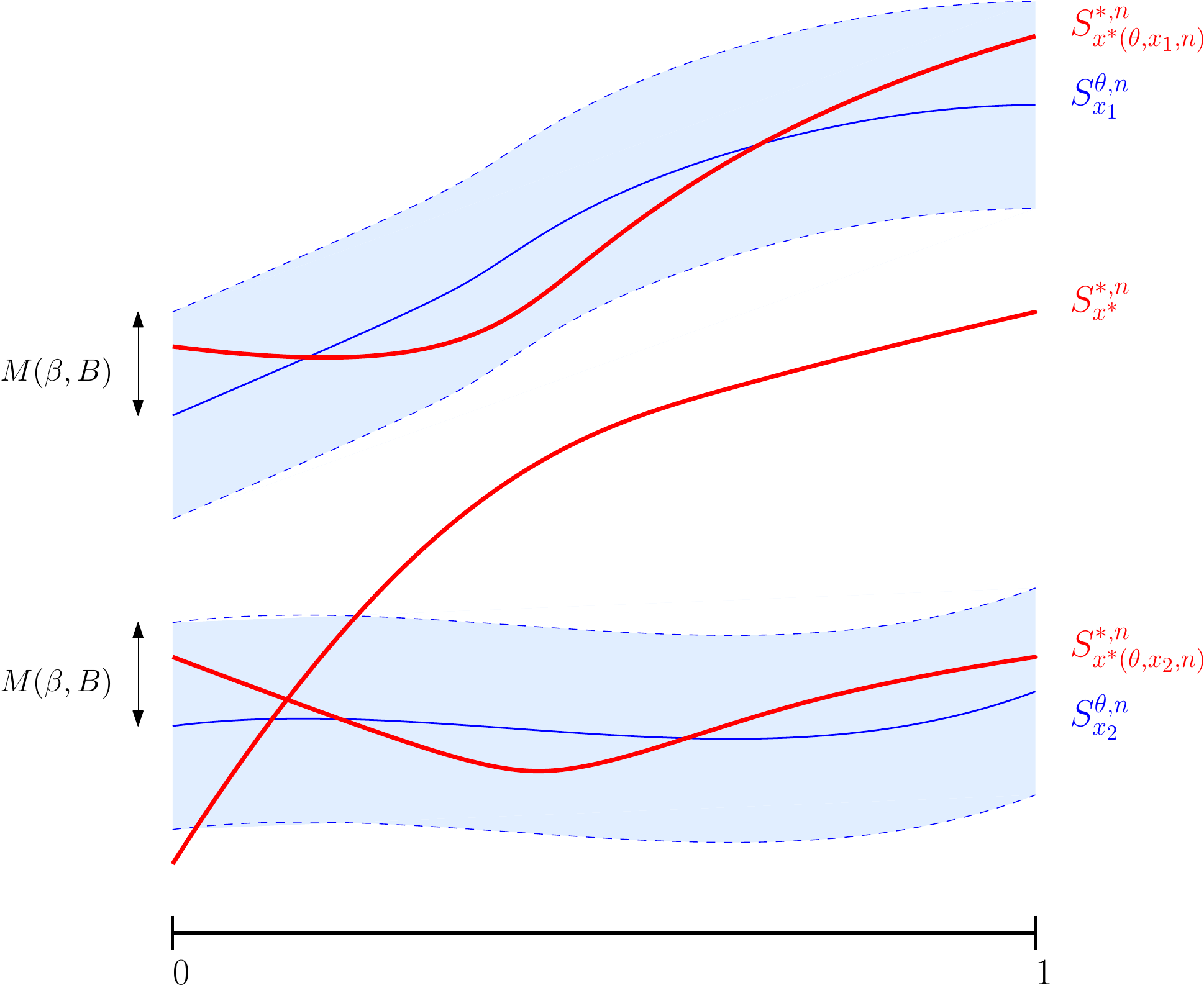}
\caption{Rescaled trends of a parameter in $\Ucal(\beta, n, B)$. Every parameter trend is at bounded distance of at least one true trend. However, some true trends may be far from all parameter trends.}
\label{dessin_Ucal}
\end{figure}

\begin{figure}
\centering
\includegraphics[scale = 0.4]{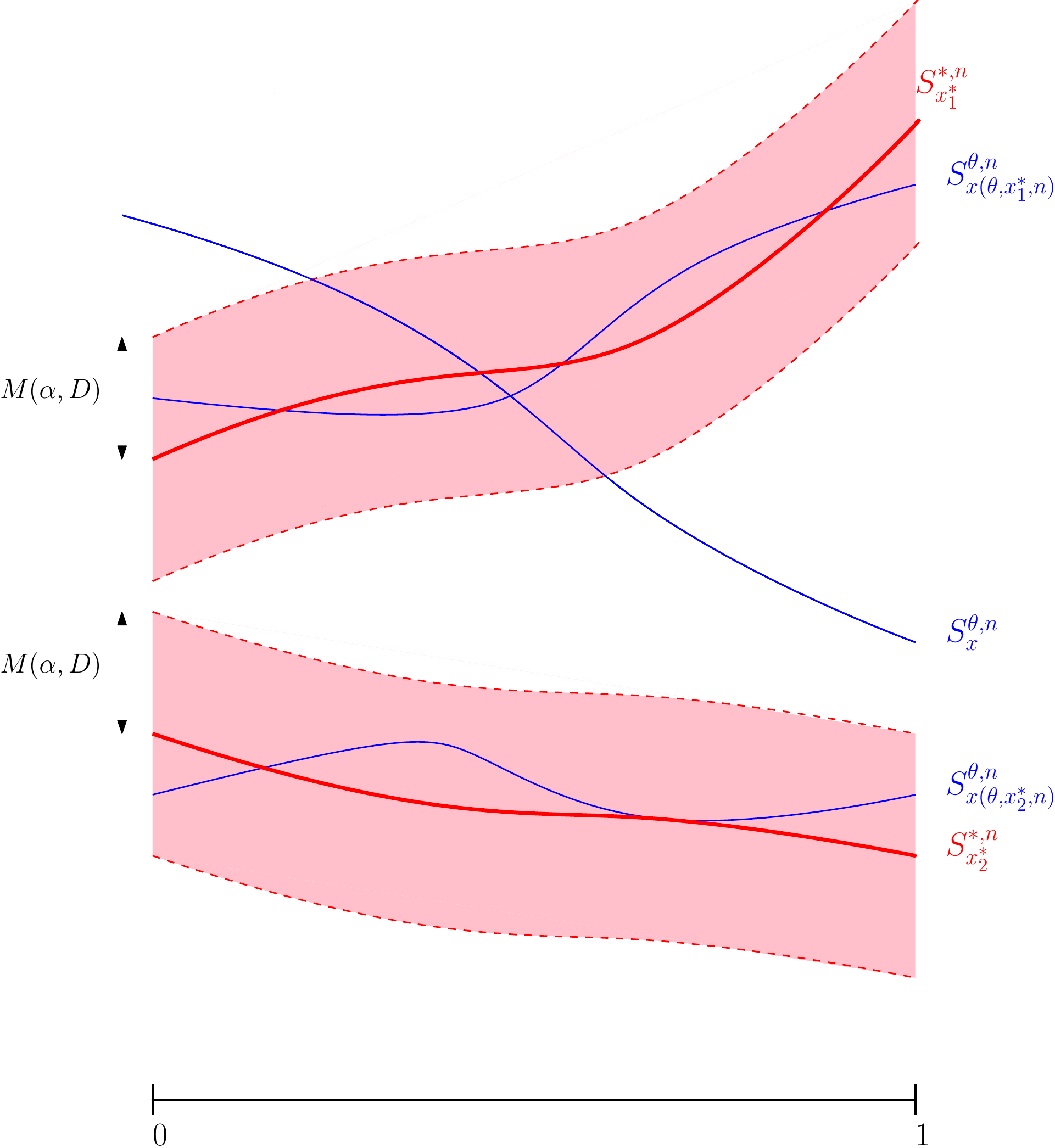}
\caption{Rescaled trends of a parameter in $\Tcal(\alpha, n, D)$. Every true trend is at bounded distance of at least one parameter trend. However, some parameter trends may be far from all true trends.}
\label{dessin_Tcal}
\end{figure}

\noindent
Hence, for all $\alpha, \beta \in (0,1]$, for all $M > 0$ and $D,B \geq M$ and for all $n$,
\begin{equation}
\label{eq_inclusion_ThetaOK_T_U}
\Theta^\text{OK}_n(M)
	\subset \Tcal(\alpha,n,D) \cap \Ucal(\beta,n,B)
	\subset \Theta^\text{OK}_n(M(\alpha,D) \vee M(\beta,B)).
\end{equation}

\noindent
Proposition~\ref{lemma_theta_dans_I_implique_tendances_proches} is a direct consequence of Arzelà–Ascoli's theorem and of the following result.
\begin{definition}
For all $\theta \in \Theta$, $n \in \Nbb^*$ and $x \in \Xcaltheta$, let
\begin{equation*}
S^{\theta, n}_x : u \in [0,1] \longmapsto T^\theta_x(nu)
\end{equation*}
be the trend $T^\theta_x$ whose time variable has been rescaled from $[0,n]$ to $[0,1]$. Likewise, define $S^{*,n}_\xast$ the rescaled true trend corresponding to the state $\xast$.
\end{definition}

\begin{theorem}
\label{th_compacite_S_pour_un_seul_I}
Let $\varepsilon \in (0,1]$ and $D > 0$. Then the set
\begin{equation*}
\Scal := \bigcup_{n \geq n_0} \  \bigcup_{\theta \in \Theta} \quad \bigcup_{(x,\xast) \in \Xcaltheta \times [\Kast] \text{ s.t. } |I_{n,D}(x,\xast,\theta)| \geq n\varepsilon} \{ S^{\theta,n}_{x} - S^{*,n}_\xast \}
\end{equation*}
is relatively compact in the set of continuous functions $\left( \Ccal^0([0,1]), \| \cdot\|_{\infty} \right)$.
\end{theorem}

\begin{remark}
This result, together with Arzelà–Ascoli's theorem, entails that $\Scal$ is uniformly equicontinuous in addition to being a bounded subset of $\Lbf^\infty([0,1])$. This will be used in Section~\ref{sec_integrated_likelihood}.
\end{remark}

\begin{proof}
Let $\varepsilon \in (0,1]$ and $D > 0$.
Let us first give another representation of the elements of $\Scal$.
\begin{lemma}
\label{lemme_carac_S}
For any $S \in \Scal$, there exists $(u_1^S, \dots, u_{d+1}^S) \in [0,1]^{d+1}$ such that
\begin{equation*}
\begin{cases}
\forall i \neq j, \quad |u_i^S - u_j^S| \geq \frac{\varepsilon}{4K(d+1)}, \\
\forall i, \quad |S(u_i^S)| \leq D.
\end{cases}
\end{equation*}
\end{lemma}

\begin{proof}
For $S \in \Scal$, let $n(S) \geq n_0$, $\theta$, $x \in \Xcaltheta$ and $\xast \in [\Kast]$ be such that $S = S^{\theta,n}_{x} - S^{*,n}_\xast$ and $|I_{n,D}(x,\xast,\theta)| \geq n\varepsilon$. Let us define $(u_1^S, \dots, u_{d+1}^S)$ iteratively.
Let $\Acal_0 = I_{n,D}(x,\xast,D)$ and, for all $i \geq 1$,
\begin{itemize}
\item $t_i^S\in\Acal_{i-1}$ ;
\item $\Acal_i = \Acal_{i-1} \setminus \bar{B}\left(t_i^S, \frac{\varepsilon}{4K(d+1)} n\right)$.
\end{itemize}

The closed ball $\bar{B}\left(t_i^S, \frac{\varepsilon}{4K(d+1)} n\right)$ contains at most $1 + \left\lfloor 2 \frac{\varepsilon}{4K(d+1)} n \right\rfloor \leq 3 \frac{\varepsilon}{4K(d+1)} n$ elements since $\frac{\varepsilon}{4K(d+1)} n \geq 1$. Thus, for all $i \geq 0$,
\begin{equation*}
| \Acal_i | \geq n \frac{\varepsilon}{K} \left( 1 - \frac{3}{4(d+1)} i \right).
\end{equation*}

In particular, $\Acal_i \neq \emptyset$ for all $i \in \{0, \dots, d+1\}$, which makes it possible to define $(t_i^S)_{1 \leq i \leq d+1}$. Taking $u_i^S = \frac{t_i^S}{n(S)}$ for all $i \in \{1, \dots, d+1\}$ concludes the proof.
\end{proof}

The next lemma is a straightforward consequence of the Lagrange form of the interpolation polynomial.

\begin{lemma}
\label{lemma_continuite_carac_S}
The mapping
\begin{equation}
\label{eq_lagrange_mapping}
\begin{aligned}
\left\{ (u_i)_i \in [0,1]^{d+1} \text{ s.t. } \inf_{i \neq j} |u_i - u_j| > 0 \right\} \times \Rbb^{d+1} &\longmapsto \left( \Ccal^0([0,1]), \| . \|_\infty\right)
\\
(u_1, \dots, u_{d+1}, s_1, \dots, s_{d+1}) \qquad \quad &\longmapsto \qquad P_{u,s}
\end{aligned}
\end{equation}
is continuous, where $P_{u,s}$ is the only polynomial with degree at most $d$ such that $P(u_i) = s_i$ for all $i \in \{1, \dots, d+1\}$.
\end{lemma}

To conclude, note that $\Scal$ is a subset of the image of the compact set $\{ (u_i)_i \in [0,1]^{d+1} \text{ s.t. } \inf_{i \neq j} |u_i - u_j| \geq \varepsilon / (4K(d+1)) \} \times [-D,D]^{d+1}$ by the mapping~\eqref{eq_lagrange_mapping}.
\end{proof}

\subsection{The MLE is in \texorpdfstring{$\Tcal(\alpha,n,D)$}{T(a,n,D)}}
\label{sec_MLE_pas_dans_Tcal}

The key idea of this section is that if one of the true trends is far from all parameter trends, then the observations coming from this true trend will significantly reduce the likelihood.

\bigskip

Let $\alpha \in (0,1)$, $n\geq 1$, $D>0$ and $\theta \notin \Tcal(\alpha,n,D)$, then
\begin{equation*}
\left|
		\bigcup_{\xast \in [\Kast]} \bigcap_{x \in \Xcaltheta} I_{n,D}(x,\xast,\theta)^\complement
	\right| > n(1-\alpha)
\end{equation*}
with the notations of Section~\ref{sec_compacity}. In particular, there exists $\xfarT(\theta) \in [\Kast]$ such that
\begin{equation}
\label{eq_minoration_card_Ifar}
\left|
		\bigcap_{x \in \Xcaltheta} I_{n,D}(x,\xfarT(\theta),\theta)^\complement
	\right| \geq n \frac{1-\alpha}{\Kast}.
\end{equation}

Write $I^\text{far}_n(\theta) := \bigcap_{x \in \Xcaltheta} I_{n,D}(x,\xfarT(\theta),\theta)^\complement$, then 
\begin{align}
\nonumber
\frac{1}{n} \ell_n(\theta)
	&= \frac{1}{n} \sum_{t=1}^n \log p^\theta(Y_t | Y_1^{t-1}) \\
	\label{eq_maj_logV_par_g}
	&\leq \log g(0) + \frac{1}{n} \sum_{t \in I^\text{far}_n(\theta)} 
		\one_{X_t = \xfarT(\theta)} \log \frac{g( \{D - Z^{\max}_t\}_+)}{g(0)}.
\end{align}

We used the fact that under Assumption \textbf{(Amax)},
\begin{align*}
p^\theta(Y_t | Y_1^{t-1})
	&= \sum_{x \in \Xcaltheta} p^\theta(X_t = x | Y_1^{t-1})
		\gamma^\theta_x(Y_t - T^\theta_x(t)) \\
	&\leq \sum_{\xast \in [\Kast]} \one_{X_t = \xast} \sup_{x \in \Xcaltheta} \gamma^\theta_x(Z_t + T^*_\xast(t) - T^\theta_x(t)) \\
	&\leq \sum_{\xast \in [\Kast]} \one_{X_t = \xast} \sup_{x \in \Xcaltheta} g(\{|T^*_\xast(t) - T^\theta_x(t)| - Z^{\max}_t\}_+ ) \\
	&\leq \sum_{\xast \in [\Kast]} \one_{X_t = \xast} g(\inf_{x \in \Xcaltheta} \{|T^*_\xast(t) - T^\theta_x(t)| - Z^{\max}_t\}_+ ).
\end{align*}

\begin{lemma}
For all $\theta \in \Theta$, $D>0$ and $\xast \in [\Kast]$, the set
\begin{equation}
\label{eq_Ifar_version_C0}
\bigcap_{x \in \Xcaltheta} 
\{ t \in [0,n] : |T^*_\xast(t) - T^\theta_x(t)| > D \}
\end{equation}
has at most $A := (d+1)^K$ connected components.
\end{lemma}
Note that $I^\text{far}_n(\theta) = \eqref{eq_Ifar_version_C0} \cap \{1, \dots, n\}$.

\begin{proof}
The functions $(t \longmapsto T^*_\xast(t) - T^\theta_x(t))_{x \in \Xcaltheta}$ are polynomials whose degree is at most $d$. Their derivatives vanish at most $d-1$ times, and the set of times $t$ where they are larger than $D$ in absolute value is a union of segments containing either, a zero of their derivative, $+\infty$ or $-\infty$. Hence there are at most $d+1$ such segments. Thus, $\tilde{I}_{n,D,\xast}(\theta)$ is an intersection of at most $K$ sets, each of them having at most $d+1$ connected components. Therefore, one may take $A=(d+1)^K$.
\end{proof}

\begin{definition}
For all $n \in \Nbb^*$, $D > 0$, $\xast \in [\Kast]$ and $\theta \in \Theta$, write $J(n,D,\xast,\theta)$ the largest connected component of~\eqref{eq_Ifar_version_C0}. In case of tie, choose the first one for the usual order in $\Rbb$.
\end{definition}

Thus, by the pigeonhole principle and equation~\eqref{eq_minoration_card_Ifar},
\begin{equation*}
| J(n,D,\xfarT(\theta),\theta) \cap \{1, \dots, n\} | \geq n\frac{1-\alpha}{A\Kast}.
\end{equation*}

Write $J^\text{far}_n(\theta) := J(n,D,\xfarT(\theta),\theta) \cap \{1, \dots, n\}$, then using equation~\eqref{eq_maj_logV_par_g}:
\begin{align}
\label{eq_estimpasdansTcal_majorationln}
\frac{1}{n} \ell_n(\theta)
	&\leq \log g(0) + \frac{1}{n} \sum_{t \in J^\text{far}_n(\theta)} 
		\one_{X_t = \xfarT(\theta)} \log \frac{g( \{D - Z^{\max}_t\}_+)}{g(0)}.
\end{align}

\begin{lemma}
\label{lem_hoeffding}
Let $\delta > 0$ and assume \textbf{(Aerg)}. Then, almost surely,
\begin{equation*}
	\liminf_{n \rightarrow \infty} 
	\underset{|S| \geq \delta n}{\underset{S \text{ segment}}{\inf_{S \subset \{1, \dots, n\}}}}
	\inf_{\xast \in [\Kast]}
		\frac{1}{n} \sum_{t \in S} \one_{X_t = \xast}
	\geq \frac{\delta \sigma_-}{4}.
\end{equation*}

By ``segment", we mean a set of the form $[a,b] \cap \Zbb$ for some $(a,b) \in \Rbb^2$.
\end{lemma}

\begin{proof}
The idea is to split $\{1, \dots, n\}$ into segments of size $\frac{\delta}{2} n$ and to control the infimum of the empirical mean over each segment. Each segment of size larger than $\delta n$ contains at least one of those segments. The proof is detailled in Section~\ref{sec_proof_infSegments}.
\end{proof}

Apply Lemma \ref{lem_hoeffding} to $S = J^\text{far}_n(\theta)$: almost surely,
\begin{equation}
\label{eq_minoration_proportionDansJ}
\liminf_{n \rightarrow \infty} 
	\inf_{\theta \notin \Tcal(\alpha,n,D)}
		\frac{1}{n} \sum_{t \in J^\text{far}_n(\theta)} \one_{X_t = \xfarT(\theta)}
	\geq \frac{(1-\alpha) \sigma_-}{4 A \Kast}.
\end{equation}

\begin{lemma}
\label{lemma_top_quantile}
Let $\delta \in (0,1)$, $(U_t)_{t \geq 1}$ a sequence of i.i.d. non-positive integrable random variables and $(\delta_n)_{n \geq 1}$ a non-decreasing sequence of $[0,1]$-valued random variables such that $\liminf_{n \rightarrow \infty}\delta_n \geq \delta$ a.s. For all $\beta \in [0,1]$, let us denote by $q_U(\beta)$ the $\beta$-quantile of $U_1$, i.e.
\begin{equation*}
q_U(\beta) = \inf \{ u \text{ s.t. } \Pbb(U_1 \leq u) \geq \beta \}.
\end{equation*}

Then, almost surely,
\begin{equation*}
\limsup_{n \rightarrow \infty} 
	\underset{|S| \geq \delta_n n}{\sup_{S \subset \{1, \dots, n\}}}
		\frac{1}{n} \sum_{t \in S} U_t
	\leq \Ebb[U_1 \one_{U_1 > q_U(1-\delta)}].
\end{equation*}

Equivalently, if $(V_t)_{t \geq 1}$ is a sequence of non-negative i.i.d. integrable random variables and $(\delta_n)_{n \geq 1}$ a non-increasing sequence of $[0,1]$-valued random variables such that $\limsup_{n \rightarrow \infty} \delta_n \leq \delta$ a.s., almost surely
\begin{equation*}
\limsup_{n \rightarrow \infty} 
	\underset{|S| \leq \delta_n n}{\sup_{S \subset \{1, \dots, n\}}}
		\frac{1}{n} \sum_{t \in S} V_t
	\leq \Ebb[V_1 \one_{V_1 \geq q_V(1-\delta)}].
\end{equation*}
\end{lemma}

\begin{remark}
The supremum is taken over all subsets $S$, not only segments.
\end{remark}

\begin{proof}
Proof in Section~\ref{sec_proof_cvgFractionSomme}.
\end{proof}

For all $t \geq 1$ and $D > 0$, let $U_t^D = \log \frac{g( \{ D - Z^{\max}_t\}_+)}{g(0)}$. $U_t^D$ is non-positive by definition. Then, taking
\begin{align*}
\begin{cases}
\delta = \frac{(1-\alpha) \sigma_-}{4 A \Kast}, \\
\displaystyle \delta_n = \inf_{m \geq n} \inf_{\theta \notin \Tcal(\alpha,m,D)}
		\frac{1}{m} \sum_{t \in J^\text{far}_m(\theta)} \one_{X_t = \xfarT(\theta)},
\end{cases}
\end{align*}
one has $\liminf_{n \rightarrow \infty} \delta_n \geq \delta$ by equation~\eqref{eq_minoration_proportionDansJ}. Therefore, Lemma~\ref{lemma_top_quantile} combined with equation~\eqref{eq_estimpasdansTcal_majorationln} implies that almost surely,
\begin{align*}
\limsup_{n \rightarrow \infty} \sup_{\theta \notin \Tcal(\alpha, n, D)} \frac{1}{n} \ell_n(\theta)
	&\leq \log g(0) + \Ebb[U_1^D \one_{U_1^D \geq q_{U^D}(1 - \frac{(1-\alpha) \sigma_-}{4 A \Kast})}].
\end{align*}

Note that $U_t^D = f^D(Z^{\max}_t) \leq 0$ where $f^D : z \in \Rbb_+ \longmapsto \log \frac{g( \{ D - z\}_+)}{g(0)}$. $f^D$ is non-decreasing, so
$q_{f^D(Z^{\max})}(1 - \delta) \leq f^D(q_{Z^{\max}}(1 - \delta))$ for all $\delta > 0$.
%
Thus, for all $z \geq 0$,
\begin{align*}
\one_{z \geq q_{Z^{\max}}(1-\delta)}
	\leq \one_{f^D(z) \geq f^D(q_{Z^{\max}}(1-\delta))}
	\leq \one_{f^D(z) \geq q_{U^D}(1-\delta)},
\end{align*}
and since $U_1^D \leq 0$,
\begin{align*}
\Ebb[U_1^D \one_{U_1^D \geq q_{U^D}(1 - \frac{(1-\alpha) \sigma_-}{4 A \Kast})}]
	\leq \Ebb[U_1^D \one_{Z^{\max}_1 \geq q_{Z^{\max}}(1-\frac{(1-\alpha) \sigma_-}{4 A \Kast})}].
\end{align*}

Then, for all $\delta > 0$, the monotone convergence theorem applied to the right-hand side entails
\begin{equation*}
\Ebb[U_1^D \one_{U_1^D \geq q_{U^D}(1 - \delta)}] \underset{D \rightarrow + \infty}{\longrightarrow} - \infty.
\end{equation*}

Thus, under the assumptions of Corollary~\ref{cor_Arate}, there exists $D(\alpha) < \infty$ such that
\begin{equation*}
\limsup_{n \rightarrow \infty} \sup_{\theta \notin \Tcal(\alpha, n, D(\alpha))} \frac{1}{n} \ell_n(\theta)
	\leq \ell(\theta^*) - 1,
\end{equation*}
so that almost surely, for $n$ large enough,
\begin{equation*}
\hat{\theta}_n \in \Tcal(\alpha,n,D(\alpha)).
\end{equation*}

\subsection{The MLE is in \texorpdfstring{$\Ucal(\beta,n,B)$}{U(b,n,B)}}
\label{sec_MLE_pas_dans_Ucal}

Let $\alpha,\beta \in (0,1)$, $n\geq \frac{4K(d+1)}{1-\alpha}$, $D>0$, $B > M(\alpha,D)$ and $\theta \in \Tcal(\alpha,n,D) \setminus \Ucal(\beta,n,B)$. Since $\theta \notin \Ucal(\beta,n,B)$, one of its trends is far from all true trends. We show that removing the state of this trend increases the likelihood of the observations. An interpretation is that because of \textbf{(Aerg)}, a proportion at least $\sigma_-$ of the observations is supposed to come from the (superfluous) state, but no observations appears in the vicinity of its trend, which penalizes the likelihood.

By definition of $\Ucal(\beta,n,B)$, there exists $\xfarU(\theta) \in \Xcaltheta$ such that
\begin{equation}
\label{eq_minoration_card_IU}
\left|
		\bigcap_{\xast \in [\Kast]} I_{n,B}(\xfarU(\theta),\xast,\theta)^\complement
	\right| > n (1-\beta).
\end{equation}

Write $I^U_n(\theta) := \bigcap_{\xast \in [\Kast]} I_{n,B}(\xfarU(\theta),\xast,\theta)^\complement$ and let $\theta^U$ be the parameter of the HMM with $K^\theta - 1$ hidden states, with hidden state space $\Xcal^{\theta^U} = \Xcaltheta \setminus \{ \xfarU(\theta) \}$ and other parameters defined by
\begin{equation*}
\begin{cases}
\forall x \in \Xcal^{\theta^U}, \quad \pi^{\theta^U}(x) = \frac{\pi^\theta(x)}{1 - \pi^\theta(\xfarU(\theta))}, \\
\forall x,x' \in \Xcal^{\theta^U}, \quad Q^{\theta^U}(x,x') = \frac{Q^\theta(x,x')}{1 - Q^\theta(x,\xfarU(\theta))}, \\
\forall x \in \Xcal^{\theta^U}, \quad \gamma^{\theta^U}_x = \gamma^\theta_x, \\
\forall x \in \Xcal^{\theta^U}, \quad T^{\theta^U}_x = T^\theta_x.
\end{cases}
\end{equation*}

By construction of $\Theta$, $\theta^U \in \Theta$.
Note that for all $x,x' \in \Xcal^{\theta^U}$, 
\begin{align*}
\pi^{\theta^U}(x)
&= \Pbb^\theta(X_1 = x \, | \, X_1 \neq \xfarU(\theta)) \\
	&= \Pbb^\theta(X_1 = x \, | \, \forall t \geq 1, X_t \neq \xfarU(\theta)) 
\end{align*}
and
\begin{align*}
\forall s \geq 1, \quad Q^{\theta^U}(x,x')
	&= \Pbb^\theta(X_{s+1} = x' \, | \, X_s = x, X_{s+1} \neq \xfarU(\theta)) \\
	&= \Pbb^\theta(X_{s+1} = x' \, | \, X_s = x, \forall t \geq 1, X_t \neq \xfarU(\theta)),
\end{align*}
so that
\begin{equation*}
\forall \in \sigma(Y_t \, | \, t \geq 1), \quad \Pbb^{\theta^U}(A) = \Pbb^\theta(A \, | \, \forall t \geq 1, X_t \neq \xfarU(\theta)). 
\end{equation*}

Then
\begin{align*}
\frac{1}{n}\ell_n(\theta^U) = \frac{1}{n}\sum_{t=1}^n \log p^\theta(Y_t \, | \, Y_1^{t-1}, \xfarU(\theta) \notin X_1^t),
\end{align*}
with the abuse of notation $x \in X_1^t \Leftrightarrow (\exists s \in \{1, \dots, t\} \  X_s = x)$.
By Assumption \textbf{(Aerg)},
\begin{align*}
p^\theta(Y_t \, | \, Y_1^{t-1})
	&= p^\theta(Y_t \, | \, X_t=\xfarU(\theta)) p^\theta(X_t=\xfarU(\theta) \, | \, Y_1^{t-1}) \\
		&\quad + p^\theta(Y_t \, | \, X_t\neq \xfarU(\theta),Y_1^{t-1}) p^\theta(X_t\neq \xfarU(\theta) \, | \, Y_1^{t-1}) \\
	&\leq (1-\sigma_-) p^\theta(Y_t \, | \, X_t = \xfarU(\theta))
		+ (1-\sigma_-) p^\theta(Y_t \, | \, X_t\neq \xfarU(\theta), Y_1^{t-1}),
\end{align*}
hence
\begin{multline*}
\frac{1}{n} \ell_n(\theta)
		- \underbrace{\frac{1}{n}\sum_{t=1}^n\log p^\theta(Y_t\, | \, X_t\neq \xfarU(\theta), Y_1^{t-1})}_{(i)}
	\leq \log(1-\sigma_-) \\
    	+ \underbrace{\frac{1}{n}\sum_{t=1}^n\log\left(1 + \frac{p^\theta(Y_t\, | \, X_t = \xfarU(\theta))}{p^\theta(Y_t\, | \, X_t\neq \xfarU(\theta), Y_1^{t-1})}\right)}_{(ii)}.
\end{multline*}

The next steps are:
\begin{itemize}
\item Prove that $(i)$ is close to $\frac{1}{n}\ell_n(\theta^U)$ for large enough $n$.
\item Prove that $(ii)$ goes to zero uniformly in $\theta \in \Tcal(\alpha,n,D) \setminus \Ucal(\beta,n,B)$.
\end{itemize}

\paragraph*{First step: controlling $(i)$}

We shall prove that for an adequate choice of $\beta$ and $B$, almost surely
\begin{equation*}
\limsup_{n \rightarrow \infty} \sup_{\theta \in \Tcal(\alpha,n,D) \setminus \Ucal(\beta,n,B)} \left| (i) - \frac{1}{n} \ell_n(\theta^U) \right|
	\leq \frac{-\log(1 - \sigma_-)}{3}.
\end{equation*}

The following forgetting property allows to control what happens in the distant past.

\begin{lemma}
\label{lem_oubli}
For all $t \geq 1$, $\theta\in\Theta$, for any probability measures $\mu$ and $\nu$ on $\Xcaltheta$, for all $x \in \Xcaltheta$ and $Y_0^t$,
\begin{align*}
| \log p^\theta(Y_t \, | \, Y_0^{t-1}, X_t \neq x, X_0 \sim \mu)
	- \log p^\theta(Y_t \, | \, Y_0^{t-1}, X_t \neq x, X_0 \sim \nu) |
	\leq C \rho^t
\end{align*}
where $\rho = 1 - \frac{\sigma_-}{1 - \sigma_-}$ and $C = \frac{2}{\rho (1 - \rho)^3}$.
\end{lemma}

\begin{proof}
Proof in Section \ref{sec_proof_oubli}.
\end{proof}

Let $a \in \Nbb^*$. It follows from Lemma \ref{lem_oubli} that for all $t$, almost surely,
\begin{align}
\label{control_i_part1}
\left|\log p^\theta(Y_t \, | \, Y_1^{t-1}, X_t\neq \xfarU(\theta)) - \log p^\theta(Y_t \, | \, Y_1^{t-1}, X_t \neq \xfarU(\theta), \xfarU(\theta) \notin X_1^{t-a})\right| \leq C \rho^a.
\end{align}

It remains to add $X_{t-a+1}^{t-1}$ to the conditioning. This is the goal of the following lemma.

\begin{lemma}
\label{lemma_conditionnementEtatsPasses}
Assume \textbf{(Amax)} and \textbf{(Amin)}. Then for all $a \in \Nbb^*$,
\begin{align*}
\sup_{\theta \in \Tcal(\alpha,n,D) \setminus \Ucal(\beta,n,B)} & 
	\Bigg| \frac{1}{n} \sum_{t=1}^n \log p^\theta(Y_t | Y_1^{t-1}, X_t \neq \xfarU(\theta), \xfarU(\theta) \notin X_1^{t-a}) \\
	&\hspace{5cm}- \frac{1}{n} \sum_{t=1}^n \log p^\theta(Y_t | Y_1^{t-1}, \xfarU(\theta) \notin X_1^t) \Bigg| \\
	&\leq \frac{2 a K^2}{\sigma_-^3} \left( \beta + \frac{1}{n} \sum_{i=1}^{n-1}
		\left( 1 \wedge \frac{g(\{B - Z^{\max}_i\}_+)}{m(Z^{\max}_i + M(\alpha,D))} \right) \right) \\
	&=: \frac{2 a K^2}{\sigma_-^3} \left( \beta + \frac{1}{n} \sum_{i=1}^{n-1} h_U(B, Z^{\max}_i) \right).
\end{align*}
\end{lemma}

\begin{proof}[Overview of the proof]
First, show that for all $\theta \in \Tcal(\alpha,n,D) \setminus \Ucal(\beta,n,B)$, $t \leq n$, $a \in \Nbb^*$ and $y_1^t \in \Ycal^t$,
\begin{align}
\nonumber
\Big| \log p^\theta & (y_t | y_1^{t-1}, X_t \neq \xfarU(\theta), \xfarU(\theta) \notin X_1^{t-a})
	- \log p^\theta(y_t | y_1^{t-1}, \xfarU(\theta) \notin X_1^t) \Big| \\
	\label{eq_proof_recentStates_erreurBlock}
	&\leq \frac{2}{\sigma_-} p^\theta \left(\xfarU(\theta) \in X_{(t-a+1) \vee 1}^{t-1} | y_1^{t-1}, X_t \neq \xfarU(\theta), \xfarU(\theta) \notin X_1^{t-a} \right) \\
	\label{eq_proof_recentStates_densites}
	&\leq \frac{2 K^2}{\sigma_-^3} \sum_{i=(t-a+1) \vee 1}^{t-1}
		\frac{\gamma^\theta_{\xfarU(\theta)}(y_i - T^\theta_{\xfarU(\theta)}(i))}{\displaystyle \sum_{x \in \Xcaltheta} \gamma^\theta_{x}(y_i - T^\theta_x(i))}.
\end{align}

Then, under \textbf{(Amax)} and \textbf{(Amin)}, for all $\theta \in \Tcal(\alpha,n,D) \setminus \Ucal(\beta,n,B)$ and $i \leq n$
\begin{equation}
\label{eq_proof_recentStates_utilisationEncadrements}
\frac{\gamma^\theta_{\xfarU(\theta)}(Y_i - T^\theta_{\xfarU(\theta)}(i))}{\displaystyle \sum_{x \in \Xcaltheta} \gamma^\theta_{x}(Y_i - T^\theta_x(i))} \leq
\begin{cases}
\displaystyle 1 \wedge \frac{g(\{B - Z^{\max}_i\}_+)}{m(Z^{\max}_i + M(\alpha,D))} =: h_U(B, Z^{\max}_i) \quad &\text{if } i \in I^U_n(\theta), \\
1 &\text{if } i \notin I^U_n(\theta).
\end{cases}
\end{equation}
so that
\begin{align*}
\frac{1}{n} \sum_{i=1}^n \frac{\gamma^\theta_{\xfarU(\theta)}(Y_i - T^\theta_{\xfarU(\theta)}(i))}{\displaystyle \sum_{x \in \Xcaltheta} \gamma^\theta_x(Y_i - T^\theta_x(i))}
	&\leq \frac{1}{n} \sum_{i \notin I^U_n(\theta)} 1
		+ \frac{1}{n} \sum_{i \in I^U_n(\theta)} h_U(B, Z^{\max}_i) \\
	&\leq \beta
		+ \frac{1}{n} \sum_{i=1}^n h_U(B, Z^{\max}_i)
\end{align*}
since $| I^U_n(\theta) | > n (1-\beta)$ by Equation~\eqref{eq_minoration_card_IU} and $h_U(b,z) \geq 0$ for all $b,z \geq 0$. The lemma follows by summing equation~\eqref{eq_proof_recentStates_densites} over $t$.
The details of the proof can be found in Section~\ref{sec_proof_conditionnementEtatsPasses}.
\end{proof}

Thus, equation~\eqref{control_i_part1} and Lemma~\ref{lemma_conditionnementEtatsPasses} imply
\begin{align*}
\limsup_{n \rightarrow \infty} \sup_{\theta \in \Tcal(\alpha,n,D) \setminus \Ucal(\beta,n,B)} \left| (i) - \frac{1}{n} \ell_n(\theta^U) \right|
	\leq C \rho^a + \frac{2 a K^2}{\sigma_-^3} \left( \beta
		+ \Ebb^*\left[ h_U(B,Z^{\max}_i) \right] \right).
\end{align*}

Now choose $a$ large enough that
\begin{align*}
C \rho^a \leq \frac{- \log(1 - \sigma_-)}{9},
\end{align*}
then $\beta$ such that
\begin{align*}
\frac{2 a K^2}{\sigma_-^3} \beta \leq \frac{- \log(1 - \sigma_-)}{9}.
\end{align*}

Finally, note that $0 \leq h_U(b,z) \leq 1$ for all $b,z \geq 0$ and that $h_U(b,z) \longrightarrow 0$ when $b \longrightarrow \infty$ for all $z$, so that by the dominated convergence theorem, there exists $B$ such that
\begin{align*}
\frac{2 a K^2}{\sigma_-^3} \Ebb^*\left[ h_U(B,Z^{\max}_i) \right] \leq \frac{- \log(1 - \sigma_-)}{9},
\end{align*}
which ensures that for all $\alpha \in (0,1)$ and $D>0$, there exists $\beta \in (0,1)$ and $B>0$ such that
\begin{align*}
\limsup_{n \rightarrow \infty} \sup_{\theta \in \Tcal(\alpha,n,D) \setminus \Ucal(\beta,n,B)} \left| (i) - \frac{1}{n} \ell_n(\theta^U) \right|
	\leq \frac{- \log(1 - \sigma_-)}{3}.
\end{align*}

This concludes the proof of the first step.

\paragraph*{Second step: controlling $(ii)$}

\begin{lemma}
\label{lemma_conditionnementEtatPresent}
Assume \textbf{(Amax)} and \textbf{(Amin)}. Then
\begin{align*}
\nonumber &\sup_{\theta \in \Tcal(\alpha,n,D) \setminus \Ucal(\beta,n,B)} (ii) \\
	&\qquad \leq \frac{1}{n}\sum_{t=1}^n \log\left(1 + \frac{ g(\{B - Z^{\max}_t\}_+) }{ \sigma_- m(Z^{\max}_t + M(\alpha,D)) } \right) \\
		&\qquad \qquad + \frac{1}{n}\sum_{t=1}^n \one_{t \notin I^U_n(\theta)} \log\left(\frac{ g(0) + \sigma_- m(Z^{\max}_t + M(\alpha,D)) }{ g(\{B - Z^{\max}_t\}_+) + \sigma_- m(Z^{\max}_t + M(\alpha,D)) } \right) \\
\nonumber &\qquad =: \frac{1}{n}\sum_{t=1}^n h_U'(B,Z^{\max}_t)
		+ \frac{1}{n}\sum_{t \notin I^U_n(\theta)} V_t^B.
\end{align*}
\end{lemma}

\begin{proof}
We show that
\begin{equation*}
\frac{p^\theta(Y_t\, | \, X_t = \xfarU(\theta))}{p^\theta(Y_t\, | \, X_t\neq \xfarU(\theta), Y_1^{t-1})}
	\leq
\begin{cases}
\displaystyle \frac{ g(\{B - Z^{\max}_t\}_+) }{ \sigma_- m(Z^{\max}_t + M(\alpha,D)) } \quad \text{if } t \in I^U_n(\theta), \\[10pt]
\displaystyle \frac{ g(0) }{ \sigma_- m(Z^{\max}_t + M(\alpha,D)) } \quad \text{if } t \notin I^U_n(\theta).
\end{cases}
\end{equation*}

The lemma follows by summing over $t$. The details of the proof can be found in Section~\ref{sec_proof_conditionnementEtatPresent}.
\end{proof}

Note that under Assumption \textbf{(Aint)},
\begin{equation*}
\Ebb^* [- \log m(Z^{\max}_t + M(\alpha,D)) ] < \infty.
\end{equation*}

Hence,
\begin{align*}
\Ebb^* | h'_U(0,Z^{\max}_t) |
	&\leq \Ebb^*\left[ \left\{\log\left(\sigma_- m\left(Z^{\max}_t + M(\alpha,D)\right) + g(0) \right) \right\}_+ \right] \\
		&\quad + \Ebb^*[- \log (\sigma_- m(Z^{\max}_t + M(\alpha,D)) )] \\
	&\leq \log 2 + |\log g(0)| + \Ebb^*[- \log m(Z^{\max}_t + M(\alpha,D))] - \log \sigma_- \\
	&< \infty.
\end{align*}

Thus, since $b \longmapsto h'_U(b,z)$ is nonincreasing and converges to zero when $b \longrightarrow \infty$ for all $z$, the dominated convergence theorem together with the law of large numbers imply that there exists $B$ such that
\begin{equation*}
\limsup_{n \rightarrow \infty} \frac{1}{n}\sum_{t=1}^n h_U'(B,Z^{\max}_t)
	\leq \frac{- \log(1 - \sigma_-)}{6}.
\end{equation*}

Then, apply Lemma \ref{lemma_top_quantile} to the i.i.d. non-negative random variables $(V_t^B)_{t\geq 1}$ using the fact that $| I^U_n(\theta) | > n(1-\beta)$, which yields
\begin{align*}
\limsup_{n \rightarrow \infty} \sup_{\theta \in \Tcal(\alpha,n,D) \setminus \Ucal(\beta,n,B)} \frac{1}{n} \sum_{t \notin I^U_n(\theta)} V_t^B
	&\leq \Ebb^*[V_1^B \one_{V_1^B \geq q_{V^B}(1-\beta)}].
\end{align*}

Note that 
\begin{align*}
\Ebb^* V_1^B \leq \log ((1 + \sigma_-) g(0)) - \log \sigma_- + \Ebb^*[ - \log m(Z^{\max}_t + M(\alpha,D))],
\end{align*}
which is finite thanks to \textbf{(Aint)}. Thus,
\begin{equation*}
\Ebb^*[V_1^B \one_{V_1^B \geq q_{V^B}(1-\beta)}] \underset{\beta \rightarrow 0}{\longrightarrow} 0,
\end{equation*}
so that there exists $\beta$ such that
\begin{align*}
\limsup_{n \rightarrow \infty} \sup_{\theta \in \Tcal(\alpha,n,D) \setminus \Ucal(\beta,n,B)} \frac{1}{n} \sum_{t \notin I^U_n(\theta)} V_t^B
	&\leq \frac{- \log(1 - \sigma_-)}{6}.
\end{align*}

Hence, we proved that there exists $\beta(\alpha,D)\in (0,1)$ and $B(\alpha,D) > 0$ such that
\begin{equation*}
\limsup_{n \rightarrow \infty} \sup_{\theta \in \Tcal(\alpha,n,D) \setminus \Ucal(\beta,n,B)} (ii)
	\leq \frac{- \log(1 - \sigma_-)}{3},
\end{equation*}
which ends the second step.

Putting together the results of the two steps, one gets that for all $\alpha \in (0,1)$ and $D > 0$, there exists $\beta \in (0,1)$ and $B>0$ such that almost surely,
\begin{align*}
\limsup_{n \rightarrow \infty} \left(
	\sup_{\theta \in \Tcal(\alpha,n,D) \setminus \Ucal(\beta,n,B)} \frac{1}{n} \ell_n(\theta)
	- \sup_{\theta \in \Theta} \frac{1}{n} \ell_n(\theta) \right)
		\leq \frac{\log(1 - \sigma_-)}{3} < 0,
\end{align*}
so that for $n$ large enough, $\hat{\theta}_n \notin (\Tcal(\alpha,n,D) \setminus \Ucal(\beta,n,B))$.

Together with Section~\ref{sec_MLE_pas_dans_Tcal}, this implies that for all $\alpha \in (0,1)$, there exists $\beta \in (0,1)$ and $D,B > 0$ such that $\hat{\theta}_n \in \Tcal(\alpha,n,D) \cap \Ucal(\beta,n,B)$ for $n$ large enough, which entails Theorem~\ref{th_localization} by equation~\eqref{eq_inclusion_ThetaOK_T_U}.

\subsection{Proofs}

\subsubsection{Proof of Lemma~\ref{lem_oubli}}\label{sec_proof_oubli}

We shall use the inequality

\begin{align*}
|\log C_\mu - \log C_\nu|\leq \frac{|C_\mu-C_\nu|}{C_\mu \wedge C_\nu}
\end{align*}
with $C_\mu = p^\theta(Y_t \mid Y_0^{t-1}, X_t \neq x, X_0 \sim \mu)$ and $C_\nu = p^\theta(Y_t \mid Y_0^{t-1}, X_t \neq x, X_0 \sim \nu)$.

\begin{align*}
|C_\mu-C_\nu|
	&= \left| \frac{p^\theta(Y_t, X_t \neq x \mid Y_0^{t-1}, X_0 \sim \mu)}{p^\theta(X_t \neq x \mid Y_0^{t-1}, X_0 \sim \mu)}
    	- \frac{p^\theta(Y_t, X_t \neq x \mid Y_0^{t-1}, X_0 \sim \nu)}{p^\theta(X_t \neq x \mid Y_0^{t-1}, X_0 \sim \nu)} \right| \\
    &=: \left| \frac{B_\mu}{A_\mu}
    	- \frac{B_\nu}{A_\nu} \right|.
\end{align*}

One has

\begin{align*}
B_\mu &= \sum_{x'\neq x}p^\theta(Y_t\mid X_t = x')p^\theta(X_t = x'\mid Y_0^{t-1},X_0\sim\mu)\\
	  &= \sum_{x'\neq x}p^\theta(Y_t\mid X_t = x')\sum_{x''\in\Xcaltheta}Q^\theta_{x''x'}p^\theta(X_{t-1} = x''\mid Y_0^{t-1},X_0\sim\mu),
\end{align*}
which yields

\begin{align*}
\sigma_-\sum_{x'\neq x}p^\theta(Y_t\mid X_t = x')\leq B_\mu\leq (1-\sigma_-)\sum_{x'\neq x}p^\theta(Y_t\mid X_t = x')
\end{align*}
and the same result holds for $B_\nu$. Besides,

\begin{align*}
A_{\mu} &= \sum_{x\neq x'}p^\theta(X_t=x'\mid Y_0^{t-1},X_0\sim\mu)\\
&=\sum_{x'\neq x}\sum_{x''\in\Xcaltheta}Q^\theta_{x''x'}p^\theta(X_{t-1} = x''\mid Y_0^{t-1},X_0\sim\mu).
\end{align*}

Hence,

\begin{align*}
\sigma_-\leq A_\mu\leq 1-\sigma_-
\end{align*}
and the same result holds for $A_\nu$. Then, letting $\phi_\mu(x') = p^\theta(X_{t-1} = x' \mid Y_0^{t-1}, X_0 \sim \mu)$, we get, using the above expressions:

\begin{align*}
|A_\mu - A_\nu|&\leq (1-\sigma_-) \| \phi_\mu - \phi_\nu\|_1\\
|B_\mu - B_{\nu}|&\leq (1-\sigma_-)\sum_{x'\neq x}p^\theta(Y_t\mid X_t = x')\| \phi_\mu - \phi_\nu\|_1.
\end{align*} 

Thus,

\begin{align*}
|C_\mu - C_\nu| &= \left| \frac{B_\mu}{A_\mu}- \frac{B_\nu}{A_\nu} \right|\\
& \leq \frac{1}{A_\mu A_\nu}\left(B_\mu|A_\mu - A_\nu| + A_{\mu}|B_\mu - B_\nu|\right)\\
& \leq \frac{2(1-\sigma_-)^2}{\sigma_-^2} \sum_{x'\neq x}p^\theta(Y_t\mid X_t = x')\| \phi_\mu - \phi_\nu\|_1.
\end{align*}

Furthermore,

\begin{align*}
\frac{1}{C_\mu \wedge C_\nu}\leq \frac{(1-\sigma_-)}{\sigma_-\sum_{x'\neq x}p^\theta(Y_t\mid X_t = x')}
\end{align*}

Finally,

\begin{align*}
|\log C_\mu - \log C_\nu|\leq \frac{2}{(1-\rho)^3 } \| \phi_\mu - \phi_\nu\|_1.
\end{align*}

It remains to prove that $\| \phi_\mu - \phi_\nu\|_1\leq\rho^{t-1}$, which follows from the geometric ergodicity of the HMM. See for instance Corollary 1 of \cite{douc2004asymptotic} or Proposition 2.1 of \cite{dCGLLC15}.

\subsubsection{Proof of Lemma~\ref{lemma_conditionnementEtatsPasses}}
\label{sec_proof_conditionnementEtatsPasses}

\paragraph*{Proof of equation~\eqref{eq_proof_recentStates_erreurBlock}}

For all $t \geq 1$ and $y_1^t \in \Rbb^t$,
\begin{align*}
p^\theta( y_t \, | & \, y_1^{t-1}, \xfarU(\theta) \notin X_1^{t-a}, X_t \neq \xfarU(\theta)) \\
	&= p^\theta(y_t \, | \, y_1^{t-1}, \xfarU(\theta) \notin X_1^t)
    		p^\theta(\xfarU(\theta) \notin X_{(t-a+1) \vee 1}^{t-1} \, | \, y_1^{t-1}, \xfarU(\theta) \notin X_1^{t-a}, X_t \neq \xfarU(\theta)) \\
        &\quad + p^\theta(y_t \, | \, y_1^{t-1}, \xfarU(\theta) \notin X_1^{t-a}, X_t \neq \xfarU(\theta), \xfarU(\theta) \in X_{(t-a+1) \vee 1}^{t-1}) \\
		&\qquad \times p^\theta(\xfarU(\theta) \in X_{(t-a+1) \vee 1}^{t-1} \, | \, y_1^{t-1}, \xfarU(\theta) \notin X_1^{t-a}, X_t \neq \xfarU(\theta)),
\end{align*}
so that
\begin{align*}
&| p^\theta(y_t \, | \, y_1^{t-1}, \xfarU(\theta) \notin X_1^{t-a}, X_t \neq \xfarU(\theta))
		- p^\theta(y_t \, | \, y_1^{t-1}, \xfarU(\theta) \notin X_1^t) | \\
    &\leq p^\theta(\xfarU(\theta) \in X_{(t-a+1) \vee 1}^{t-1} \, | \, y_1^{t-1}, \xfarU(\theta) \notin X_1^{t-a}, X_t \neq \xfarU(\theta)) \\
    &\quad \times \!\! \Big(
    	p^\theta(y_t \, | \, y_1^{t-1}, \xfarU(\theta) \notin X_1^t)
    	+ p^\theta(y_t \, | \, y_1^{t-1}, \xfarU(\theta) \notin X_1^{t-a}, X_t \neq \xfarU(\theta), \xfarU(\theta) \in X_{t-a+1}^{t-1}) \Big) \\
    &\leq 2 p^\theta(\xfarU(\theta) \in X_{(t-a+1) \vee 1}^{t-1} \, | \, y_1^{t-1}, \xfarU(\theta) \notin X_1^{t-a}, X_t \neq \xfarU(\theta)) \sum_{x \in \Xcaltheta \setminus \{\xfarU(\theta)\}} \gamma^\theta_x(y_t - T^\theta_x(t)).
\end{align*}

In addition, under \textbf{(Aerg)},
\begin{align*}
p^\theta(y_t \, | \, y_1^{t-1}, \xfarU(\theta) \notin X_1^{t-a} & , X_t \neq \xfarU(\theta)) \\
	&= \sum_{x \in \Xcaltheta \setminus \{\xfarU(\theta)\}} p^\theta(y_t | X_t = x) p^\theta(X_t = x \, | \, y_1^{t-1}, \xfarU(\theta) \notin X_1^{t-a}) \\
	&\geq \sigma_- \sum_{x \in \Xcaltheta \setminus \{\xfarU(\theta)\}} \gamma^\theta_x(y_t - T^\theta_x(t))
\end{align*}
and the same holds for $p^\theta(y_t \, | \, y_1^{t-1}, \xfarU(\theta) \notin X_1^t)$, so that using that $|\log x - \log y|\leq \frac{|x-y|}{x\wedge y}$ for all $x,y > 0$, we obtain that for all $t \geq 1$ and $y_1^t \in \Rbb^t$
\begin{multline*}
| \log p^\theta(Y_t \, | \, X_1^{t-a} \neq \xfarU(\theta), X_t \neq \xfarU(\theta), Y_1^{t-1})
		- \log p^\theta(Y_t \, | \, X_1^t \neq \xfarU(\theta), Y_1^{t-1}) | \\
    \leq \frac{2}{\sigma_-} p^\theta(\xfarU(\theta) \in X_{(t-a+1) \vee 1}^{t-1} \, | \, y_1^{t-1}, \xfarU(\theta) \notin X_1^{t-a}, X_t \neq \xfarU(\theta)).
\end{multline*}

\paragraph*{Proof of equation~\eqref{eq_proof_recentStates_densites}}

By union bound,
\begin{align*}
&p^\theta(\xfarU(\theta) \in X_{(t-a+1) \vee 1}^{t-1} \, | \, y_1^{t-1}, \xfarU(\theta) \notin X_1^{t-a}, X_t \neq \xfarU(\theta)) \\
	&\leq \sum_{i=(t-a+1) \vee 1}^{t-1} p^\theta(X_i = \xfarU(\theta) \, | \, y_1^{t-1}, \xfarU(\theta) \notin X_1^{t-a}, X_t \neq \xfarU(\theta)) \\
	&= \sum_{i=(t-a+1) \vee 1}^{t-1} \sum_{x_{i-1}, x_{i+1} \in \Xcaltheta}
    	p^\theta(X_i = \xfarU(\theta) \, | \, y_i, X_{i-1}=x_{i-1}, X_{i+1}=x_{i+1}) \\
    	&\hspace{4cm} \times p^\theta(X_{i-1}=x_{i-1}, X_{i+1}=x_{i+1} \, | \, y_1^{t-1}, \xfarU(\theta) \notin X_1^{t-a}, X_t \neq \xfarU(\theta)) \\
	&\leq \sum_{i=(t-a+1) \vee 1}^{t-1} \sum_{x_{i-1}, x_{i+1} \in \Xcaltheta}
    	p^\theta(X_i = \xfarU(\theta) \, | \, y_i, X_{i-1}=x_{i-1}, X_{i+1}=x_{i+1}) \\
	&= \sum_{i=(t-a+1) \vee 1}^{t-1} \sum_{x_{i-1}, x_{i+1} \in \Xcaltheta}
    	\frac{
        	p^\theta(X_i = \xfarU(\theta) \, | \, X_{i-1}=x_{i-1}, X_{i+1}=x_{i+1}) \gamma^\theta_{\xfarU(\theta)}(y_i - T^\theta_{\xfarU(\theta)}(i))
        }{
        	\displaystyle \sum_{x \in \Xcaltheta} p^\theta(X_i = x \, | \, X_{i-1}=x_{i-1}, X_{i+1}=x_{i+1}) \gamma^\theta_x(y_i - T^\theta_x(i))
        }.
\end{align*}

Using the Markov property and \textbf{(Aerg)}, for all $x_{i-1}, x_{i+1} \in \Xcaltheta$,
\begin{align*}
p^\theta(X_i = x \, | \, X_{i-1}=x_{i-1}, X_{i+1}=x_{i+1}) \in [\sigma_-^2, 1].
\end{align*}

Hence,
\begin{align*}
p^\theta(\xfarU(\theta) \in X_{(t-a+1) \vee 1}^{t-1} &\, | \, y_1^{t-1}, \xfarU(\theta) \notin X_1^{t-a}, X_t \neq \xfarU(\theta)) \\
	&\leq \sum_{i=(t-a+1) \vee 1}^{t-1} \sum_{x_{i-1}, x_{i+1} \in \Xcaltheta}
    	\frac{\gamma^\theta_{\xfarU(\theta)}(y_i - T^\theta_{\xfarU(\theta)}(i))}{\displaystyle \sigma_-^2 \sum_{x \in \Xcaltheta} \gamma^\theta_x(y_i - T^\theta_x(i))} \\
	&\leq \frac{K^2}{\sigma_-^2} \sum_{i=(t-a+1) \vee 1}^{t-1} \frac{\gamma^\theta_{\xfarU(\theta)}(y_i - T^\theta_{\xfarU(\theta)}(i))}{\displaystyle \sum_{x \in \Xcaltheta} \gamma^\theta_x(y_i - T^\theta_x(i))}
\end{align*}
which concludes the proof.

\paragraph*{Proof of equation~\eqref{eq_proof_recentStates_utilisationEncadrements}}

This quantity is always bounded by 1 since all terms are nonnegative. In addition, under Assumptions \textbf{(Amax)} and \textbf{(Amin)}, for all $i \in I^U_n(\theta)$,
\begin{align*}
\frac{\gamma^\theta_{\xfarU(\theta)}(Y_i - T^\theta_{\xfarU(\theta)}(i))}{\displaystyle \sum_{x \in \Xcaltheta} \gamma^\theta_x(Y_i - T^\theta_x(i))}
	&\leq \sum_{x \in [\Kast]} \one_{X_i = \xast} \left(1 \wedge \frac{\gamma^\theta_{\xfarU(\theta)}(Z_i + T^*_\xast(i) - T^\theta_{\xfarU(\theta)}(i))}{\displaystyle \sup_{x \in \Xcaltheta} \gamma^\theta_x(Z_i + T^*_\xast(i) - T^\theta_x(i))} \right) \\
	&\leq 1 \wedge\frac{\displaystyle g\left(\left\{\inf_{\xast \in [\Kast]} |T^*_\xast(i) - T^\theta_{\xfarU(\theta)}(i)| - Z^{\max}_i\right\}_+\right)}{\displaystyle m\left(Z^{\max}_i + \sup_{\xast \in [\Kast]}\inf_{x \in \Xcaltheta} |T^*_\xast(i) - T^\theta_x(i)|\right)}.
\end{align*}

Since $\theta \in \Tcal(\alpha,n,D) \setminus \Ucal(\beta,n,B)$, Corollary \ref{cor_TcalUcal_tendances_proches} ensures that $\inf_{x \in \Xcaltheta} |T^*_\xast(i) - T^\theta_x(i)| \leq M(\alpha,D)$ for all $\xast \in [\Kast]$ and $i \in \{1, \dots, n\}$. Moreover, by definition of $I^U_n(\theta)$, $\inf_{\xast \in [\Kast]} |T^*_\xast(i) - T^\theta_{\xfarU(\theta)}(i)| \geq B$ for all $i \in I^U_n(\theta)$, so that 
\begin{align*}
\frac{\gamma^\theta_{\xfarU(\theta)}(Y_i - T^\theta_{\xfarU(\theta)}(i))}{\displaystyle \sum_{x \in \Xcaltheta} \gamma^\theta_x(Y_i - T^\theta_x(i))}
	&\leq 1 \wedge \frac{g(\{B - Z^{\max}_i\}_+)}{m(Z^{\max}_i + M(\alpha,D))}
\end{align*}
for all $i \in I^U_n(\theta)$, which concludes the proof.

\subsubsection{Proof of Lemma~\ref{lemma_conditionnementEtatPresent}}
\label{sec_proof_conditionnementEtatPresent}

Under Assumption \textbf{(Amax)},

\begin{align*}
p^\theta(Y_t\, | \, X_t = \xfarU(\theta))
	&= \gamma^\theta_{\xfarU(\theta)}( Y_t - T^\theta_{\xfarU(\theta)}(t) ) \\
	&= \sum_{\xast \in [\Kast]} \one_{X_t = \xast} \gamma^\theta_{\xfarU(\theta)} \left(Z_t+ T^*_\xast(t) - T^\theta_{\xfarU(\theta)}(t)\right) \\
	&\leq \sup_{\xast \in [\Kast]} g\left(\left\{ |T^*_\xast(t) - T^\theta_{\xfarU(\theta)}(t)| - Z^{\max}_t\right\}_+ \right) \\
	&\leq g\left( \left\{ \inf_{\xast \in [\Kast]} |T^*_\xast(t) - T^\theta_{\xfarU(\theta)}(t)| - Z^{\max}_t \right\}_+ \right),
\end{align*}
hence, for all $\theta \in \Tcal(\alpha,n,D) \setminus \Ucal(\beta,n,B)$,
\begin{equation}
\label{eq_proof_conditionnementEtatPresent_majoration}
p^\theta(Y_t\, | \, X_t = \xfarU(\theta))
	\leq 
\begin{cases}
g( \{ B - Z^{\max}_t \}_+ ) \quad \text{if } t \in I^U_n(\theta), \\
g(0) \quad \text{otherwise}.
\end{cases}
\end{equation}

On the other hand, under Assumptions \textbf{(Amin)} and \textbf{(Aerg)},
\begin{align*}
p^\theta(Y_t \, | \, X_t \neq \xfarU(\theta), Y_1^{t-1})
	&= \frac{p^\theta(Y_t, X_t \neq \xfarU(\theta) \, | \, Y_1^{t-1})}{p^\theta(X_t \neq \xfarU(\theta) \, | \, Y_1^{t-1})} \\
	&\geq \sum_{x \in \Xcaltheta, x \neq \xfarU(\theta)} p^\theta(Y_t \, | \, X_t=x, Y_1^{t-1}) p^\theta(X_t=x \, | \, Y_1^{t-1}) \\
&\geq \sigma_- \sum_{x \in \Xcaltheta, x \neq \xfarU(\theta)} p^\theta(Y_t \, | \, X_t=x) \\
&= \sigma_- \sum_{x \in \Xcaltheta, x \neq \xfarU(\theta)} \gamma^\theta_x(Y_t - T^\theta_x(t)) \\
&= \sigma_- \sum_{\xast \in [\Kast]} \one_{X_t = \xast} \sum_{x \in \Xcaltheta, x \neq \xfarU(\theta)} \gamma^\theta_x(Z_t + T^*_\xast(t) - T^\theta_x(t)) \\
&\geq \sigma_- \inf_{\xast \in [\Kast]} \sup_{x \in \Xcaltheta, x \neq \xfarU(\theta)} m(Z^{\max}_t + |T^*_\xast(t) - T^\theta_x(t)|) \\
&\geq \sigma_- m\left(Z^{\max}_t + \sup_{\xast \in [\Kast]} \inf_{x \in \Xcaltheta, x \neq \xfarU(\theta)} |T^*_\xast(t) - T^\theta_x(t)| \right).
\end{align*}

Using $\theta \notin \Tcal(\alpha,n,D)$ and Corollary \ref{cor_TcalUcal_tendances_proches}, for all $\xast \in [\Kast]$,
\begin{equation*}
\inf_{x \in \Xcaltheta, x \neq \xfarU(\theta)} |T^*_\xast(t) - T^\theta_x(t)| \leq M(\alpha,D).
\end{equation*}

For example, we can choose $x = x(\theta,\xast,n)$ (defined in Definition~\ref{def_xetxstar_reference}). We know that $\xfarU(\theta) \neq x(\theta,\xast,n)$ because we chose $B > M(\alpha,D)$, so that $|T^*_\xast(i) - T^\theta_{\xfarU(\theta)}(i)| > M(\alpha,D)$ for at least one $i \in \{1, \dots, n\}$, and because $|T^*_\xast(i) - T^\theta_{x(\theta,\xast,n)}(i)| \leq M(\alpha,D)$ for all $i \in \{1, \dots, n\}$. 
Therefore, for all $\theta \in \Tcal(\alpha,n,D) \setminus \Ucal(\beta,n,B)$,
\begin{align*}
p^\theta(Y_t \, | \, X_t \neq \xfarU(\theta), Y_1^{t-1})
	\geq \sigma_- m(Z^{\max}_t + M(\alpha,D)),
\end{align*}
which concludes the proof together with equation~\eqref{eq_proof_conditionnementEtatPresent_majoration}.

\section{Integrated log-likelihood}
\label{sec_integrated_likelihood}

In this section, we use the fact that the observed process $(Y_t)_{t \geq 1}$ may be replaced by the process $(Y_t - \Tbb^*_{B_t}(t), B_t)_{t \geq 1}$. While this process is not homogeneous under the parameter $\theta$, its distribution varies slowly over time
. We take advantage of this property to show the uniform convergence of the log-likelihood by approximating $(Y_t - \Tbb^*_{B_t}(t), B_t)_{t \geq 1}$ by a process that is locally (in time) homogeneous under the parameter $\theta$.
The limit can be written as an integral of limits of log-likelihoods of homogeneous HMM, hence the name \emph{integrated log-likelihood}.

\subsection{Convergence of the log-likelihood to the integrated log-likelihood}

Assume \textbf{(Aerg)}, \textbf{(Amax)}, \textbf{(Amin)}, \textbf{(Aint)} and \textbf{(Areg)}. In this section we shall prove Theorem \ref{th_conv_int_loglik}.
Let $M > 0$.
\bigskip

The normalized log-likelihood associated with the HMM $(Y_t,B_t)_{t\geq 1}$ can be written as
\begin{align}
\nonumber
\frac{1}{n} \ell_n^{(Y,B)}(\theta)
	&= \frac{1}{n} \log \!\!\!\! \underset{\forall t, \; \bbf^\theta(x_t) = B_t}{\sum_{x_1^n \text{ s.t. }}} \!\!\!\! \pi^\theta(x_1) Q^\theta(x_1, x_2) \dots Q^\theta(x_{n-1},x_n)
	\prod_{t=1}^n \gamma^\theta_{x_t}(Y_t - T^\theta_{x_t}(t)) \\
\label{llh_Zprime_B}
	&= \frac{1}{n} \log \!\!\!\! \underset{\forall t, \; \bbf^\theta(x_t) = B_t}{\sum_{x_1^n \text{ s.t. }}} \!\!\!\! \pi^\theta(x_1) Q^\theta(x_1, x_2) \dots Q^\theta(x_{n-1},x_n)
	\prod_{t=1}^n \gamma^\theta_{x_t}\left(Z'_t - D^{\theta,n}_{x_t}\left(\frac{t}{n}\right)\right),
\end{align}
with $D^{\theta,n}_x : u \in [0,1] \mapsto T^\theta_x(nu) - \Tbb^*_{\bbf^\theta(x)}(nu)$ and $Z'_t := Y_t - \Tbb^*_{B_t}(t)$. Note that $Z'_t = Z_t + \Delta(X_t)$. Recall that $\Dcal(M)$ is defined by Equation \eqref{eq_def_Dcal1}.

Theorem~\ref{th_compacite_S_pour_un_seul_I} implies that $\Cl(\Dcal(M))$ is compact, where $\Cl(\cdot)$ denotes the closure with respect to the supremum norm topology, hence Proposition \ref{prop_compacite_Dcal} holds. Together with Arzelà–Ascoli's theorem, this entails that $\Dcal(M)$ is uniformly equicontinuous and uniformly bounded by $M$. Hence there exists a continuity modulus $\nu$ such that for all $\delta>0$ and all $(s,u) \in [0,1]^2$, 
\begin{equation}
\label{eq_equiC0}
|s-u| \leq \delta \  \Rightarrow \  \sup_{n \geq 4K(d+1)} \sup_{\theta \in \Theta^\text{OK}_n(M)} \sup_{x \in \Xcaltheta} | D^{\theta,n}_x(s) - D^{\theta,n}_x(u) | \leq \nu(\delta).
\end{equation}

\begin{definition}[Log-likelihood of the homogenized process]
For each $\eta > 0$ and $\theta \in \Theta^\text{OK}_n(M)$, let
\begin{multline*}
\frac{1}{n} \ell_n^{(Y,B)}[\eta](\theta)
	:= \frac{1}{n} \log \sum_{x_1^n \text{ s.t. } \forall t, \; \bbf^\theta(x_t) = B_t} \pi^\theta(x_1) Q^\theta(x_1, x_2) \dots Q^\theta(x_{n-1},x_n) \\
	\times \prod_{t=1}^n \gamma^\theta_{x_t}\left(Z'_t - D^{\theta,n}_{x_t}\left( \eta \left\lfloor \frac{t}{\eta n} \right\rfloor \right)\right),
\end{multline*}
be the normalized log-likelihood of the process where each residual trend is made constant over segments of length $\eta$.
\end{definition}

\begin{remark}
This quantity is indeed a log-likelihood: $\frac{1}{n} \ell_n^{(Y,B)}[\eta](\theta) = \frac{1}{n} \ell_n^{(Y,B)}(\theta[N,n])$, where the parameter $\theta[N,n]$ is defined by $\pi^{\theta[N,n]} = \pi^\theta$, $Q^{\theta[N,n]} = Q^\theta$, $\gamma^{\theta[N,n]} = \gamma^\theta$, $\bbf^{\theta[N,n]} = \bbf^\theta$ and
\begin{equation}
\label{eq_thetaNn}
\forall x \in \Xcaltheta, \quad
	T^{\theta[N,n]}_x(t) = \Tbb^*_{\bbf^\theta(x)}(t) + D^{\theta,n}_x(\lfloor N\frac{t}{n} \rfloor).
\end{equation}

However, $\theta[N,n]$ has piecewise polynomial trends instead of polynomial trends, so that it does not belong to $\Theta$.
\end{remark}

\begin{figure}
\centering
\includegraphics[scale = 0.6]{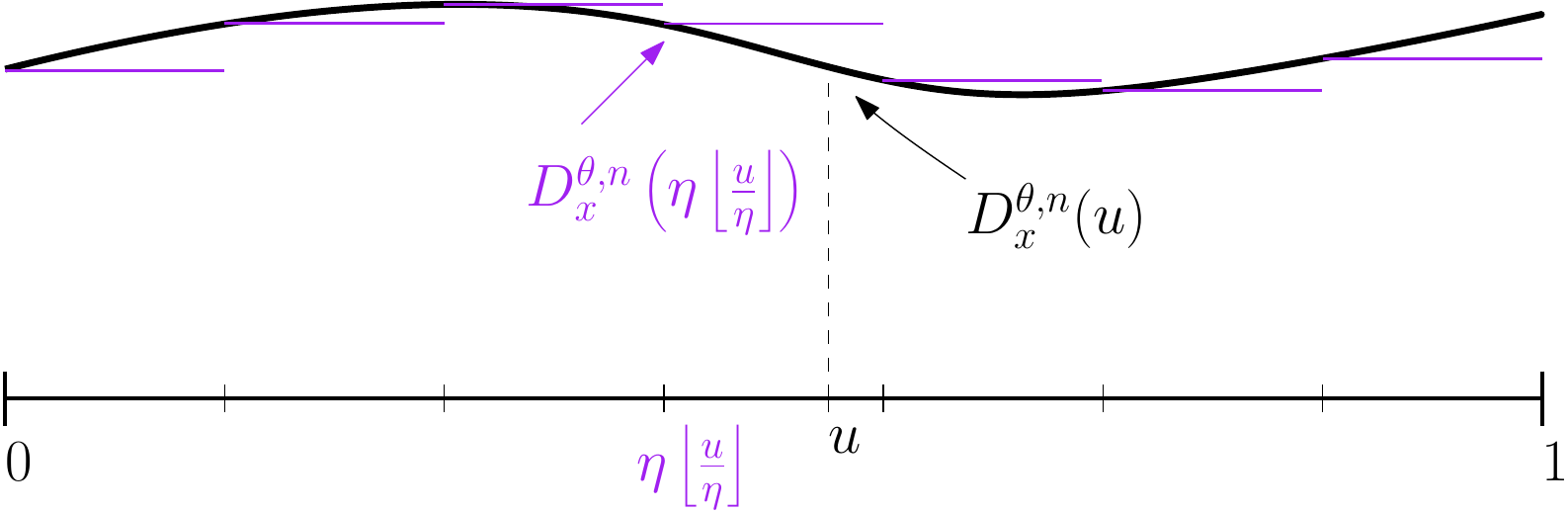}
\caption{Construction of the trends of the homogenized process.}
\label{dessin_lstat}
\end{figure}

$\frac{1}{n} \ell_n^{(Y,B)}[\eta](\theta)$ is an approximation of the log-likelihood of equation \eqref{llh_Zprime_B}. Assumption \textbf{(Areg)} together with Equation \eqref{eq_equiC0} ensure that for all $\delta > 0$, $n\geq 1$, $\theta \in \Theta_n^{\textrm{OK}}(M)$, $x \in \Xcaltheta$ and $t \in \{1,\dots,n\}$, 
\begin{equation*}
\gamma^\theta_{x_t}\left(Z'_t - D^{\theta,n}_{x_t}\left(\frac{t}{n}\right)\right)
	\in \left[ e^{- L(|Z'_t|+M) \omega(\nu(\delta))}, e^{L(|Z'_t|+M) \omega(\nu(\delta))} \right]
		\gamma^\theta_{x_t}\left(Z'_t - D^{\theta,n}_{x_t}\left( \delta \left\lfloor \frac{t}{\delta n} \right\rfloor \right)\right),
\end{equation*}
hence, for all $\delta>0$ and $n\geq 1$,
\begin{align}
\nonumber
\sup_{\theta \in \Theta^\text{OK}_n(M)}
	\left| \frac{1}{n} \ell_n^{(Y,B)}(\theta) - \frac{1}{n} \ell_n^{(Y,B)}[\delta](\theta) \right|
		&\leq \omega(\nu(\delta)) \times \frac{1}{n} \sum_{t=1}^n L(|Z'_t|+M) \\
\label{eq_stationnarisation}
		&\leq \omega(\nu(\delta)) \times \frac{1}{n} \sum_{t=1}^n L(\|\Delta\|_\infty + M + Z^{\max}_t).
\end{align}

\begin{remark}
Under \textbf{(Areg)}, the law of large numbers entails that almost surely, for all $N\geq 1$,
\begin{multline*}
\limsup_{n \rightarrow +\infty} \sup_{\theta \in \Theta^\text{OK}_n(M)}
	\left| \frac{1}{n} \ell_n^{(Y,B)}(\theta) - \frac{1}{n} \ell_n^{(Y,B)}\left[\frac{1}{N}\right](\theta) \right| \\
		\leq \omega\left(\nu\left(\frac{1}{N}\right)\right) \Ebb^*[L(\|\Delta\|_\infty + M + Z^{\max}_1)].
\end{multline*}
\end{remark}

Recall the following notation. For all $K' \in \Nbb^*$, for all $K'$-uple $\gamma = (\gamma_x)_{x \in [K']}$ of measurable functions and for all $\Dbf = (D_x)_{x \in [K']} \in \Rbb^{K'}$, let
\begin{equation*}
\tau(\gamma, \Dbf) := (z' \longmapsto \gamma_x(z' - D_x))_{x \in [K']}
\end{equation*}
the vector of functions $\gamma$ translated by the vector $\Dbf$.

\begin{definition}\label{def_lhom}
Let $\pi$ be a probability measure on $[K']$, $Q$ be a $K' \times K'$ transition matrix, $\gamma$ be a vector of $K'$ emission densities on $\Rbb$ and $\bbf$ be a function $[K'] \longrightarrow \Bcal^*$.
Let $(X_t, (\tilde{Z}_t,\tilde{B}_t))_{t \geq 1}$ be a homogeneous HMM taking values in $[K'] \times (\Rbb \times \Bcal^*)$ with parameter $(\pi, Q, (\gamma_x \otimes \one_{\bbf(x)})_{x \in [K']})$.

Denote by $\frac{1}{n} \ell_n^\text{hom}(\pi,Q,\gamma,\bbf)\{(\tilde{z},\tilde{b})_1^n\}$ (resp. $\ell^\text{hom}(Q,\gamma,\bbf)$) the normalized log-likelihood of the parameter $(\pi,Q,\gamma,\bbf)$ for the observations $(\tilde{z},\tilde{b})_1^n$ (resp. the limit of the log-likelihood, if it exists), that is
\begin{equation*}
\frac{1}{n} \ell_n^\text{hom}(\pi,Q,\gamma,\bbf)\{(\tilde{z},\tilde{b})_1^n\}
	= \frac{1}{n} \log \!\!\! \sum_{x_1^n \in [K']^n} \!\!\! \pi(x_1) Q(x_1, x_2) \dots Q(x_{n-1}, x_n) \prod_{t=1}^n \gamma_{x_t}(\tilde{z}_t) \one_{\bbf(x_t) = \tilde{b}_t}
\end{equation*}
and
\begin{equation}
\label{eq_cvg_logV_homogene}
\ell^\text{hom}(Q,\gamma,\bbf)
	= \lim_{n \rightarrow \infty} \frac{1}{n} \ell_n^\text{hom}(\pi,Q,\gamma,\bbf)\{(\tilde{Z},\tilde{B})_1^n\}.
\end{equation}
\end{definition}

The following Lemma ensures the existence of the limit of the normalized log-likelihood in Definition~\ref{def_lhom} as well as its uniform continuity with respect to the parameter. It is a consequence of a result concerning homogeneous HMM stated in \cite{douc2004asymptotic}.

\begin{lemma}
\label{lem_cvg_lstat}
Assume \textbf{(Amax)}, \textbf{(Amin)}, \textbf{(Aint)} and \textbf{(Areg)}. Let $K' \in \Nbb^*$. The following points hold.
\begin{itemize}
\item Almost surely, for all $Q \in \Sigma_{K'}^{\sigma_-}$, $\gamma \in \Gamma^{K'}$, $\Dbf \in \Rbb^{K'}$ and $\bbf : [K'] \longrightarrow \Bcal^*$, the quantity
\begin{equation*}
\ell^\text{hom}(Q^\theta, \tau(\gamma^\theta,\Dbf), \bbf)
\end{equation*}
from Equation~\eqref{eq_cvg_logV_homogene} exists and is finite almost surely under $\Pbb^*$ when $(\tilde{Z}_t,\tilde{B}_t)_t = (Z'_t,B_t)_t$.

\item For all $K' \in \Nbb^*$, the mapping
\begin{multline*}
(Q^\theta,\gamma^\theta, \Dfrak, u, \bbf) \in \Sigma_{K'}^{\sigma_-}\times \Gamma^{K'} \times \Cl(\Dcal(M))^{K'} \times [0,1] \times (\Bcal^*)^{K'} \\
	\longmapsto \ell^\text{hom}(Q^\theta, \tau(\gamma^\theta,\Dfrak(u)), \bbf)
\end{multline*}
is continuous and its domain is compact, so that it is uniformly continuous.

\item Almost surely, for all $N \in \Nbb^*$, 
\begin{multline*}
\sup_{(\pi, Q,\gamma, \Dfrak, u, \bbf)} \sup_{s \in \{0, \dots, (N-1)n\}}
    \Bigg| \frac{1}{n} \ell^\text{hom}_n(\pi, Q, \tau(\gamma,\Dfrak(u)), \bbf)\{(Z',B)_{s+1}^{s+n}\} \\
    	- \ell^\text{hom}(Q, \tau(\gamma,\Dfrak(u)), \bbf) \Bigg|
    \underset{n \rightarrow \infty}{\longrightarrow} 0
\end{multline*}
where the supremum is taken for $(\pi,Q,\gamma, \Dfrak, u, \bbf) \in \Delta_{K'} \times \Sigma_{K'}^{\sigma_-} \times \Gamma^{K'} \times \Cl(\Dcal(M))^{K'} \times [0,1] \times (\Bcal^*)^{K'}$.
\end{itemize}
\end{lemma}

\begin{proof}
Proof in Section~\ref{sec_preuve_cvg_unif_logV}
\end{proof}

As a consequence, the family of functions 
\begin{equation*}
\bigcup_{K'=1}^K \{ u \in [0,1]
	\longmapsto \ell^\text{hom}(Q, \tau(\gamma,\Dfrak(u)), \bbf) \}_{Q\in \Sigma_{K'}^{\sigma_-}, \gamma\in\Gamma^{K'}, \Dfrak \in \Dcal(M)^{K'}, \bbf \in (\Bcal^*)^{K'}}
\end{equation*}
is uniformly equicontinuous, which ensures the following result.

\begin{corollary}[Riemann approximation of the integral]
\label{cor_approx_riemannuniforme}
The quantity
\begin{multline}
\label{eq_RN_approx_riemannuniforme}
R_N := \sup_{n \geq n_1(M)} \sup_{\theta \in \Theta^\text{OK}_n(M)}
	\Bigg| \frac{1}{N} \sum_{i=0}^{N-1} \ell^\text{hom}\left(Q^{\theta}, \tau\left(\gamma^{\theta},\Dfrak^{\theta,n}\left(\frac{i}{N}\right)\right), \bbf^\theta\right) \\
		- \int_0^1 \ell^\text{hom}\left(Q^{\theta}, \tau\left(\gamma^{\theta},\Dfrak^{\theta,n}(u)\right), \bbf^\theta\right) du
	\Bigg|
\end{multline}
satisfies
\begin{equation*}
R_N \underset{N \rightarrow +\infty}{\longrightarrow} 0.
\end{equation*}
\end{corollary}

The integrated log-likelihood $\ell^\text{int}$ from Definition \ref{def_int_loglik1} is continuous by uniform continuity of $\ell^\text{hom}$. We may now prove the main result of this section, that is the convergence of the normalized log-likelihood to the integrated log-likelihood.

\paragraph*{Proof of Theorem \ref{th_conv_int_loglik}}

By the triangle inequality and using Equations~\eqref{eq_RN_approx_riemannuniforme} and~\eqref{eq_stationnarisation}, for all $n \geq n_1(M)$ and $N \in \Nbb^*$,
\begin{align*}
\sup_{\theta \in \Theta^\text{OK}_n(M)} &
	\left| \frac{1}{n} \ell_n^{(Y,B)}(\theta)
	- \ell^\text{int}(Q^\theta,\gamma^\theta,\Dfrak^{\theta,n},\bbf^\theta) \right| \\
		&\leq \omega\left(\nu\left(\frac{1}{N}\right)\right) \frac{1}{n} \sum_{t=1}^n L(\|\Delta\|_\infty + M + Z^{\max}_t) + R_N \\
		&+ \sup_{\theta \in \Theta^\text{OK}_n(M)}
	\left| \frac{1}{n} \ell_n^{(Y,B)}\left[\frac{1}{N}\right](\theta)
	- \frac{1}{N} \sum_{i=0}^{N-1} \ell^\text{hom}\left(Q^{\theta}, \tau\left(\gamma^{\theta},\Dfrak^{\theta,n}\left(\frac{i}{N}\right)\right), \bbf^\theta \right) \right|.
\end{align*}

For the sake of simplicity, assume that $\frac{n}{N}$ is an integer. By Equation~\eqref{eq_thetaNn}, for all $\theta \in \Theta^\text{OK}_n(M)$, there exists $\theta[N,n]$ such that
\begin{equation*}
\frac{1}{n} \ell_n^{(Y,B)}\left[\frac{1}{N}\right](\theta) = \frac{1}{n} \ell_n^{(Y,B)}(\theta[N,n]),
\end{equation*}
so that
\begin{align*}
\frac{1}{n} \ell_n^{(Y,B)}\left[\frac{1}{N}\right](\theta)
	&= \frac{1}{n} \sum_{i=0}^{N-1} \log p^{\theta[N,n]}\left((Y,B)_{1 + i\frac{n}{N}}^{\frac{n}{N} + i\frac{n}{N}} \; | \; (Y,B)_1^{i\frac{n}{N}}\right) \\
	&= \frac{1}{N} \sum_{i=0}^{N-1} \frac{1}{\frac{n}{N}}
		\ell^\text{hom}_{\frac{n}{N}}\left(\pi^{\theta[N,n]}_{i\frac{n}{N}}, Q^\theta, \tau\left(\gamma^\theta,\Dfrak^{\theta,n}\left(\frac{i}{N}\right)\right), \bbf^\theta\right)\left\{ (Z',B)_{1 + i\frac{n}{N}}^{\frac{n}{N} + i\frac{n}{N}} \right\},
\end{align*}
where $\pi^{\theta[N,n]}_{i\frac{n}{N}}$ is defined as the distribution of $X_{1 + i\frac{n}{N}}$ conditionally to $(Y,B)_{1 + i\frac{n}{N}}^{\frac{n}{N} + i\frac{n}{N}}$ under the parameter $\theta[N,n]$.
Hence, Lemma~\ref{lem_cvg_lstat} implies that almost surely,
\begin{align*}
\limsup_{n \rightarrow +\infty} \sup_{\theta \in \Theta^\text{OK}_n(M)}
	&\left| \frac{1}{n} \ell_n^{(Y,B)}(\theta)
	- \ell^\text{int}(Q^\theta,\gamma^\theta,\Dfrak^{\theta,n},\bbf^\theta) \right| \\
	&\leq \inf_{N \in \Nbb^*} \left[ \omega\left(\nu\left(\frac{1}{N}\right)\right) \Ebb^*[L(\|\Delta\|_\infty + M + Z^{\max}_1)] + R_N \right] \\
	&= 0.
\end{align*}
The conclusion follows from Theorem~\ref{th_approx_blocs}.\bigskip

\subsection{Maximizers of the integrated log-likelihood and identifiability}

In this section we prove Proposition \ref{prop_identifiabilite_via_lint_0}. Assume \textbf{(Amax)}, \textbf{(Amin)}, \textbf{(Areg)}, \textbf{(Aid)} and \textbf{(Acentering)} and assume that $\Kast = K$ is known.

\begin{remark}
Assumptions \textbf{(Areg)}, \textbf{(Amax)} and \textbf{(Amin)} can be replaced by
\begin{equation*}
\begin{cases}
\forall \theta \in \Theta, \quad
\forall x \in [\Kast], \quad
z \in \Rbb \longmapsto \gamma^{\theta}(z) \text{ is continuous,} \\
\forall x \in [\Kast], \quad
\gamma^*_x(z) \underset{|z| \rightarrow +\infty}{\longrightarrow} 0 \\
\forall x \in [\Kast], \quad
\forall z \in \Rbb, \quad
\gamma^*_x(z) > 0.
\end{cases}
\end{equation*}
\end{remark}

The maximum of $\ell^\text{int}$ is reached at $(Q,\gamma, \Dfrak = (D_x)_{x \in [\Kast]}, \bbf)$ if and only if the integrand is maximal for almost every $u \in [0,1]$, which means under \textbf{(Aid)} that
\begin{equation*}
\left(Q, \left(\gamma(\cdot - D_x(u)) \otimes \one_{\bbf(x)}\right)_{x \in [\Kast]}\right)
	= \left(Q^*, \left(\gamma^*_x(\cdot - \Delta(x)) \otimes \one_{\bbf^*(x)}\right)_{x \in [\Kast]}\right)
\end{equation*}
up to permutation of the hidden states for all $u \in [0,1]$.

Let us assume that the permutation is not constant at $u$. Since there are only a finite number of possible permutations of $[\Kast]$, there exist two sequences $(u_i)_{i \geq 1}$ and $(v_i)_{i \geq 1}$ converging to $u$, one corresponding to a permutation $p$ and the other to a permutation $p' \neq p$, that is
\begin{equation*}
\forall i \geq 1, \quad
\forall x \in [\Kast], \quad 
\begin{cases}
\gamma_{x}(\cdot - D_x(u_i)) = \gamma^*_{p(x)}(\cdot - \Delta(p(x))) \quad \text{and} \quad \bbf(x) = \bbf^*(p(x)) \\
\gamma_{x}(\cdot - D_x(v_i)) = \gamma^*_{p'(x)}(\cdot - \Delta(p'(x))) \quad \text{and} \quad \bbf(x) = \bbf^*(p'(x))
\end{cases}
\end{equation*}

Therefore, by continuity, for all $x \in [\Kast]$ 
\begin{align*}
(\gamma^*_{p(x)}(\cdot - \Delta(p(x))), \bbf^*(p(x))) = (\gamma^*_{p'(x)}(\cdot - \Delta(p'(x))), \bbf^*(p'(x))),
\end{align*}
so that $p = p'$ according to \textbf{(Aid)}, which contradicts the assumption that the permutation is not constant in $u$. Therefore, the permutation does not depend on $u$.

One may assume without loss of generality that the permutation is the identity, in other words $Q = Q^*$, $\bbf = \bbf^*$ and
\begin{equation*}
\forall u \in [0,1], \quad
\forall x \in [\Kast], \quad 
\gamma_{x}(\cdot - D_x(u)) = \gamma^*_x(\cdot - \Delta(x)).
\end{equation*}

Here, we took $u$ in the whole segment $[0,1]$ instead of a subset with measure 1 because the mapping $u \in [0,1] \longmapsto \gamma_{x}(\cdot - D_x(u))$ is continuous under \textbf{(Areg)}. If $D_x$ is not constant at some $x \in [\Kast]$, this entails that $\gamma^*_x$ is invariant by translation, so that it is constant, which contradicts \textbf{(Amax)}. Therefore, $\Dfrak$ is constant.

Finally,
\begin{equation*}
\forall x \in [\Kast], \quad 
\frac{1}{2} = \int_{z \leq D_x} \gamma_{x}(z - D_x) dz = \int_{z \leq D_x} \gamma^*_x(z - \Delta(x)) dz
\end{equation*}
using \textbf{(Acentering)}, so that $D_x$ is a median of $\gamma^*_x$. To conclude, note that under \textbf{(Amin)} and \textbf{(Acentering)}, $\Delta(x)$ is the only median of $\gamma^*_x$.

\subsection{Uniform convergence of the homogeneous log-likelihood}
\label{sec_preuve_cvg_unif_logV}

Let us prove Lemma~\ref{lem_cvg_lstat}.
The following theorem is a reformulation of Proposition 2 of~\cite{douc2004asymptotic}. Note that their proof also works when the space of parameters is not parametric.

\begin{theorem}
Let $\Vcal$ be a Polish space and write $\Dcal(\Vcal)$ the set of nonnegative functions of $\Vcal$.
Let $K' \in \Nbb^*$. Let $\Qcal_{K'}$ be the set of transition matrices of size $K'$ and $\Delta_{K'}$ the set of probability measures on $[K']$. Let $(V_t)_{t \geq 1}$ be an ergodic and stationary process taking values in $\Vcal$ with distribution $\Pbb^*$.

Consider a compact metric space $\Omega$ and mappings $\omega \longmapsto Q^\omega \in \Qcal_{K'}$ and $\omega \longmapsto \gamma^\omega \in \Dcal(\Vcal)^{K'}$.
Assume that $\omega \longmapsto Q^\omega$ is continuous and for all $v \in \Vcal$, the mapping $\omega \longmapsto \gamma^\omega(v) \in \Rbb_+^{K'}$ is continuous. Finally, assume that there exists a constant $\sigma_- > 0$ such that
\begin{gather}
\label{eq_douc_minorationQ}
\inf_{\omega \in \Omega} \inf_{x,x' \in [K']} Q^\omega(x,x') \geq \sigma_-, \\ 
\label{eq_douc_majorationgamma}
\sup_{\omega \in \Omega} \sup_{x \in [K']} \sup_{v \in \Vcal} \gamma^\omega_x(v) < \infty, \\ 
\label{eq_douc_emdensitesintegrables}
\Ebb^*\left[ \sup_{\omega \in \Omega} \Big(\log \sum_{x \in [K']} \gamma^\omega_x(V_1)\Big)_- \right] < \infty.
\end{gather}

For all $\pi \in \Delta_{K'}$, $\omega \in \Omega$ and $v_1^n \in \Vcal^n$, let
\begin{equation*}
\frac{1}{n} l_n(\pi, Q^\omega, \gamma^\omega)\{v_1^n\}
	:= \frac{1}{n} \log \sum_{x_1^n \in [K']^n} \pi(x_1) Q^\omega(x_1, x_2) \dots Q^\omega(x_{n-1}, x_n) \prod_{t=1}^n \gamma^\omega_{x_t}(v_t)
\end{equation*}
be the log-likelihood corresponding to the HMM with parameters $(\pi, Q^\omega, \gamma^\omega)$ and to the observations $v_1^n$.

Then for all $\pi \in \Delta_{K'}$ and $\omega \in \Omega$, there exists a finite $l(Q^\omega, \gamma^\omega)$ such that almost surely,
\begin{equation*}
\frac{1}{n} l_n(\pi, Q^\omega, \gamma^\omega)\{V_1^n\}
	\underset{n \rightarrow \infty}{\longrightarrow} l(Q^\omega, \gamma^\omega).
\end{equation*}

In addition, for all $N \in \Nbb^*$, the mapping $\omega \longmapsto l(Q^\omega, \gamma^\omega)$ is continuous and
\begin{equation*}
\sup_{\omega \in \Omega} \sup_{\pi \in \Delta_{K'}}
	\left| \frac{1}{n} l_n(\pi, Q^\omega, \gamma^\omega)\{V_1^n \}
    	- l(Q^\omega, \gamma^\omega) \right|
    \underset{n \rightarrow \infty}{\longrightarrow} 0
\end{equation*}
almost surely.
\end{theorem}

Let us check these assumptions. First, let $\Vcal = \Rbb \times \Bcal^*$, $V_t = (Z'_t,B_t)$ and
\begin{equation*}
\Omega = \left\{ \omega = (Q,\gamma,\Dfrak,u,\bbf) \in \Sigma_{K'}^{\sigma_-} \times \Gamma^{K'} \times \text{Cl}(\Dcal(M))^{K'} \times [0,1] \times (\Bcal^*)^{K'} \right\}.
\end{equation*}

By Proposition~\ref{prop_compacite_Dcal}, $\Omega$ is compact. It is also metrizable under~\textbf{(Areg)}: for instance, let $(x_i)_{i \geq 1}$ be a dense sequence in $\Rbb$, endow $\Gamma$ with the distance
\begin{equation*}
d_\Gamma(\gamma, \gamma') = \sum_{i \geq 1} 2^{-i} ( |\gamma(x_i) - \gamma'(x_i)| \wedge 1 )
\end{equation*}
and thus $\Omega$ is metrizable as a product of metric spaces.

By the uniform continuity of $\text{Cl}(\Dcal(M))$, the mappings 
\begin{equation*}
\begin{cases}
\omega = (Q,\gamma,\Dfrak,u,\bbf) \in \Omega \longmapsto Q^\omega := Q \\
\omega = (Q,\gamma,\Dfrak,u,\bbf) \in \Omega \longmapsto \gamma^\omega(z',b) := (\gamma_{x}(z' - D_x(u)) \one_{\bbf(x)}(b))_{x \in [K']}
\end{cases}
\end{equation*}
are continuous for all $(z',b) \in \Rbb \times \Bcal^*$. The lower bound~\eqref{eq_douc_minorationQ} on the transition matrices is ensured by \textbf{(Aerg)} and the upper bound~\eqref{eq_douc_majorationgamma} on the densities is implied by \textbf{(Amax)}. Finally, the integrability condition~\eqref{eq_douc_emdensitesintegrables} follows from the fact that for all $\omega = (Q,\gamma,\Dfrak = (D_x)_{x \in [K']},u,\bbf) \in \Omega$,
\begin{equation*}
\sum_{x \in [K']} \gamma_x^\omega(Z'_1,B_1)
	\geq \inf_{x \in [K']} \gamma_x(Z'_1 - D_x(u))
    \geq m(M + Z^{\max}_1)
\end{equation*}
by \textbf{(Amin)}, and $\Ebb^*[ - \log m(M + Z^{\max}_1) ] < \infty$ by \textbf{(Aint)}.

Thus, the previous theorem holds, which shows that the application
\begin{equation*}
\omega \longmapsto l(Q^\omega, \gamma^\omega) =: l^\text{hom}(Q, \tau(\gamma, \Dfrak(u)), \bbf)
\end{equation*}
is continuous on $\Omega$. For the uniform convergence, let $\pi_U$ be the uniform distribution on $[K']$ and let
\begin{equation*}
S_{s,n}(\omega) = \frac{1}{n} l_n(\pi_U, Q^\omega, \gamma^\omega)\{V_{s+1}^{s+n}\}
\end{equation*}
for all $s,n \in \Nbb^*$ and $\omega \in \Omega$.

The theorem implies that almost surely,
\begin{equation*}
\lim_{n \rightarrow \infty} \sup_{\omega \in \Omega} \left| \frac{1}{n} S_{0,n}(\omega) - l(Q^\omega,\gamma^\omega) \right| = 0.
\end{equation*}

Hence, for all $\epsilon > 0$, there exists a (random) $n(\epsilon)$ such that,
\begin{equation}
\label{eq_approx_logV_par_S}
\forall n \geq n(\epsilon), \quad 
\sup_{\omega \in \Omega} \left| \frac{1}{n} S_{0,n}(\omega) - l(Q^\omega,\gamma^\omega) \right| \leq \epsilon.
\end{equation}

The following Lemma is a reformulation of Lemma 2 of~\cite{douc2004asymptotic} for compact nonparametric parameter spaces.

\begin{lemma}
\label{lem_oubli_initial_logV}
Under the same assumptions as the previous theorem, for all $v_1^n \in \Vcal^n$,
\begin{equation*}
\sup_{\omega \in \Omega} \sup_{\pi \in \Delta_K}
	|l_n(\pi,Q^\omega, \gamma^\omega)\{v_1^n\}
    	- l_n(\pi_U,Q^\omega, \gamma^\omega)\{v_1^n\}|
    \leq \frac{1}{\sigma_-^2}.
\end{equation*}
\end{lemma}

Therefore,
\begin{align*}
| n S_{s,n}(\omega) - l_n(\pi_{X_{s+1} | V_1^s, X_1 \sim \pi_U},Q^\omega,\gamma^\omega)\{V_{s+1}^{s+n}\} | \leq \frac{1}{\sigma_-^2}.
\end{align*}

Note that
\begin{align*}
l_n(\pi_{X_{s+1} | V_1^s, X_1 \sim \pi_U},Q^\omega,\gamma^\omega)\{V_{s+1}^{s+n}\}
	&= l_{s+n}(\pi_U,Q^\omega,\gamma^\omega)\{V_{1}^{s+n}\} - l_s(\pi_U,Q^\omega,\gamma^\omega)\{V_{1}^{s}\} \\
    &= (s+n) S_{0,s+n}(\omega) - s S_{0,s}(\omega),
\end{align*}
so that
\begin{align*}
| n S_{s,n}(\omega) - (s+n) S_{0,s+n}(\omega) - s S_{0,s}(\omega) | \leq \frac{1}{\sigma_-^2}.
\end{align*}

Thus, equation~\eqref{eq_approx_logV_par_S} entails that for all $s \geq 1$, $n \geq n(\epsilon)$ and $\omega \in \Omega$,
\begin{align*}
| n S_{s,n}(\omega) - n l(Q^\omega,\gamma^\omega) |
	\leq (2s+n) \epsilon + \frac{1}{\sigma_-^2}.
\end{align*}

Therefore, by Lemma~\ref{lem_oubli_initial_logV}, for all $n \geq n(\epsilon)$ and $s \in \{0, \dots, (N-1)n \}$:
\begin{align*}
\sup_{\omega \in \Omega}
\sup_{\pi \in \Delta_{K'}}
	\left| \frac{1}{n} l_n(\pi,Q^\omega,\gamma^\omega)\{V_{s+1}^{s+n}\} - l(Q^\omega,\gamma^\omega) \right|
	\leq (2N-1) \epsilon + \frac{2}{n \sigma_-^2},
\end{align*}
which concludes the proof.

\section{Miscellaneous proofs}

\subsection{Proof of Lemma~\ref{lem_hoeffding}}
\label{sec_proof_infSegments}

Let us first state a Hoeffding inequality for uniformly ergodic Markov chains using \textbf{(Aerg)} (see e.g. \cite{GO02MChoeffding}): for all $\epsilon > 0$, $x_1 \in [\Kast]$ and $n \geq \frac{1}{2 \epsilon \sigma_-}$,
\begin{equation*}
\Pbb\left( \Pbb(X_1 = \xast) - \frac{1}{n} \sum_{t=1}^n \one_{X_t = \xast} \geq \epsilon \; \Big| \; X_1 = x_1 \right) \leq \exp\left( - \frac{\sigma_-^2 \epsilon^2}{2} n \right).
\end{equation*}

The value of $\Pbb(X_1 = \xast)$ in the inequality is the one corresponding to the stationary distribution, so it is bounded below by $\sigma_-$ using \textbf{(Aerg)}. Thus, for all $\delta > 0$, $\epsilon > 0$, $n \geq \frac{1}{\delta \epsilon \sigma_-}$ and $x_1 \in [\Kast]$,
\begin{equation*}
\Pbb\left( \frac{2}{\delta n} \sum_{t=1}^{\delta n / 2} \one_{X_t = \xast} \leq \sigma_- - \epsilon \; \Big| \; X_1 = x_1 \right) \leq \exp\left( - \frac{\sigma_-^2 \epsilon^2 \delta}{4} n \right).
\end{equation*}

Assume $n \geq \frac{2}{\delta (\sigma_-)^2}$. Choose $\epsilon = \sigma_- / 2$ and apply a union bound on a covering $\Rcal$ of $\{1, \dots, n\}$ in at most $2n / \delta$ segments of size $\delta n/2$:
\begin{equation*}
\Pbb\left( \inf_{S \in \Rcal} \frac{1}{n} \sum_{t \in S} \one_{X_t = \xast} \leq \frac{\delta \sigma_-}{4} \right) \leq \frac{2 n}{\delta} \exp\left( - \frac{\sigma_-^4 \delta}{16} n \right).
\end{equation*}

Borel-Cantelli's lemma yields the result.

\subsection{Proof of Lemma~\ref{lemma_top_quantile}}
\label{sec_proof_cvgFractionSomme}

Without loss of generality, we assume $\delta_n \longrightarrow \delta$ almost surely (this is possible by replacing $\delta_n$ by $\delta_n \wedge \delta$ in the first statement and $\delta_n \vee \delta$ in the second).

Let us first show that the two statements are equivalent. Assume that the second one holds. Let $(U_t)_{t \geq 1}$ and $(\delta_n)_n$ as in the first statement. Apply the second one to the i.i.d sequence of non-negative integrable random variables $(-U_t)_{t \geq 1}$:
\begin{align*}
\limsup_{n \rightarrow \infty} 
	\underset{|S| \leq \delta_n n}{\sup_{S \subset \{1, \dots, n\}}}
		\frac{1}{n} \sum_{t \in S} (-U_t)
	&\leq \Ebb[(-U_1) \one_{-U_1 \geq -q_U(\delta)}] \\
	&\leq - \Ebb[U_1 \one_{U_1 \leq q_U(\delta)}].
\end{align*}

Add $\Ebb[U_1]$ on each side and use the law of large numbers:
\begin{align*}
\limsup_{n \rightarrow \infty} 
	\underset{|S| \leq \delta_n n}{\sup_{S \subset \{1, \dots, n\}}}
		\frac{1}{n} \sum_{t \notin S} U_t
	&\leq \Ebb[U_1 \one_{U_1 > q_U(\delta)}].
\end{align*}

Finally, replace $\delta_n$ by $1 - \delta_n$, $\delta$ by $1 - \delta$ and the sets $S$ by their complementary to obtain the first statement.
\bigskip

Let us now show the second statement (regarding non-negative random variables). Given a random vector $(V_1, \dots, V_n)$, we write $V_{(1)} \leq V_{(2)} \leq \dots \leq V_{(n)}$ its order statistics. Let $\delta \in (0,1)$ and let $(\delta_n)_n$ be a non-increasing sequence of $[0,1]$-valued random variables whose limit is $\delta$ almost surely.

For all $\beta \in (0,1)$, write $\hat{q}_V(\beta) := V_{(\lfloor \beta n \rfloor)}$ the empirical $\beta$-quantile. Then
\begin{equation*}
\underset{|S| \leq \delta_n n}{\sup_{S \subset \{1, \dots, n\}}}
	\frac{1}{n} \sum_{t \in S} V_t
	= \frac{1}{n} \sum_{t=1}^n V_t \one_{V_t \geq \hat{q}_V(1 - \delta_n)}.
\end{equation*}

Let us show that almost surely
\begin{equation*}
\left| \frac{1}{n} \sum_{t=1}^n V_t \one_{V_t \geq \hat{q}_V(1 - \delta_n)} - \frac{1}{n} \sum_{t=1}^n V_t \one_{V_t \geq q_V(1 - \delta)}\right| \underset{n\to\infty} {\longrightarrow} 0
\end{equation*}
and the result will follow by the law of large numbers. Thus we have to show that
\begin{equation*}
\frac{1}{n} \sum_{t=1}^n V_t \left(\one_{q_V(1 - \delta) \leq V_t \leq \hat{q}_V(1 - \delta_n)} + \one_{\hat{q}_V(1 - \delta_n) \leq V_t \leq q_V(1 - \delta)} \right)
\end{equation*}
goes to $0$ almost surely. Using Hoeffding's inequality, for all $\beta \in (0,1)$,
\begin{equation*}
\Pbb\left(\left|\frac{|\{t \in \{1, \dots, n\} \text{ s.t. } V_t \geq q_V(\beta)\}|}{n} - \Pbb(V_1 \geq q_V(\beta))\right| \geq \sqrt{\frac{\log n}{n}} \right) \leq 2 n^{-2}.
\end{equation*}

In particular, taking $\beta = 1-\delta$, Borel-Cantelli's lemma shows that almost surely, for large enough $n$,
\begin{equation*}
\hat{q}_V\left(1 - \delta - \sqrt{\frac{\log n}{n}}\right)
	\leq q_V(1 - \delta)
	\leq \hat{q}_V\left(1 - \delta + \sqrt{\frac{\log n}{n}}\right),
\end{equation*}
so that there are at most $\sqrt{n \log n}$ terms between $q_V(1 - \delta)$ and $\hat{q}_V(1 - \delta)$. Hence there are at most $\sqrt{n \log n} + (\delta_n - \delta) n$ terms between $q_V(1 - \delta)$ and $\hat{q}_V(1 - \delta_n)$. As $(\delta_n)_n$ is non-increasing, this yields
\begin{align*}
\hat{q}_V(1 - \delta_n)
	\leq \hat{q}_V(1 - \delta)
	\leq q_V\left(1 - \delta + \sqrt{\frac{\log n}{n}}\right),
\end{align*}
which ensures that these terms are bounded above by $q_V(1 - \delta + \sqrt{\frac{\log n}{n}})$. Thus, almost surely, for $n$ large enough,
\begin{align*}
\frac{1}{n} \sum_{t=1}^n V_t \left(\one_{q_V(1 - \delta) \leq V_t \leq \hat{q}_V(1 - \delta_n)} + \one_{\hat{q}_V(1 - \delta_n) \leq V_t \leq q_V(1 - \delta)} \right)
	\leq \left[ \sqrt{\frac{\log n}{n}} + (\delta_n - \delta) \right] q_V(1 - \delta/2),
\end{align*}
which indeed converges to $0$.

\bibliographystyle{plainnat}
\bibliography{these}

\end{document}